\documentclass[12pt, a4paper]{amsart}

\usepackage[hmargin=30mm, vmargin=25mm, includefoot, twoside]{geometry}
\usepackage[bookmarksopen=true]{hyperref}

\usepackage{amsfonts,amssymb}
\usepackage{verbatim}
\usepackage{latexsym}
\usepackage{mathrsfs}
\usepackage{stmaryrd}
\usepackage{xspace}
\usepackage{enumerate}
\usepackage{paralist}
\usepackage{graphicx}
\usepackage[all]{xy}
\usepackage{extarrows}

\usepackage[usenames,dvipsnames]{xcolor}

\usepackage{txfonts, pxfonts}
\usepackage{mathptmx}

\usepackage{amsthm}
\usepackage{amsmath}

\usepackage[T1]{fontenc}	
\usepackage[utf8]{inputenc}			% can use \k{}
\usepackage{mathtools} 					%typsetting tools
\usepackage{microtype} 					%better looking pdf output
\usepackage{bbm}					%more blackboard symbols
\usepackage{bm}						%more boldface symbols
\usepackage[shortcuts]{extdash} 			% hyphen \=/
\usepackage{subcaption}					%necessary for subfigures
\usepackage{mparhack}					%fix a bug relating \marginpar 
\usepackage{accents}

\usepackage{todonotes}
\newcommand{\CCC}{\mathbb{C}}

		\newcommand{\NN}{\mathbb{N}}	
			
		\newcommand{\RR}{\mathbb{R}}

\newcommand{\CC}{\mathcal{C}}			
\newcommand{\CE}{\mathcal{E}}			
\newcommand{\CG}{\mathcal{G}}		\newcommand{\CH}{\mathcal{H}}

\newcommand{\CO}{\mathcal{O}}		\newcommand{\CP}{\mathcal{P}}	
\newcommand{\CQ}{\mathcal{Q}}

\newcommand{\CW}{\mathcal{W}}

\newcommand{\ubar}[1]{\underaccent{\bar}{#1}} 		%\bar under letter

\DeclarePairedDelimiter\abs{\lvert}{\rvert}		%absolute value
\DeclarePairedDelimiter\norm{\lVert}{\rVert}		%norm
\DeclarePairedDelimiter\angles{\langle}{\rangle}	% used in \scal
\DeclarePairedDelimiter\paren{(}{)}			%   These are to fix the spacing between 
\DeclarePairedDelimiter\braces{\{}{\}}			%   Also, use with * to get variable size

	\newcommand{\bigparen}[1]{\paren[\big]{#1}}
	\newcommand{\Bigparen}[1]{\paren[\Big]{#1}}

	\newcommand{\bigbraces}[1]{\braces[\big]{#1}}
	\newcommand{\Bigbraces}[1]{\braces[\Big]{#1}}
\newcommand{\bigmid}{\mathrel{\big|}}			%bigger \mid
\theoremstyle{plain}
\newtheorem{thm}{Theorem}[section]				
\newtheorem{prop}[thm]{Proposition}		
\newtheorem{lem}[thm]{Lemma}						
\newtheorem{cor}[thm]{Corollary}

\newtheorem*{thm*}{Theorem}			\newtheorem*{theorem*}{Theorem}		
\newtheorem*{prop*}{Proposition}		\newtheorem*{proposition*}{Proposition}
\newtheorem*{lem*}{Lemma}			\newtheorem*{lemma*}{Lemma}			
\newtheorem*{cor*}{Corollary}			\newtheorem*{corollary*}{Corollary}
\newtheorem*{qu*}{Question}			\newtheorem*{question*}{Question}
\newtheorem*{conj*}{Conjecture}			\newtheorem*{conjecture*}{Question}
\newtheorem*{fact*}{Fact}
\newtheorem*{claim*}{Claim}

\newtheorem{alphthm}{Theorem}			%letter numbering

\newtheorem{alphprop}[alphthm]{Proposition}

\newtheorem{alphcor}[alphthm]{Corollary}

\theoremstyle{definition}
\newtheorem{de}[thm]{Definition}

\newtheorem*{de*}{Definition}			\newtheorem{definition*}{Definition}	
\newtheorem*{notation*}{Notation}	
\newtheorem*{conv*}{Convention}			\newtheorem*{convention*}{Convention}

\theoremstyle{remark}
\newtheorem{rmk}[thm]{Remark}

\numberwithin{equation}{section}

 \theoremstyle{definition}
  \newtheorem{defn}[thm]{Definition}
  
 \theoremstyle{remark}
 \newtheorem{rem}[thm]{Remark}
  \newtheorem{ex}[thm]{Example}

\def\B{\mathfrak B}
\def\K{\mathfrak K}

\def\H{\mathcal H}

\def\supp{\mathrm{supp}}
\def\diam{\mathrm{diam}}
\def\Id{\mathrm{Id}}

\def\N{\mathbb N}
\def\C{\mathbb C}

\newcommand{\acbdry}{\partial^{\Gamma}}
\newcommand*{\mkbdry}[1]{{\abs{\partial_{\Pi}(#1)}_m}}
\newcommand*{\locmkbdry}[1]{{\abs{\partial_{\Pi_{Y,S}}(#1)}_{\snu}}}

\newcommand{\fnc}{\ubar c}
\newcommand{\fnb}{\ubar b}
\newcommand{\fnk}{\ubar k}

\newcommand{\fnh}{\ubar h}

\newcommand{\fkG}{\mathfrak{G}}
\newcommand{\fkP}{\mathfrak{P}}

\newcommand{\snu}{{\tilde\nu_{Y,S}}}

\newcommand*{\wc}[1]{{\CW\CC\paren{#1}}}

\newcommand{\symdiff}{{\, \resizebox{0.8 em}{!}{$\triangle$}\, }}
\newcommand{\mhyphen}{\operatorname{-}}
\newcommand{\variable}{\,\mhyphen\,}

\def\O{\mathcal O}

\def\d{\mathrm{d}}
 
\def\act{\curvearrowright}

\begin{document}

\title[A Markovian and Roe-algebraic approach to asymptotic expansion]{A Markovian and Roe-algebraic approach to asymptotic expansion in measure}

\author{Kang Li}
\address{Department of Mathematics, Friedrich-Alexander-Universität Erlangen-Nürnberg, Cauerstraße 11, 91058 Erlangen, Germany}
\email{kang.li@fau.de}

\author{Federico Vigolo}
\address{Faculty of Mathematics and Informatics, Westfälische Wilhelms-Universität Münster, Einsteinstrasse 62, 48149, Münster, Germany}
\email{fvigolo@uni-muenster.de}

\author{Jiawen Zhang}
\address{School of Mathematical Sciences, Fudan University, 220 Handan Road, Shanghai, 200433, China}
\email{jiawenzhang@fudan.edu.cn}

\thanks{KL has received funding from the European Research Council (ERC) under the European Union's Horizon 2020 research and innovation programme (grant agreement no. 677120-INDEX)}
\thanks{FV was supported by the ISF Moked 713510 grant number 2919/19 for his stay at the Weizmann Institute of Science, Israel.}
\thanks{JZ was supported by the Sino-British Trust Fellowship by Royal Society and NSFC11871342.}

\begin{abstract}
In this paper, we conduct further studies on geometric and analytic properties of asymptotic expansion in measure. More precisely, we develop a machinery of Markov expansion and obtain an associated structure theorem for asymptotically expanding actions. Based on this, we establish an analytic characterisation for asymptotic expansion in terms of the Dru\c{t}u-Nowak projection and the Roe algebra of the associated warped cones. As an application, we provide new counterexamples to the coarse Baum-Connes conjecture.
\end{abstract}

\date{\today} 

\maketitle


\section{Introduction}

This paper is the second part of a broader study of the notion of asymptotic expansion in measure for measurable actions of countable groups on probability spaces.
We introduced this notion in \cite{dynamics1}, as a dynamical analogue of a previously defined notion of asymptotic expansion for metric spaces \cite{intro}.

Asymptotic expansion in measure is a weakening of expansion in measure as defined in \cite{Vig19} and it turns out that---for measure\=/class\=/preserving actions---it is also equivalent to the classical notion of strong ergodicity introduced by Schmidt \cite{Sch81} and Connes-Weiss \cite{connes_property_1980} (see \cite{dynamics1} for more details). 

More precisely, a measurable action $\rho\colon\Gamma\curvearrowright(X,\nu)$ of a countable group on a probability space is \emph{asymptotically expanding in measure} if for each $\alpha \in (0,\frac{1}{2}]$ there exist $c_\alpha>0$ and a finite symmetric set $S_\alpha \subseteq \Gamma$ such that for every measurable subset $A\subseteq X$ with $\alpha\leq \nu(A) \leq \frac{1}{2}$ we have
\begin{equation}\label{eq:intro:as.expansion in measure}
  \nu\Bigparen{\bigcup_{s\in S_\alpha} s\cdot A} > (1+c_\alpha)\nu(A).
\end{equation}
The action $\rho$ is called \emph{expanding in measure} if we let $c_\alpha \equiv c$ and $S_\alpha\equiv S$ for some $c>0$ and a finite subset $S$ in $\Gamma$.

In \cite{dynamics1} we studied the general structure theory of asymptotically expanding actions. Most notably, we showed that an action is asymptotically expanding in measure \emph{if and only if} it admits an exhaustion by domains of expansion (see Section~\ref{ssec:asymptotic.expansion in measure} for a more detailed account). This fact allowed us to reprove a few recent--and--old results for strongly ergodic actions and it is also a key technical tool for the present paper. In addition, we also made explicit the connection between the notion of asymptotic expansion for measurable actions and that of asymptotic expansion for metric spaces. This allowed us to provide a rich source of concrete examples of asymptotic expander graphs (see \cite{dynamics1} for details).

In this paper, we will further study the notion of asymptotic expansion in measure in the context of measure\=/class\=/preserving actions (in particular, all the results here described hold for strongly ergodic actions). Adapting the techniques developed in \cite{structure} to the dynamical setting, we are able to prove some rather striking analytic and geometric properties of asymptotic expansion in measure. More precisely, we obtain an analytic characterisation of asymptotic expansion in measure in terms of quasi\=/locality of averaging projections and Roe algebras of the associated warped cones. As a consequence, we will provide a new source of counterexamples to the coarse Baum--Connes conjecture, which is a central problem in higher index theory (see, \emph{e.g.}, \cite{Roe96, WY20}).

To obtain the above results, we develop a spectral characterisation of (asymptotic) expansion in measure which we find of independent interest. This is obtained by associating measure\=/class\=/preserving actions with some reversible Markov kernels and by studying the resulting Laplacians and averaging operators. The spectral characterisation is obtained by extending some classical results for Markov processes on finite state\=/spaces to general Markov kernels. We find that this theory provides a solid framework to study spectral properties of actions that are not necessarily measure\=/preserving.

\subsection{Spectral gaps and Markov expansion}
A probability measure-preserving action $\rho\colon\Gamma\curvearrowright(X,\nu)$ always induces a unitary representation $\pi\colon\Gamma\curvearrowright L^2(X,\nu)$. If $\Gamma$ is generated by a finite symmetric subset $S$, the action $\rho$ has a \emph{spectral gap} if there exists some positive constant $\kappa>0$ such that every $f\in L^2(X,\nu)$ with $\int_Xf\d\nu=0$ satisfies
\begin{equation}\label{eq:intro:spectral gap}
 \norm{f}_2\leq \kappa\sum_{s\in S}\norm{\pi(s)f-f}_2.
\end{equation}
This can be seen as an extremely strong version of ergodicity, and it is not very hard to show that $\rho$ has a spectral gap \emph{if and only if} it is expanding in measure (this was shown more or less independently in \cite{BIG17,grabowski_measurable_2016,houdayer2017strongly,Vig18}, and was already implicit in earlier works of K. Schmidt and Connes--Feldmann--Weiss).

With the action $\rho$ is associated a \emph{Markov operator} $\fkP \in \B(L^2(X,\nu))$ defined by $\fkP \coloneqq \frac{1}{\abs S}\sum_{s\in S}\pi(s)$ and a \emph{Laplacian} $\Delta\coloneqq 1-\fkP \in \B(L^2(X,\nu))$. These operators are self\=/adjoint, and $\rho$ has a spectral gap \emph{if and only if} $0$ is a \emph{simple} (\emph{i.e.}, with multiplicity one) isolated point in the spectrum of $\Delta$ (equivalently, $1$ is a simple isolated point in the spectrum of $\fkP$).
This characterisation in terms of self-adjoint operators is crucial to provide explicit examples of actions with spectral gap, as it opens a door to algebraic and representation theoretical tools. 
In fact, this point of view leads to very deep connections between dynamical systems, analysis and number theory. These connections make the study of the spectral gap property for measure-preserving actions into a very active and important field of research  (\cite{benoist_spectral_2014,bourgain_spectral_2007,BIG17,GJS99,lubotzky1986hecke,margulis1980some}).

As an intermediate step toward an analytic study of asymptotic expansion in measure, we set a framework to extend the above connections to the setting of measure\=/class\=/preserving actions. It follows from the work of Houdayer--Marrakchi--Verraedt \cite[Theorem 3.2]{houdayer2017strongly} that expansion in measure is equivalent to \eqref{eq:intro:spectral gap}, whenever $\rho(s)$ has bounded Radon--Nikodym derivative for every $s\in S$. In turn, \eqref{eq:intro:spectral gap} holds \emph{if and only if} $0$ is a simple isolated point in the spectrum of the self\=/adjoint operator $T\coloneqq \sum_{s\in S}\abs{1-\pi(s)}$. However, the spectrum of the operator $T$ remains difficult to control. It is therefore desirable to produce some spectral condition which can more adequately describe the notion of expansion.

In this paper, we provide a rather satisfactory answer to the above need by using Markov kernels and Markov expansion.
Our approach is based on a shift in paradigm, and can be justified by some analogies between finite graphs and dynamical systems.
A systematic study of these analogies by means of the approximation procedure can be found in \cite{Vig18} and further developed in \cite{dynamics1}.
According to this procedure, expansion in measure corresponds to \emph{vertex\=/expansion} for finite graphs \cite{Vig18} and, if the action is measure-preserving, the spectral gap condition \eqref{eq:intro:spectral gap} can be seen as an analogue of \emph{spectral expansion} for graphs (see also \cite{sawicki_super-expanders_2017}). 
It is a classical result that spectral\=/expansion is equivalent to \emph{edge\=/expansion} (\cite{alon1986eigenvalues,alon1985lambda1,dodziuk1984difference}), and it is easy to verify that the latter is equivalent\footnote{Assuming that the graphs have uniformly bounded degree.} to vertex\=/expansion. This can be seen as the graph\=/theoretic analogue of the equivalence between \eqref{eq:intro:spectral gap} and expansion in measure for measure-preserving actions.

To be more precise, \emph{spectral expansion} for a regular finite graph $\CG$ is defined in terms of the spectral gap of the discrete Laplacian $\Delta\in\B(L^2(\CG,\nu))$. Here $\nu$ is the counting measure on the set of vertices of $\CG$ and $\Delta$ is defined as $1-\fkP$, where $\fkP$ is the averaging operator (\emph{a.k.a.} Markov operator) defined by $\fkP f(v)\coloneqq \sum_{v\sim w}f(w)/ \text{degree(v)}$ for $f\in L^2(\CG,\nu)$. Importantly, if the graph $\CG$ is \emph{not} regular then the discrete Laplacian is \emph{no longer} self\=/adjoint in $\B(L^2(\CG,\nu))$. Instead, it is self\=/adjoint in $\B(L^2(\CG,\tilde\nu))$, where $\tilde\nu$ is a different measure which takes into account the degree of each vertex. Spectral expansion is then defined in terms of the spectrum of $\Delta$ seen as an operator on $L^2(\CG,\tilde \nu)$. A more sophisticated way of rephrasing this is that the (lazy) simple random walk on a finite connected graph $\CG$ has a unique \emph{stationary probability measure} $\tilde \nu$. The probability distribution of the $n$\=/th step of such a random walk converges exponentially fast to $\tilde\nu$ (in the $L^2$\=/norm), and the spectral expansion measures the rate of exponential convergence.

The above discussion can be used as heuristics in the dynamical setting. We remark that graphs corresponding to a measure-preserving action are ``regular on a large scale''.\footnote{
To give a somewhat precise meaning to the notion of ``regular on a large scale'' it is necessary to use the terminology of \cite{dynamics1,Vig18}: given any measurable subset $A\subseteq X$ and a sufficiently fine approximation $[A]_\CP$, the ratio $\abs{\partial[A]_\CP}/\abs{[A]_\CP}$ will be roughly equal to $\nu(S\cdot A)/\nu(A)$. If $\rho$ is measure-preserving and $A$ is disjoint from $s\cdot A$ for every $s\in S$, then the latter ratio is equal to $\abs{S}$. That is, the approximating graphs are ``$\abs{S}$\=/regular on a large scale''.
} 
It is therefore natural to expect a correspondence between spectral expansion and expansion in measure. On the other hand, actions that are not measure\=/preserving correspond to irregular graphs (the ``large scale degrees'' are governed by the Radon--Nikodym derivatives). This suggests us to search for a spectral characterisation of expansion in measure in terms of some operator in $\B(L^2(X,\tilde \nu))$---where $\tilde \nu$ is some stationary measure depending on the Radon--Nikodym derivatives. This is precisely the approach that we take in this paper.

Let $\Gamma$ be a finitely generated group and $S\subseteq \Gamma$ a finite symmetric generating set containing the identity element, and let $\rho\colon\Gamma\curvearrowright (X,\nu)$ be a measure\=/class\=/preserving action with the Radon--Nikodym derivatives $r(\gamma,x)\coloneqq\frac{d \gamma^{-1}_*\nu}{d\nu}(x)$. It turns out that the measure $\tilde \nu$ defined by 
\[
 \d\tilde\nu(x)\coloneqq\sum_{s\in S}r(s,x)^{\frac 12}\d\nu(x) 
\]
is a stationary measure for the reversible Markov kernel 
\[
 \Pi(x,\mhyphen)\coloneqq\frac{1}{\sum_{s\in S}r(s,x)^{\frac 12}}\sum_{s\in S}r(s,x)^{\frac 12}\delta_{s\cdot x}.
\]

Naturally associated to $\Pi$, there are a Markov operator $\fkP$ and a Laplacian $\Delta=1-\fkP$. Both of these are self\=/adjoint operators in $\B(L^2(X,\tilde \nu))$---we defer to Section~\ref{sec:Markov.expansion} for preliminaries and definitions regarding Markov kernels.
Every measurable subset $A\subseteq X$ has a natural notion of ``measure of the boundary'' $\abs{\partial_{\Pi}(A)}\in\RR_{\geq 0}$ (Definition~\ref{defn:markov boundary} or \cite{Kai92}), and we say that $\rho$ is \emph{Markov expanding} if there is a $c>0$ such that
\[
 \abs{\partial_{\Pi}(A)}> c\tilde\nu(A)
\]
for every $A\subseteq X$ with $0< \tilde\nu(A)\leq\frac{1}{2}\tilde\nu(X)$. This should be thought of as a dynamical analogue of edge\=/expansion for graphs.
Importantly, the equivalence between edge\=/expansion and spectral expansion can be extended from the context of random walks on graphs to that of general reversible Markov kernels:

% Our first technical result is a generalisation of the proof of the equivalence between edge\=/expansion and spectral expansion from the context of random walks on graphs to that of general reversible Markov kernels. 
% This result is probably known to experts, but we were not able to find a proof in the literature.

\begin{alphthm}[{\cite[Theorem 2.1]{lawler1988bounds}, see also the appendix to this paper}]\label{thm:intro: expansion Markov kernel}
 Let $\Pi$ be a reversible Markov kernel on $X$ with finite reversing measure $m$. Let $\lambda_2$ be the infimum of the spectrum of the restriction of $\Delta$ to the space of functions with zero average, and let $\kappa\coloneqq\inf \abs{\partial_{\Pi}(A)}/m(A)$ for $A\subseteq X$ with $0<m(A)\leq\frac 12 m(X)$. Then
 \[
  \frac{\kappa^2}{2} \leq 1-\lambda_2\leq 2\kappa.
 \]
\end{alphthm}

As a consequence, we obtain a characterisation for Markov expansion in terms of the spectrum of $\Delta\in\B(L^2(X,\tilde\nu))$. 
Furthermore, it is relatively easy to show that, when the Radon--Nikodym derivatives are bounded, Markov expansion is equivalent to the original notion of expansion in measure (this is analogous to the equivalence between edge\=/expansion and vertex\=/expansion for graphs of uniformly bounded degrees). This leads us to the following:

\begin{alphprop}[Corollary \ref{cor:local.spec.gap.iff.domain.of.expansion}, Remark \ref{rmk:bdd ratio}]\label{prop:intro spectral charact. expansion}
Let $\Theta\geq 1$ be a constant. A measure\=/class\=/preserving action $\rho\colon\Gamma\curvearrowright (X,\nu)$ with $1/\Theta\leq r(s,x)\leq\Theta$ for every $s\in S$ and $x\in X$ is expanding in measure \emph{if and only if} $0$ is a simple isolated point in the spectrum of $\Delta\in\B(L^2(X,\tilde\nu))$.
\end{alphprop}

\begin{rmk}
 It is not hard to show that Proposition~\ref{prop:intro spectral charact. expansion} and \cite[Theorem 3.2]{houdayer2017strongly} are in fact equivalent. However, we find that our approach has various advantages:
 \begin{enumerate}
  \item We find that the Laplacian operator $\Delta$ is more natural than $T$. It should be easier to handle (\emph{e.g.}, to control spectral gap), and it allows us to borrow several calculations and results from the classical setting of random walks on graphs.
  \item The spectral gap condition can be rephrased by saying that the restriction of the Markov operator $\fkP$ to the space functions with zero\=/average has operator norm strictly less than $1$. It can be useful to know that $\fkP^n$ converges in the operator norm to the projection onto constant functions (see also Section~\ref{ssec:domains.of.exp.and.projections.as.limits}). 
  \item It allows for a finer control of the expansion constants.
  \item The spectral characterisation of Markov expansion holds true also for actions with unbounded Radon--Nikodym derivatives (this should be of independent interest).
 \end{enumerate}
\end{rmk}

We restricted the previous discussion to the case of actions of finitely generated groups for the sake of simplicity. However, the machinery of Markov kernels is very flexible, and all the results mentioned above will actually be proved for actions of arbitrary discrete countable groups. Furthermore, we will also study restrictions of actions to subsets of $X$ which are not necessarily invariant.\footnote{It would be also possible to extend this theory to include general countable measurable equivalence relations.} 
As a sample application, we note that Proposition~\ref{prop:intro spectral charact. expansion} implies the following (see Section~\ref{sec:prelims} for the relevant definitions and Corollary~\ref{cor:local.spec.gap.iff.domain.of.expansion}):

\begin{alphcor}[{\cite{grabowski_measurable_2016}}]\label{cor:intro:Matkov}
 A measure\=/preserving action $\Gamma\curvearrowright(X,\nu)$ has local spectral gap with respect to $Y\subseteq X$ \emph{if and only if} $Y$ is a domain of expansion. 
\end{alphcor}

Being able to work with subsets of $X$ is a necessary requirement to use the structure theorems established in \cite{dynamics1}, which characterise asymptotic expansion in terms of exhaustions (see Section~\ref{ssec:asymptotic.expansion in measure}). Combining those results with the Markov machinery developed above, we are able to prove an additional structure result which will play a key role in the rest of the paper:

\begin{alphthm}[Theorem~\ref{thm:structure theorem Markov}]\label{thm:intro:structure.Matkov}
 A measure\=/class\=/preserving action $\Gamma\curvearrowright (X,\nu)$ on a probability space is asymptotically expanding in measure \emph{if and only if} every subset $Y\subseteq X$ admits an exhaustion by domains of Markov expansion.
\end{alphthm}

\begin{rmk}
 The above theorem remains true when replacing ``probability'' by ``$\sigma$\=/finite'' and ``asymptotically expanding in measure'' by ``strongly ergodic''. 
\end{rmk}

\subsection{Warped cones and finite propagation approximations}

Our next aim is to study asymptotically expanding actions via analytic properties of certain projection operators. This is done by using the warped cone construction as a bridge between the metric and dynamical setting, and then utilizing Markov expansion. The end result is a dynamical analogue of the theory developed in \cite{structure} to characterise asymptotic expanders using averaging projections.

The notion of warped cone was firstly introduced by Roe in \cite{roe_warped_2005} to explore more examples with/without Yu's property A and coarse embeddings into Hilbert spaces. The geometry of warped cones was subsequently studied by a number of people, \emph{e.g.}, \cite{DN17,fisher_rigidity_2019, nowak_warped_2017, sawicki_super-expanders_2017, sawicki_warped_2017, SawickiThesis, sawicki_straightening_2017, Vig18,vigolo2019discrete, wang2017warped}. Roughly speaking, given a continuous action $\Gamma\curvearrowright (X,d)$ on a compact metric space with diameter at most $2$, the associated \emph{unified warped cone} is the metric space $(\CO_\Gamma X,d_\Gamma)$, where $\CO_\Gamma X= X\times[1,\infty)$ as a set and $d_\Gamma$ is a metric on $\CO_\Gamma X$ defined in terms of the group action (see Section \ref{ssec:pre on warped cones} for details).

Given a probability measure $\nu$ on $(X,d)$, we consider the \emph{averaging projection $P_X$} on $L^2(X,\nu)$, which is the rank-one orthogonal projection onto the space of constant functions on $X$. Denoting by $\lambda$ the Lebesgue measure on $[1,\infty)$, the \emph{Dru\c{t}u--Nowak projection} is defined as $\fkG= P_X\otimes \Id_{L^2([1,\infty))}\in \B(L^2(\CO_\Gamma X, \nu\times\lambda))$, which is the orthogonal projection onto $\CCC\otimes L^2([1,\infty),\lambda)$. 

The Dru\c{t}u--Nowak projection $\fkG$ was first introduced in \cite[Section~6.c.]{DN17} in their study on the coarse Baum--Connes conjecture (more details will be provided later). They showed that if an action is measure-preserving and has a spectral gap, then the projection $\fkG$ is a norm limit of \emph{finite propagation} operators in $\B(L^2(\CO_\Gamma X, \nu\times\lambda))$. Recall that an operator $T\in \B(L^2(\CO_\Gamma X, \nu\times\lambda))$ has finite propagation if there exists $R>0$ such that for any $f,g\in C_0(\CO_\Gamma X)$ with $d_\Gamma (\supp (f), \supp (g))>R$ we have $fTg=0$, where $f$ and $g$ are regarded as diagonal operators on $L^2(\CO_\Gamma X, \nu\times\lambda)$ via the multiplication representation.

In this paper, we study the converse of Dru\c{t}u--Nowak's result and prove the following analytic characterisation for asymptotically expanding actions:

\begin{alphthm}[Theorem~\ref{thm: characterise quasi-locality} and Theorem ~\ref{thm: characterise finite ppg}]\label{thm:intro:analytic char}
Let $(X,d)$ be a metric space with diameter at most $2$ equipped with a Radon probability measure $\nu$, and $\rho\colon \Gamma\curvearrowright X$ be a continuous measure\=/class\=/preserving action. The following are equivalent:
\begin{enumerate}
  \item $\rho$ is asymptotically expanding;
  \item the Druţu--Nowak projection $\fkG$ is quasi-local;
  \item the Druţu--Nowak projection $\fkG$ is a norm limit of operators with finite propagation.
\end{enumerate}
\end{alphthm}

\begin{rmk}
 The notion of quasi-locality was introduced by Roe in \cite{Roe88}. It is weaker\footnote{It is conjectured that quasi\=/locality should be strictly weaker than admitting such approximations.} than the property of admitting an approximation by finite propagation operators, and it is relatively easy to verify. For more details on quasi-locality, we refer readers to \cite{Eng15, intro, LWZ19, ST19, SZ18}.
\end{rmk}

Theorem \ref{thm:intro:analytic char} is a dynamical analogue of \cite[Theorem 6.1]{structure}, and there are two main ingredients in its proof. Firstly, we introduce a dynamical notion of finite propagation approximation and quasi-locality (see Section \ref{ssec:Warped cones and Druţu--Nowak projections} and \ref{ssec:finite propagation wcone}) as an intermediate bridge to connect asymptotic expansion and analytic properties of the Druţu--Nowak projection. Secondly, we apply the tool of Markov expansion to approximate dynamical quasi-local operators with finite dynamical propagation ones. Due to some correspondence results (Proposition \ref{prop: quasi-locality equiv} and \ref{prop: finite ppg equiv}), we can then pass from the dynamical notions to their analytic analogues for unified warped cones and obtain Theorem \ref{thm:intro:analytic char}.

As a byproduct, we construct numerous projections which can be approximated by operators with finite propagation (see Corollary \ref{cor:projection finite ppg}). These projections will be important in the next section, where we deal with the coarse Baum--Connes conjecture.

\subsection{Roe algebras and the coarse Baum--Connes conjecture}

Roe algebras are $C^*$-algebras associated with metric spaces. These $C^*$-algebras encode coarse geometric information of the metric spaces and play key roles in higher index theory (see, \emph{e.g.}, \cite{Roe88, Roe96, WY20} for more details). We conclude this paper by studying Roe algebras of warped cones associated to asymptotically expanding actions and provide an application to the so-called coarse Baum--Connes conjecture.

Given a continuous action $\Gamma\curvearrowright (X,d)$ on a compact metric space with diameter at most $2$ and a non-atomic probability measure $\nu$ on $(X,d)$ with full support, we consider the multiplication representation $C_0(\CO_\Gamma X) \to \B(L^2(\CO_\Gamma X, \nu\times\lambda))$. The \emph{Roe algebra} of the unified warped cone, denoted by $C^*(\CO_\Gamma X)$, is the norm closure of all finite propagation locally compact operators in $\B(L^2(\CO_\Gamma X, \nu\times\lambda))$ (see Section \ref{ssec:roe algebras} for more details). 

Although the Druţu--Nowak projection $\fkG$ can be approximated by finite propagation operators, it is \emph{not} locally compact because its restriction on $L^2([1,\infty),\lambda)$ is the identity operator. In order to obtain non-trivial projections in the Roe algebra, Sawicki \cite{sawicki_warped_2017} suggested to consider the \emph{integral} warped cone and the associated \emph{integral} Druţu--Nowak projection (see Section \ref{ssec:roe algebras}). Since the integral warped cone is coarsely equivalent to the original warped cone, they have $*$-isomorphic Roe algebras. Based on \cite{DN17}, Sawicki \cite[Proposition 1.3]{sawicki_warped_2017} showed that for a measure-preserving action with spectral gap, the integral Druţu--Nowak projection belongs to the associated Roe algebra. 

Theorem \ref{thm:intro:analytic char} allows us to both extend and provide a converse to Sawicki's result:

\begin{alphthm}[Theorem~\ref{thm:char for asymp. expansion via Roe}, Corollary \ref{cor:ghost.projection.in.Roe}]\label{thm:intro:analytic char Roe}
Let $(X,d)$ be a compact metric space with diameter at most $2$, $\nu$ a non-atomic Radon probability measure on $X$ of full support, and $\rho\colon \Gamma\curvearrowright (X,d,\nu)$ a continuous measure-class-preserving action. Then $\rho$ is asymptotically expanding \emph{if and only if} the integral Druţu--Nowak projection belongs to the Roe algebra $C^*(\mathcal{O}_\Gamma X)$. Moreover, the integral Druţu--Nowak projection is non-compact and ghost.
\end{alphthm}

The study of projections in Roe algebras is motivated by the computation of their K-theories. The coarse Baum--Connes conjecture asserts that K\=/theories of Roe algebras can be computed in terms of homology information of underlying metric spaces. When true, this establishes a connection between geometry, topology and analysis. One ground-breaking result on the subject is due to Yu \cite{Yu00}, as he showed that the coarse Baum--Connes conjecture holds for all metric spaces with bounded geometry that are coarsely embeddable into Hilbert spaces. On the other hand, counterexamples to the conjecture were subsequently discovered by Higson \cite{higson1999counterexamples} (see also \cite{HLS02}) using expander graphs. In a recent joint work with Khukhro, we found more counterexamples using asymptotic expanders \cite{structure}. 

Understanding which spaces satisfy the coarse Baum--Connes conjecture is still one of the major questions in higher index theory, as it has significant applications to other areas of mathematics, such as topology and geometry (see \cite{higson1995coarse, schick2014:ICM,  skandalis2002coarse, yu1997zero} for more details).

It is an open question whether warped cones arising from actions with spectral gap are counterexamples to the coarse Baum--Connes conjecture. This question was the motivation behind the introduction of the Druţu--Nowak projection in \cite{DN17}. Recently, Sawicki \cite[Theorem 3.5]{sawicki_warped_2017} proved that \emph{sparse warped cones} (see Section \ref{ssec:ceg to cBc}) do provide counterexamples to the coarse Baum--Connes conjecture. His proof follows a similar outline of Higson's original proof for expander graphs. Using our work on asymptotically expanding actions, we can generalise Sawicki's result as follows:

\begin{alphthm}[Corollary \ref{cor:ceg to CBC}]\label{thm:intro:ceg to cBC}
Let $(X,d)$ be a compact metric space of diameter at most $2$ equipped with a non-atomic probability measure $\nu$ of full support, and $\rho\colon\Gamma \act (X,d,\nu)$ be a free Lipschitz measure\=/class\=/preserving asymptotically expanding action. Under \textbf{either} of the following conditions:
 \begin{itemize}
 \item[(1)] if $\Gamma$ has property $A$ and $X$ is a manifold;
 \item[(2)] if the asymptotic dimension of $\Gamma$ is finite and $X$ is an ultrametric space;
 \end{itemize}  
 the coarse Baum--Connes conjecture for the sparse warped cone fails.
\end{alphthm}

\begin{rmk}
We can produce examples whose violation of the coarse Baum--Connes conjecture can be deduced from Theorem \ref{thm:intro:ceg to cBC}, but not from any previously known results (see Example \ref{eg:counterexample BCC}).
\end{rmk}

Under some extra conditions (ONL and bounded geometry), it follows by combining Theorem~\ref{thm:intro:ceg to cBC} with Yu's result \cite{Yu00} that warped cones arising from asymptotically expanding actions cannot coarsely embed into Hilbert spaces. Our last result shows that these extra conditions are in fact unnecessary (this partially generalises \cite[Theorem 3.1]{nowak_warped_2017}):

\begin{alphprop}[Proposition \ref{prop:non CE}]\label{prop:intro:non CE}
Let $(X,d)$ be a compact metric space of diameter at most $2$ equipped with a non-atomic probability measure $\nu$, and $\rho\colon \Gamma\curvearrowright (X,d,\nu)$ be a continuous measure-class-preserving and asymptotically expanding action. Then the warped cone $\mathcal{O}_\Gamma X$ does not admit a coarse embedding into any Hilbert space. 
\end{alphprop}

\subsection{Structure of the paper}
Section~\ref{sec:prelims} covers some preliminaries and further illustrates the connections between this paper and other works. The first half of Section~\ref{sec:Markov.expansion} can be read independently from the rest of the paper and is devoted to introducing reversible Markov kernels/expansion and the statement of Theorem~\ref{thm:intro: expansion Markov kernel}. A self\=/contained proof of Theorem~\ref{thm:intro: expansion Markov kernel} is given in the appendix. The second part of Section~\ref{sec:Markov.expansion} connects this theory to the study of measure\=/class\=/preserving actions. Here we prove Proposition~\ref{prop:intro spectral charact. expansion}, Corollary~\ref{cor:intro:Matkov} and Theorem~\ref{thm:intro:structure.Matkov}. These results will be important to both of the following sections. In Section~\ref{sec:warped cones} we recall the warped cone construction and study asymptotic expansion from the point of view of warped cones. Here we prove Theorem~\ref{thm:intro:analytic char}. Section~\ref{sec:baum_connes} is mostly devoted to the study of Roe algebras of warped cones. In the first part, we prove Theorem~\ref{thm:intro:analytic char Roe}, and in the second part we provide new counterexamples to the coarse Baum--Connes conjecture by proving Theorem~\ref{thm:intro:ceg to cBC}. Finally, we conclude this section by proving Proposition~\ref{prop:intro:non CE}.

\subsection*{Acknowledgments}
We wish to thank Amine Marrakchi for pointing out \cite{chifan2010ergodic,Mar18} to us and for manifesting interest in our work. The first author wishes to thank Damian Sawicki for helpful discussions on the coarse Baum--Connes conjecture. The second author wishes to thank Uri Bader for his helpful conversations. The third author wishes to thank Jan \v{S}pakula for several useful discussions on dynamical quasi-locality.

We also thank the anonymous referees for drawing our attention to \cite{lawler1988bounds} and pointing out a few mistakes in earlier versions of this paper.

%\subsection*{Declarations}
%Funding: KL has received funding from the European Research Council (ERC) under the European Union's Horizon 2020 research and innovation programme (grant agreement no. 677120-INDEX). FV was supported by the ISF Moked 713510 grant number 2919/19. JZ was supported by the Sino-British Trust Fellowship by Royal Society and NSFC11871342. Conflicts of interest: On behalf of all authors, the corresponding author states that there is no conflict of interest. Availability of data and material: Not applicable. Code availability: Not applicable. Authors' contributions: Not applicable.

\section{Preliminaries}
\label{sec:prelims}

\subsection{Standing conventions}
Throughout the paper, $\Gamma$ will always be a countable discrete group. The group $\Gamma$ will be made into a metric space by fixing a proper length function (see below).
The letter $S$ will always denote a finite subset in $\Gamma$. Such a set will often---but not always---be symmetric (\emph{i.e.}, $\gamma\in S$ implies that $\gamma^{-1} \in S$) and containing the identity element $1\in \Gamma$. We will \emph{not} generally assume that $S$ generates $\Gamma$.

All the measure spaces will be $\sigma$-finite and all the actions will be measurable. More precisely, we say that \emph{$\Gamma\curvearrowright (X,\nu)$ is an action} as shorthand for saying that $\Gamma$ is a countable discrete group acting measurably on a $\sigma$-finite measure space $(X,\nu)$.
When we equip a metric space $(X,d)$ with a measure $\nu$, we will always assume that $\nu$ is defined on the Borel $\sigma$\=/algebra.

\subsection{Actions on measure spaces}

Let $(X,\nu)$ be a   measure space. A measurable subset $A\subseteq X$ of positive finite measure is called a \emph{domain}. An \emph{exhaustion} of $(X,\nu)$ is a sequence of nested measurable subsets $Y_1\subseteq Y_2\subseteq\cdots $ such that $\bigcup_{n\in\NN}Y_n=X$ up to measure zero. We denote exhaustions by $Y_n\nearrow (X,\nu)$, or simply $Y_n\nearrow X$ if the measure is clear from the context. 

A \emph{proper length function} on $\Gamma$ is a function $\ell: \Gamma \to \{0\}\cup\NN$ which satisfies the following:
\begin{itemize}
  \item $\ell(\gamma)=0$ if and only if $\gamma=1$ (the identity element in $\Gamma$);
  \item $\ell(\gamma)=\ell(\gamma^{-1})$ for every $\gamma\in \Gamma$;
  \item $\ell(\gamma_1\gamma_2) \leq \ell(\gamma_1)+\ell(\gamma_2)$ for every $\gamma_1,\gamma_2\in \Gamma$;
  \item the number of $\gamma\in\Gamma$ with $\ell(\gamma)\leq k$ is finite for every $k\in\NN$.
\end{itemize}

It is easy to show that every countable discrete group $\Gamma$ admits a proper length function (see \emph{e.g.} \cite[Proposition~1.2.2]{NY12}). For example, if $\Gamma$ is a finitely generated group then we can simply take the word length with respect to an arbitrary finite symmetric generating set. Any proper length function $\ell$ induces a left-invariant metric $d_\ell$ on $\Gamma$ by $d_\ell(\gamma_1,\gamma_2) \coloneqq \ell(\gamma_1^{-1}\gamma_2)$. This makes $\Gamma$ into a \emph{proper} discrete metric space. Choosing a different length function $\ell'$ will yield a \emph{coarsely equivalent} metric on $\Gamma$ (we will not need this fact).

For each $k\in \NN$, we denote by $B_k$ the closed ball in $(\Gamma,\ell)$ with radius $k$ and centred at the identity:
\[
B_k \coloneqq \{\gamma\in \Gamma \mid \ell(\gamma) \leq k\}.
\]
It follows from the definition of length function that each $B_k$ is finite and symmetric, $1\in B_k$ and $B_k \cdot B_l \subseteq B_{k+l}$ for every $k,l\in \NN$.

We will be concerned with actions of $\Gamma$ on $(X,\nu)$. Given $A \subseteq X$ and $K \subseteq \Gamma$, let
\[
K \cdot A \coloneqq \bigcup_{\gamma\in K} \gamma \cdot A.
\] 
Since $1\in B_k$, we note that $A\subseteq B_k\cdot A$ for every $A\subseteq X$ and every $k\in \N$. 

Recall that an action $\Gamma\curvearrowright (X,\nu)$ is \emph{measure\=/class\=/preserving} if it sends measure\=/zero sets to measure\=/zero sets.
In this case, for every $\gamma\in\Gamma$ there is an associated Radon--Nikodym derivative $\d \gamma^{-1}_*\nu/\d\nu$ that is well\=/defined up to measure-zero sets.

\subsection{Expansion in measure}

 Let $\Gamma\curvearrowright X$ be an action and $S\subseteq\Gamma$ a finite symmetric set. For any measurable subset $A\subseteq X$ we denote $\acbdry_S A\coloneqq S\cdot A\smallsetminus A$, which should be regarded as the ``boundary of $A$ with respect to the action by $S$''.

\begin{de}[\cite{Vig19}]\label{defn:expanding in measure}
 An action $\rho\colon\Gamma\curvearrowright(X,\nu)$ on a probability measure space $(X,\nu)$ is called \emph{expanding (in measure)} if there exist a constant $c>0$ and a finite $S\subset \Gamma$ such that for any measurable subset $A\subseteq X$ with $0<\nu(A) \leq \frac{1}{2}$, we have $\nu(\acbdry_S A)> c\nu(A)$. In this case, we say that $\rho$ is \emph{$(c,S)$\=/expanding} or simply \emph{$S$\=/expanding}.
 
If an action is $(c,B_k)$\=/expanding for some $k\in\NN$, we may also say that it is $(c,k)$\=/expanding. Note that every expanding action is $(c,k)$\=/expanding for some $c>0$ and $k\in\NN$.
\end{de}

In an independent work, Grabowski--Máthé--Pikhurko defined a ``local'' version of expansion under the name of \emph{domain of expansion}: 

\begin{de}[\cite{grabowski_measurable_2016}\footnote{The authors of \cite{grabowski_measurable_2016} only consider measure-preserving actions, but their definition makes sense for general measurable actions as well.}]\label{defn:domain of expansion} 
Let $\rho\colon \Gamma\curvearrowright (X,\nu)$ be an action. A domain $Y \subseteq X$ is called a \emph{domain of expansion} for $\rho$ if there exist a constant $c>0$ and a finite $S\subseteq \Gamma$ such that for every measurable subset $A\subseteq Y$ with $0<\nu(A) \leq \frac{\nu(Y)}{2}$, we have
 \[
 \nu\bigparen{(S \cdot A)\cap Y} > \paren{1+c}\nu(A).
\]
In this case, we say that $Y$ is a \emph{domain of $(c,S)$\=/expansion} or simply of \emph{$S$\=/expansion}.
As before, if $S=B_k$ we may say that $Y\subseteq X$ is a domain of $(c,k)$\=/expansion.
\end{de}

We note that when $\nu$ is finite, $\rho\colon  \Gamma\curvearrowright (X,\nu)$ is expanding \emph{if and only if} $X$ is a domain of expansion for $\rho$. We end this subsection by recalling the following elementary fact, which will be used in the proof of Proposition~\ref{prop:Markov local version}:

\begin{lem}[{\cite[Lemma 3.14]{dynamics1}}]\label{lem:union of domains}
 Let $\rho\colon \Gamma \act (X,\nu)$ be an action and $Y \subseteq X$ a domain. Assume that $Y_1, Y_2\subseteq Y$ are domains of $S$\=/expansion. If $\nu(Y_1)>\frac{3}{4}\nu(Y)$ and $\nu(Y_2)>\frac{3}{4}\nu(Y)$ then the union $Y_1 \cup Y_2$ is a domain of $S$\=/expansion as well.
\end{lem}

\subsection{(Local) spectral gap}
We will work with complex $L^p$-spaces for $p\in [1,\infty)$. If we wish to stress that the $L^p$\=/norm of a function on $X$ is computed with respect to the measure $\nu$, we denote it by $\norm{f}_{\nu,p}$. Similarly, we will denote the inner product on the Hilbert space $L^2(X,\nu)$ by $\angles{f,g}_\nu$.
 
 Given a measurable subset $Y$ in a measure space $(X,\nu)$, we denote the restriction of $\nu$ to $Y$ by $\nu|_Y$. With a slight abuse of notation, we also use the symbol $\nu|_Y$ to denote the measure on $X$ which gives measure $0$ to $X\smallsetminus Y$ and coincides with $\nu$ on all measurable subsets of $Y$ (\emph{i.e.}, $\nu|_Y=\chi_Y\cdot \nu$ where $\chi_Y$ is the indicator function of $Y$). This will not cause confusion, as the meaning will be clear from the context.

A measure\=/preserving action $\rho\colon\Gamma\curvearrowright(X,\nu)$ on a probability measure space $(X,\nu)$ has a \emph{spectral gap} if there exist a constant $\kappa>0$ and a finite $S\subseteq \Gamma$ such that for every function $f\in L^2(X,\nu)$ with $\int_Xf\d\nu=0$ we have
\begin{equation}\label{eq:spectral gap}
 \norm{f}_2\leq \kappa\sum_{\gamma\in S}\norm{\gamma\cdot f-f}_2,
\end{equation}
where $\gamma\cdot f(x)\coloneqq f(\gamma^{-1} \cdot x)$.
It can be shown that the action $\rho$ is expanding in measure \emph{if and only if} it has a spectral gap (see, \emph{e.g.}, \cite[Section 7]{Vig19}).

In \cite{BIG17}, Boutonnet--Ioana--Golsefidy introduced the following localised version of spectral gap:

\begin{defn}[{\cite[Definition 1.2]{BIG17}}]\label{defn:local spectral gap}
Let $\rho\colon \Gamma \act (X,\nu)$ be a measure\=/preserving action and $Y \subseteq X$ be a domain.
The action $\rho$ has \emph{local spectral gap} with respect to $Y$ if there exist a constant $\kappa>0$ and a finite $S\subseteq \Gamma$ such that
\begin{equation}\label{eq:local spectral gap}
 \|f\|_{\nu|_Y,2} \leq \kappa \sum_{\gamma \in S} \norm{\gamma\cdot f -f}_{\nu|_Y,2}
\end{equation}
for every $f \in L^2(X,\nu)$ with $\int_Y f \mathrm{d}\nu=0$. 
\end{defn}

It is clear that when $\nu$ is a probability measure, $\rho$ has spectral gap \emph{if and only if} it has local spectral gap with respect to the whole $X$. 

It is shown in \cite[Lemma 5.2]{grabowski_measurable_2016} that a measure\=/preserving action $\rho\colon \Gamma\curvearrowright (X,\nu)$ has local spectral gap with respect to a domain $Y\subseteq X$ \emph{if and only if} $Y$ is a domain of expansion for $\rho$. This fact can also be deduced from \cite[Theorem 3.2]{houdayer2017strongly} (or by adapting the arguments of \cite[Section 7]{Vig19}). Later on, we will provide an alternative proof based on our study of Markov kernels (see Corollary \ref{cor:local.spec.gap.iff.domain.of.expansion}).

\begin{rmk}
 Equations \eqref{eq:spectral gap} and \eqref{eq:local spectral gap} make sense also if the action is not measure\=/preserving (although in this case it would be perhaps more appropriate to refer to them as Poincaré inequalities, rather than spectral gaps).
 It follows from \cite[Theorem 3.2]{houdayer2017strongly} that---as long as the Radon--Nikodym derivatives are bounded---the characterisation of (domains of) expansion in measure in terms of (local) spectral gaps also holds for actions that do not necessarily preserve measures.
\end{rmk}

\subsection{Asymptotic expansion in measure and structure theorems}
\label{ssec:asymptotic.expansion in measure}
The following weakening of expansion in measure was defined in \cite{dynamics1} in analogy with \cite{structure,intro}:

\begin{defn}[{\cite[Definition 3.1]{dynamics1}}]\label{defn:asymptotic expanding in measure}
Let $\rho\colon \Gamma \act (X,\nu)$ be an action on a space $(X,\nu)$ of finite measure. The action $\rho$ is called \emph{asymptotically expanding (in measure)} if there exist functions $\fnc\colon (0,\frac{1}{2}]\to \RR_{>0}$ and $\fnk\colon(0,\frac{1}{2}]\to \NN$ such that for every $\alpha\in(0,\frac{1}{2}]$ we have
\begin{equation}\label{eq:def.asymptotic.expansion}
 \nu\bigparen{B_{\fnk(\alpha)} \cdot A} > \paren{1+\fnc(\alpha)}\nu(A)
\end{equation}
for every measurable subset $A\subseteq X$ with $\alpha \nu(X)\leq \nu(A) \leq \frac{\nu(X)}{2}$.

For a finite $S\subseteq\Gamma$, we say that $\rho$ is \emph{$(\fnc,S)$\=/asymptotically expanding (in measure)} (or simply \emph{$S$\=/asymptotically expanding}) if for every $\alpha\in(0,\frac{1}{2}]$ and measurable subset $A\subseteq X$ with $\alpha \nu(X)\leq \nu(A) \leq \frac{\nu(X)}{2}$, we have $ \nu(S \cdot A) > \paren{1+\fnc(\alpha)}\nu(A)$.
\end{defn}

\begin{rem}\label{rem:finitely generated case}
 When the acting group $\Gamma$ is finitely generated by a finite symmetric set $S$, a measure\=/class\=/preserving action $\rho\colon \Gamma \act (X,\nu)$ on a probability space is asymptotically expanding \emph{if and only if} it is $S$\=/asymptotically expanding (see \cite[Lemma 3.16]{dynamics1}). We will not need this fact in this paper.
\end{rem}

This notion turns out to be naturally related to strong ergodicity. Recall that a measure\=/class\=/preserving action $\rho\colon \Gamma \act (X,\nu)$ on a probability space $(X,\nu)$ is called \emph{strongly ergodic} \cite{connes_property_1980,schmidt1980asymptotically} if any sequence of measurable subsets $\{C_n\}_{n \in \N}$ in $X$ with $\lim\limits_{n\to \infty} \nu(C_n \symdiff \gamma C_n)=0$ for every $\gamma \in \Gamma$, must satisfy
\[
\lim_{n\to \infty} \nu(C_n)(1-\nu(C_n))=0.
\]

Note that for two equivalent finite measures $\nu$, $\nu'$ on a space $X$ and any sequence of measurable subsets $(A_n)_{n\in\NN}$ in $X$, $\nu(A_n)\to 0$ \emph{if and only if} $\nu'(A_n)\to 0$. Hence, strong ergodicity only depends on the measure\=/class of the given measure. Therefore, the following is well\=/posed:

\begin{de}[\cite{ioana2017strong}]\label{strong ergodic on infinite space}
 A measure\=/class\=/preserving action $\Gamma\curvearrowright (X,\nu)$ on a (possibly infinite) measure space is \emph{strongly ergodic} if $\Gamma\curvearrowright (X,\nu')$ is strongly ergodic with respect to some (hence every) probability measure $\nu'$ equivalent to $\nu$.
\end{de}

\begin{prop}[{\cite[Proposition 3.5]{dynamics1}}]\label{prop:strongly ergodic iff asymptotic expanding}
Let $\rho\colon  \Gamma \act (X,\nu)$ be a measure\=/class\=/preserving action on a probability space.
Then $\rho$ is strongly ergodic \emph{if and only if} it is asymptotically expanding in measure.
\end{prop}

In particular, explicit examples of strongly ergodic actions (\emph{e.g.}, those constructed in \cite{abert2012dynamical}) give rise to explicit examples of asymptotically expanding actions.

\

The following is a localised version of Definition~\ref{defn:asymptotic expanding in measure}:

\begin{de}[{\cite[Definition 3.7]{dynamics1}}]\label{defn:domain of asymp expansion} 
Let $\rho\colon \Gamma\curvearrowright (X,\nu)$ be an action. A domain $Y \subseteq X$ is called a \emph{domain of asymptotic expansion} for $\rho$ if there exist functions $\fnc\colon (0,\frac{1}{2}]\to \RR_{>0}$ and $\fnk\colon(0,\frac{1}{2}]\to \NN$ such that for every $\alpha\in(0,\frac{1}{2}]$ and measurable $A\subseteq Y$ with $\alpha\nu(Y)\leq \nu(A) \leq \frac{\nu(Y)}{2}$, we have
\[
 \nu\bigparen{(B_{\fnk(\alpha)} \cdot A)\cap Y} > \paren{1+\fnc(\alpha)}\nu(A).
\]

For a finite $S\subseteq\Gamma$, we say that $Y$ is a \emph{domain of $(\fnc,S)$\=/asymptotic expansion} (or simply \emph{domain of $S$\=/asymptotic expansion}) if for every $\alpha\in(0,\frac{1}{2}]$ and measurable subset $A\subseteq X$ with $\alpha \nu(X)\leq \nu(A) \leq \frac{\nu(X)}{2}$, we have $ \nu\bigparen{(S \cdot A)\cap Y} > \paren{1+\fnc(\alpha)}\nu(A)$.
\end{de}

When $\nu$ is finite, an action $\rho\colon  \Gamma\curvearrowright (X,\nu)$ is asymptotically expanding (in measure) \emph{if and only if} $X$ is a domain of asymptotic expansion for $\rho$.

The fact that Definitions~\ref{defn:asymptotic expanding in measure} and \ref{defn:domain of asymp expansion} are only concerned with domains of measure at most $\nu(X)/2$ (resp. $\nu(Y)/2$) makes them easier to verify, but sometimes awkward to use. The following elementary result is helpful to bypass this issue:

\begin{lem}[{\cite[Lemma 3.8]{dynamics1}}]\label{lem:annoying 1/2 upper bound} 
Let  $Y\subseteq X$ be a domain of asymptotic expansion for an action $\Gamma\curvearrowright (X,\nu)$. Then there exist functions $\fnb\colon [\frac{1}{2},1)\to \RR_{>0}$ and $\fnh\colon[\frac{1}{2},1)\to \NN$ such that for every $\beta\in[\frac{1}{2},1)$, we have
\[
 \nu\bigparen{(B_{\fnh(\beta)} \cdot A)\cap Y} > \paren{1+\fnb(\beta)}\nu(A)
\]
 for every measurable subset $A\subseteq Y$ with $\frac{1}{2}\nu(Y)\leq \nu(A)\leq\beta \nu(Y)$. 
 \end{lem}

For later use, we end this subsection by recalling some structure results established in \cite[Section~4]{dynamics1}.

\begin{prop}[{\cite[Proposition 4.5 and 4.11]{dynamics1}}]\label{prop:exhausting domains by domains}
 Let $\Gamma\curvearrowright (X,\nu)$ be an action, $Y\subseteq X$ be a domain of asymptotic expansion and $(Z_n)_{n\in\NN}$ be a sequence of nested subsets of $Y$ with $\nu(Z_n)\to 0$. 
 Then there exist $N_0\in \NN$, a sequence of finite subsets $S_n\subseteq \Gamma$ and an exhaustion $Y_n\nearrow Y$ by domains of $S_n$\=/expansion such that $Y_n\subseteq Y\smallsetminus Z_n$ for every $n>N_0$.
\end{prop}

\begin{thm}[{\cite[Theorem~4.9]{dynamics1}}]\label{thm:structure theorem general}
 Let $\rho\colon\Gamma\curvearrowright (X,\nu)$ be a measure\=/class\=/preserving action. Then the following are equivalent:
 \begin{enumerate}[(1)]
  \item $\rho$ is strongly ergodic; 
  \item every finite measure subset is a domain of asymptotic expansion; 
  \item $\rho$ is ergodic and $X$ admits a domain of expansion.
 \end{enumerate}
\end{thm}

\begin{rmk}
 For measure\=/preserving actions, the equivalence ``$(1)\Leftrightarrow(3)$'' of Theorem~\ref{thm:structure theorem general} had been previously proved in \cite[Theorem A]{Mar18}.
\end{rmk}

\section{Expansion and reversible Markov kernels}
\label{sec:Markov.expansion}
In the first part of this section, we introduce the language of Markov kernels and review a general estimate for the Cheeger constant of a reversible Markov kernel in terms of the spectrum of the associated Laplacian operator. In the second part, we show that measure\=/class preserving actions give rise to reversible Markov kernels. This allows us to define the notion of (domain of) Markov expansion and to characterise asymptotic expansion in terms of exhaustions by domains of Markov expansion. This result will be pivotal in the subsequent sections.

\subsection{Preliminaries on Markov kernels}\label{sec:preliminaries on Markov kernels}
We begin by recalling a few elementary properties of reversible Markov kernels. We refer to the first chapters of \cite{Rev75} for more background and details.

\begin{de}
 Let $\CE$ be a $\sigma$\=/algebra on a set $X$. A \emph{Markov kernel} on the measurable space $(X,\CE)$ is a function $\Pi\colon X\times\CE\to [0,1]$ such that:
 \begin{enumerate}
  \item for every $x\in X$, the function $\Pi(x,\variable)\colon\CE\to[0,1]$ is a probability measure;
  \item for every $A\in \CE$, the function $\Pi(\variable,A)\colon X\to[0,1]$ is $\CE$\=/measurable.
 \end{enumerate}
\end{de}

If $f\colon X\to \RR$ is integrable with respect to the probability measure $\Pi(x,\variable)$, we denote its integral by 
\[
 \int_{X}f(y)\Pi(x,\d y)\coloneqq \int_{X}f(y) \d \Pi(x,\variable)(y)
\]
(the integral is then naturally extended to complex-valued functions). The associated \emph{Markov operator} $\fkP$ is a linear operator on the space of bounded $\CE$\=/measurable functions, defined by
\[
 \fkP f(x)\coloneqq  \int_{X}f(y)\Pi(x,\d y).
\]

Since $\Pi(\variable,A)$ is measurable for every $A\in\CE$, we can define an operator $\check\fkP$ on the space of measures on $(X,\CE)$ by letting
\[
 \check\fkP\nu(A)\coloneqq\int_X \Pi(x,A)\d\nu(x)
\]
for every measure $\nu$ on $(X,\CE)$. The operators $\fkP$ and $\check\fkP$ are dual to one another in the sense that
\begin{equation}\label{eq:dual.Markov.operator}
 \int_X \fkP f(x)\d\nu(x)=\int_X f(x)\d \check\fkP\nu(x),
\end{equation}
whenever the integrals are defined.

\begin{de}[\cite{Kai92}]\label{defn:markov boundary}
 Given a measure $\nu$ on $(X,\CE)$ and an $A\in\CE$, the \emph{($\nu$-)size of the boundary} of $A$ (with respect to $\Pi$) is defined as
 \[
  \abs{\partial_\Pi A}_\nu\coloneqq\int_A\Pi(x,X\smallsetminus A)\d \nu(x).
 \]
\end{de}

\begin{rmk}
 Heuristically, a Markov kernel can be described as ``moving mass across $X$'' without creating nor destroying it: the value $\Pi(x,A)$ is the proportion of the mass that is moved from the point $x$ into the set $A$. The measure $\check\fkP\nu$ is the distribution of mass on $X$ that is obtained after moving the initial distribution $\nu$ according to the kernel $\Pi$.
 The function $\fkP f$ assigns to a point $x\in X$ the expected value of $f$ when spreading $x$ across $X$ according to the kernel $\Pi$.

 The duality formula \eqref{eq:dual.Markov.operator} on the indicator function $f=\chi_A$ can be understood as saying that the total $\nu$\=/mass that is moved into a set $A$ by the kernel $\Pi$ is equal to $\nu$\=/integral of the likelihood that $\Pi$ will take $x$ into $A$.
 The size of the boundary of $A\in\CE$ is the amount of $\nu$\=/mass that is carried outside $A$ by $\Pi$.
\end{rmk}

We will only be concerned with some special Markov kernels:

\begin{de}
 A Markov kernel $\Pi$ is called \emph{reversible} if there exists a measure $m$ on $(X,\CE)$ such that
 \[
  \int_X f(x)\fkP g(x)\d m(x)=
  \int_X \fkP f(x) g(x)\d m(x)
 \]
for every pair of measurable bounded functions $f,g\colon X\to\RR$. The measure $m$ is said to be a \emph{reversing measure} for $\Pi$ (note that $m$ need not be unique in general).
To specify which reversing measure is being considered, we say that $\Pi$ is a reversible Markov kernel \emph{on $(X,m)$}.
\end{de}
Let $m$ be a measure on $X$. We define the measure $\mu$ on $X\times X$ by letting
\begin{equation}\label{eq:mu}
 \mu(A\times B)
 \coloneqq \int_A \Pi(x,B) \d m(x)
 =\int_{X}\chi_A(x)\fkP\chi_B(x)\d m(x)
\end{equation}
for every $A,B\in\CE$. Then $m$ is a reversing measure \emph{if and only if} $\mu$ is \emph{symmetric}, \emph{i.e.}, $\mu(A\times B)=\mu(B\times A)$ for every $A,B\in\CE$. In this case, we have 
\begin{equation}\label{eq:boundary}
 \mkbdry{A}= \mu\bigparen{A\times (X\smallsetminus A)} = \mu\bigparen{(X\smallsetminus A) \times A} = \mkbdry{X\smallsetminus A}.
\end{equation}
In other words, the $m$-size of the boundary of any measurable set is equal to the $m$-size of the boundary of its complement. 

For the rest of this section, let us fix a reversible Markov kernel $\Pi$ on $(X,m)$. We note that
\[
 \check\fkP m(A)
 =\mu(X\times A)=\mu(A\times X)
 =\int_A \Pi(x,X) \d m(x)=m(A),
\]
\emph{i.e.}, $m$ is \emph{invariant} under $\check\fkP$. Hence, the Jensen inequality yields:
\[
 \int_X \abs{\fkP f(x)}^2\d m(x)
 \leq \int_X \fkP \abs{f}^2(x) \d m(x)
 = \int_X \abs{f}^2(x) \d \check\fkP m(x)
 = \int_X \abs{f}^2(x) \d m(x)
 =\norm{f}_{m,2}^2.
\]
Therefore, the Markov operator $\fkP$ can be regarded as a bounded operator on $L^2(X,m)$ with norm $\norm{\fkP}\leq 1$. Since $m$ is reversing, the operator $\fkP$ is self\=/adjoint.

Now, for any $p \in [1,\infty)$ and any $f\in L^p(X,m)$, we define its \emph{$p$\=/Dirichlet energy} as
\[
 \CE_p(f)\coloneqq\frac{1}{2}\int_{X\times X}\abs{f(x)-f(y)}^p\d \mu(x,y).
\]
Since $\mu$ is symmetric, we note that
 \[
 \int_{X\times X}\abs{\chi_{A}(x)-\chi_{A}(y)}^p \d\mu(x,y)
 =\mu\bigparen{A\times (X\smallsetminus A)}+\mu\bigparen{(X\smallsetminus A)\times A} = 2\mu\bigparen{A\times (X\smallsetminus A)}.
 \]
Hence for every $p \in [1,\infty)$, we have
\begin{equation}\label{eq:dirichlet.equals.boundary}
  \CE_p(\chi_A)= \mkbdry{A}.
\end{equation}

Finally, we observe that for any $f\in L^2(X,m)$ we have
\[
 \CE_2(f)
 =\frac 12 \int_{X\times X}\abs{f}^2(x)+\abs{f}^2(y)-2 \mathfrak{Re}(f(x)\overline{f(y)})\ \d \mu(x,y)=\norm{f}_{m,2}^2-\angles{f,\fkP f}_m,
\]
where the last equality uses the reversibility.

We define the \emph{Laplacian} of $\Pi$ as $\Delta\coloneqq 1-\fkP$, then we have $\CE_2(f)=\angles{f,\Delta f}_m$ for every $f\in L^2(X,m)$. In particular, the Laplacian $\Delta$ is a positive self\=/adjoint operator whose spectrum is contained in $[0,2]$.

\subsection{Isoperimetric inequalities and spectra of Markov kernels}
\label{ssec:Markov.isoperimetric and spectra}

It is a well-known result that a sequence of finite graphs is a family of expanders \emph{if and only if} the Markov operators associated with the simple random walks have a uniform spectral gap \cite{alon1986eigenvalues,alon1985lambda1,dodziuk1984difference}.
A similar result is true---albeit not as widely known---in the context of Markov kernels.

Let $\Pi$ be a reversible Markov kernel on $(X,m)$, \textbf{where $m$ is a finite measure}. Then all constant functions on $X$ belong to $L^2(X,m)$ and are fixed by $\fkP$. It follows that $\norm{\fkP}=1$ and $1$ belongs to the spectrum of $\fkP$. Denote the orthogonal complement of the constant functions in $L^2(X,m)$ by $L^2_0(X,m)$, \emph{i.e.}, 
\[
L^2_0(X,m)\coloneqq\Bigbraces{f\in L^2(X,m)\bigmid \int_X f(x)\d m(x)=0}.
\]
 Note that $L^2_0(X,m)$ is $\fkP$-invariant and that the spectrum of the restriction of $\fkP$ on $L^2_0(X,m)$ is contained in $[-1,1]$. We denote the supremum of this spectrum by $\lambda_2\in \RR$. We make the following definition:

\begin{de}\label{spectral gap markov kernel}
 A reversible Markov kernel on a finite measure space $(X,m)$ is said to have a \emph{spectral gap} if $\lambda_2<1$. 
\end{de} 

It is clear from the definition that the reversible kernel $\Pi$ has a spectral gap \emph{if and only if} $1$ is isolated in the spectrum of $\fkP$ and the $1$\=/eigenspace consists of constant functions on $X$. Equivalently, this happens \emph{if and only if} $0$ is isolated in the spectrum of $\Delta=1-\fkP$ and the $0$\=/eigenspace consists of constant functions. Obviously, we have that 
\begin{equation}\label{eq: spectral gap}
 1-\lambda_2=\inf_{} \left\{\frac{\CE_2(f)}{\norm{f}_{m,2}^2}\;\middle|\; f\in L^2_0(X,m) \right\}.
\end{equation}

In analogy with the notion of Cheeger constants for finite graphs, we define:
\begin{de}\label{defn:Cheeger.constant.Markov}
The \emph{Cheeger constant} for a reversible Markov kernel $\Pi$ on $(X,m)$ is 
 \[
  \kappa\coloneqq\inf\left\{\frac{\mkbdry{A}}{m(A)}\;\middle|\; A\in\CE,\ 0<m(A)\leq\frac 12 m(X)\right\}.
 \]
\end{de}

We can now state the following theorem relating Cheeger constants and spectral gaps in the context of Markov kernels:
\begin{thm}[{\cite[Theorem 2.1]{lawler1988bounds}}]\label{thm:spectral.characterisation.markov.exp}
 Let $\Pi$ be a reversible Markov kernel on $(X,m)$ where $m$ is finite. Then
 \[
  \frac{\kappa^2}{2} \leq 1-\lambda_2\leq 2\kappa.
 \]
\end{thm}

\begin{rmk}
 We are grateful to the anonymous referee for pointing out \cite{lawler1988bounds} to us. We should remark that the authors of \cite{lawler1988bounds} use a slightly different notion of Cheeger constant for Markov kernels. Moreover, the inequality they prove has slightly different constants and it is not sharp (see also the remark below \cite[Proposition 2.2]{lawler1988bounds}). For completeness, we provide a self\=/contained proof of Theorem~\ref{thm:spectral.characterisation.markov.exp} in the appendix.
\end{rmk}

\subsection{Markov kernels from actions}
\label{ssec:actions.and.markov.kernels}
Let $\Gamma\curvearrowright (X,\nu)$ be a measure\=/class-preserving action. For every $\gamma\in \Gamma$ and $x\in X$, let $r(\gamma,x)\coloneqq\frac{d \gamma^{-1}_*\nu}{d\nu}(x)$ be the Radon\=/Nikodym derivative. Note that $r(\gamma,x)=r(\gamma^{-1},\gamma(x))^{-1}$ and for any measurable function $f$ on $X$ we have
\[
 \int_X f(\gamma\cdot x)d\nu(x)=\int_X f(x)r(\gamma^{-1},x)d\nu(x)
\]
when the integrals exist. In particular, for every measurable $Y\subseteq X$ we have
\begin{equation}\label{eq:change.of.variable}
 \int_Y f(\gamma\cdot x)r(\gamma,x)^{\frac 12}d\nu(x)=\int_{\gamma(Y)} f(x)r(\gamma^{-1},x)^\frac 12 d\nu(x)
\end{equation}
when the integrals exist. 

Now fix a finite symmetric subset $S\subseteq \Gamma$ containing the identity $1$, and a measurable subset $Y\subseteq X$ (which might have infinite measure). For every $x\in Y$, let $S_{Y,x}\coloneqq\braces{s\in S\mid s\cdot x\in Y}$ and
\begin{equation}\label{function sigma}
 \sigma_{Y,S}(x)\coloneqq \sum_{s\in S_{Y,x}}r(s,x)^\frac{1}{2}.
\end{equation}

\begin{de}\label{defn:normalized.Markov.kernel}
Let $\Gamma\curvearrowright (X,\nu)$ be a measure\=/class-preserving action. The \emph{normalised local Markov kernel} associated with $Y$ and $S$ is the Markov kernel on $Y$ defined by
 \[
  \Pi_{Y,S}(x,\variable)\coloneqq\frac{1}{\sigma_{Y,S}(x)}\sum_{s\in S_{Y,x}}r(s,x)^\frac 12\delta_{s\cdot x}
 \]
 where $\delta_y$ is the Dirac delta measure on the point $y$, and we denote the associated Markov operator by $\fkP_{Y,S}$. We say that $\Pi_S\coloneqq\Pi_{X,S}$ is the \emph{normalised Markov kernel} associated with $S$. 
\end{de}

For later use, we record here an elementary but convenient integration formula: for every measurable function $G\colon S\times Y\to \CCC$ we have
\begin{equation}\label{eq:integral.Sum.Sx}
 \int_Y\sum_{s\in S_{Y,x}}G(s,x)\d\nu(x)
 =\int_Y\sum_{s\in S}\chi_{\braces{s^{-1}(Y)}}(x)G(s,x)\d\nu(x)
 =\sum_{s\in S}\int_{Y\cap s^{-1}(Y)}\hspace{-2em}G(s,x)\,\d\nu(x)
\end{equation}
(when the integrals are defined).

One of the key properties of normalised local Markov kernels is that they are reversible. In fact, consider the measure $\snu$ on $Y$ defined by 
\[
 \d\snu \coloneqq \sigma_{Y,S}\cdot \d(\nu|_Y).
\]
In other words, $\snu$ is obtained by rescaling the restriction of $\nu$ to $Y$ by the density function $\sigma_{Y,S}$. Then the following holds true:

\begin{prop}\label{prop:normalised.markov.is.reversible}
Let $\Gamma\curvearrowright (X,\nu)$ be a measure\=/class-preserving action and $S$ be a finite symmetric subset of $\Gamma$ containing the identity $1$. Then:
\begin{enumerate}
  \item The measure $\snu$ is equivalent to the restriction $\nu|_Y$.
  \item If $\nu(Y)$ is finite, then $\nu(A) \leq \snu(A) \leq |S|\sqrt{\nu(A)\nu(Y)}$ for any measurable $A \subseteq Y$. In particular, in this case $\snu(Y)$ is also finite.
  \item The measure $\snu$ is reversing for the normalised local Markov kernel $\Pi_{Y,S}$. The associated measure $\mu$ on $Y\times Y$---defined by (\ref{eq:mu})---is determined by the formula:
 \begin{equation}\label{eq:mu for Markov from actions}
  \int_{Y\times Y} F(x,y)\d \mu(x,y)=\sum_{s\in S}\int_{Y\cap s^{-1}(Y)} r(s,x)^\frac 12 F(x,s\cdot x) \d\nu(x)
 \end{equation}
 for every integrable function $F$ on $Y\times Y$.
\end{enumerate}
\end{prop}

\begin{proof}
(1). Note that $0<r(s,x)<\infty$ for $\nu$\=/almost every $x\in X$ because the action is measure\=/class-preserving. Since $S$ contains the identity $1$, we know that $S_{Y,x}$ is non\=/empty for every $x\in Y$. It follows immediately that a measurable subset of $Y$ is $\nu$\=/null if and only if it is $\snu$\=/null. 
 
(2). Since $1 \in S_{Y,x}$, we have $\nu(A) \leq \snu(A)$ for any measurable $A \subseteq Y$. On the other hand, by \eqref{eq:integral.Sum.Sx} and the Cauchy--Schwarz inequality we have
 \begin{align*}
  \snu(A)
  &=\sum_{s\in S}\int_{A\cap s^{-1}(Y)} r(s,x)^\frac 12\d\nu(x)
  \leq \sum_{s\in S} \nu(A\cap s^{-1}(Y))^{\frac{1}{2}} \big(\int_{A\cap s^{-1}(Y)} r(s,x) \d\nu(x) \big)^{\frac{1}{2}}\\
  &\leq \big( \sum_{s\in S} \nu(A) \big)^\frac{1}{2} \cdot \big( \sum_{s\in S} \nu((s \cdot A) \cap Y)\big)^\frac{1}{2}
  \leq |S|\sqrt{\nu(A)\nu(Y)}.
  \end{align*}

(3). Let us first verify the formula for $\mu$. By definition, for any measurable function $F$ on $Y\times Y$ we have that
 \begin{align*}
  \int_{Y\times Y} F(x,y)\d \mu(x,y)
  &=\int_Y \int_Y F(x,y) \Pi_{Y,S}(x,\d y)\d \snu(x)  \\
  &=\int_Y \frac{1}{\sigma_{Y,S}(x)}
  \sum_{s\in S_{Y,x}} r(s,x)^\frac 12 F(x,s\cdot x) \d \snu(x)  \\
  &=\int_Y\sum_{s\in S_{Y,x}}r(s,x)^\frac 12 F(x,s\cdot x)\d \nu(x) \\
  &=\sum_{s\in S}\int_{Y\cap s^{-1}(Y)} r(s,x)^\frac 12 F(x,s\cdot x) \d\nu(x),
 \end{align*}
 where the last step follows from \eqref{eq:integral.Sum.Sx}.

In order to show that $\snu$ is reversing for $\Pi_{Y,S}$, it suffices to prove:
\[
\int_{Y\times Y} F(x,y)\d \mu(x,y) = \int_{Y\times Y} F(y,x)\d \mu(x,y)
\]
for every measurable function $F$ on $Y \times Y$. From \eqref{eq:mu for Markov from actions}, we have that
\begin{align*}
\int_{Y\times Y} F(y,x)\d \mu(x,y) &= \sum_{s\in S}\int_{Y\cap s^{-1}(Y)} r(s,x)^\frac 12 F(s\cdot x,x) \d\nu(x)\\
&=\sum_{s\in S}\int_{sY\cap Y} r(s^{-1},x)^\frac 12 F(x,s^{-1}\cdot x) \d\nu(x)\\
&=\sum_{s\in S}\int_{Y\cap s^{-1}(Y)} r(s,x)^\frac 12 F(x,s\cdot x) \d\nu(x) \\
&=\int_{Y\times Y} F(x,y)\d \mu(x,y),
\end{align*}
where we use \eqref{eq:change.of.variable} for the second equation, and use $S=S^{-1}$ for the third one.
\end{proof}

\begin{rmk}\label{rmk:1.in.S}
 The assumption that $1\in S$ is only used to ensure that $S_{Y,x}$ is always non\=/empty for every $x\in Y$. This assumption can be dropped if one already knows, \emph{a priori}, that $S_{Y,x}$ is non\=/empty (\emph{e.g.}, $Y$ is $\Gamma$\=/invariant). On the contrary, the condition that $S=S^{-1}$ is essential for the proof of reversibility in Proposition~\ref{prop:normalised.markov.is.reversible}. 
\end{rmk}

Having introduced the reversible normalised local Markov kernel $\Pi_{Y,S}$, we would like to apply the techniques developed in previous subsections to asymptotically expanding actions. Firstly, let us give the following definition:

\begin{de}\label{defn:domain.Markov.expansion}
Let $\Gamma\curvearrowright (X,\nu)$ be a measure\=/class\=/preserving action and $Y\subseteq X$ be a domain. Let $S\subseteq \Gamma$ be a finite symmetric subset with $1\in S$, then $Y$ is called a \emph{domain of Markov $S$\=/expansion (for the action)} if the associated normalised local Markov kernel $\Pi_{Y,S}$ has strictly positive Cheeger constant (see Definition~\ref{defn:Cheeger.constant.Markov}). $Y$ is called a \emph{domain of Markov expansion} if it is a domain of Markov $S$\=/expansion for some finite symmetric $S\subseteq \Gamma$ with $1\in S$.
\end{de}

By Theorem \ref{thm:spectral.characterisation.markov.exp}, $Y$ is a domain of Markov $S$\=/expansion \emph{if and only if} the normalised local Markov kernel $\Pi_{Y,S}$ has spectral gap. In other words, $1$ is isolated in the spectrum of the Markov operator $\fkP_{Y,S}$ and the $1$\=/eigenspace consists of constant functions on $X$.
When this is the case, restriction of the Markov operator $\fkP_{Y,S}$ on $L^2_0(Y,\snu)$ has spectrum contained in $[-1,\lambda_2]\subset[-1,1)$. The restriction of the operator $\frac 12 +\frac 12\fkP_{Y,S}$ to $L^2_0(Y,\snu)$ has spectrum contained in $[-\frac{3}4,\frac{\lambda_2+1}2]$ (this is the Markov operator obtained by lazyfying the Markov process). For future reference, we record this observation as a lemma.

\begin{lem}\label{lem:spetral gap equiv}
The Markov kernel $\Pi_{Y,S}$ has spectral gap \emph{if and only if} the restriction of the lazy Markov operator $\frac 12 + \frac 12\fkP_{Y,S}$ to $L^2_0(Y,\snu)$ has spectrum contained in $[-\frac{3}{4},1-\epsilon]$ with $\epsilon>0$ (and hence has norm strictly less than $1$).
\end{lem}
% \begin{proof}
% The proof follows from direct calculations and a classical argument using the uniform convexity of Hilbert spaces and is hence left to the reader (note that the spectrum of $\fkP_{Y,S}$ is bounded away from $-1$ as $1\in S$ implies that
% \[
%   \scal{\fkP_{Y,S}f}{f}
%   =\frac{1}{\abs{S}}\norm{f}^2+\frac{1}{\abs{S}}\sum_{s\in S\smallsetminus\{e\}}\scal{\pi(s)f}{f}
%   \geq \frac{1}{\abs{S}}-\frac{\abs{S}-1}{\abs{S}}=-1+\frac{2}{\abs{S}}.
%  \]).
% \end{proof}

The following result provides the connection between expansion in measure (Definition \ref{defn:domain of expansion}) and Markov expansion (Definition~\ref{defn:domain.Markov.expansion}) under an assumption of bounded Radon--Nikodym derivatives:

\begin{lem}\label{lem:domain.of.exp iff Markov.exp}
Let $\Gamma\curvearrowright (X,\nu)$ be a measure\=/class\=/preserving action, $Y\subseteq X$ be a domain and $S$ be a finite symmetric subset of $\Gamma$ containing the identity. If there is a constant $\Theta\geq 1$ such that $1/\Theta\leq r(s,x)\leq\Theta$ for every $x\in Y$ and $s\in S_{Y,x}$, then $Y$ is a domain of $S$\=/expansion \emph{if and only if} it is a domain of Markov $S$\=/expansion. 
\end{lem}

\begin{proof}
By definition and \eqref{eq:mu for Markov from actions}, for any measurable subset $A\subseteq Y$ we have that
\begin{align*}%\label{eq:formula.locmkbdry}
  \locmkbdry{A}
&=\int_{Y \times Y}\chi_A(x)\chi_{Y\smallsetminus A}(y) \d\mu(x,y)   
= \sum_{s\in S}\int_{Y\cap s^{-1}(Y)} r(s,x)^\frac 12 \chi_A(x)\chi_{Y\smallsetminus A}(s\cdot x)\d\nu(x)\\
&=\sum_{s\in S}\int_{(A\smallsetminus s^{-1}(A))\cap s^{-1}(Y)}\hspace{-4ex}r(s,x)^\frac 12 \d\nu(x)=\sum_{s\in S}\int_{(s\cdot A\smallsetminus A)\cap Y}\hspace{-4ex}r(s^{-1},x)^\frac 12 \d\nu(x),
\end{align*}
where the last equality uses (\ref{eq:change.of.variable}).

Hence, it follows from the assumption on $r$ that
 \begin{equation}\label{eq:edge.vertex.estimate}
  \frac{1}{\sqrt{\Theta}}\locmkbdry{A}
  \leq \sum_{s\in S}\nu\bigparen{(s\cdot A\smallsetminus A)\cap Y}
  \leq \sqrt{\Theta}\locmkbdry{A}.
 \end{equation}
 Moreover, it follows by the definition of $\snu$ that
\begin{equation}\label{eq:comparing.measures}
  \nu(A)
  \leq\snu(A)
  \leq \abs{S}\sqrt{\Theta} \cdot \nu(A)
\end{equation}
for any measurable subset $A\subseteq Y$.
 
Now we assume that $Y$ is a domain of $(c,S)$\=/expansion for some constant $c>0$. Fix a measurable subset $A\subseteq Y$ with $0<\snu(A)\leq\frac 12\snu(Y)$. In particular, both $\nu(A)>0$ and $\nu(Y\smallsetminus A)>0$ by Proposition~\ref{prop:normalised.markov.is.reversible}(1).

If $\nu(A)\leq \frac 12\nu(Y)$, it follows from Definition~\ref{defn:domain of expansion} and (\ref{eq:edge.vertex.estimate}) that
 \[
  c\nu(A)
  <\nu\paren{(S\cdot A\smallsetminus A)\cap Y}
  \leq \sum_{s\in S} \nu\paren{(s\cdot A\smallsetminus A)\cap Y}
  \leq\sqrt{\Theta}\locmkbdry{A}.
 \]
Together with \eqref{eq:comparing.measures}, we conclude that 
 \[
  \locmkbdry{A}
  > \frac{c}{\abs{S}\Theta}\snu(A).
 \]
 If $\nu(A)> \frac 12\nu(Y)$, we can apply the same argument to $Y\smallsetminus A$ and deduce from (\ref{eq:boundary}) that
 \[
  \locmkbdry{A}=\locmkbdry{Y\smallsetminus A}
  > \frac{c}{\abs{S}\Theta}\snu(Y\smallsetminus A)
  \geq \frac{c}{\abs{S}\Theta}\snu(A),
 \]
where the last inequality follows from the assumption that $\snu(A)\leq\frac 12\snu(Y)$. Thus, $Y$ is a domain of Markov $S$\=/expansion as desired.

 The proof of the converse implication is similar.
 Let $\kappa>0$ be the Cheeger constant for the normalised local Markov kernel $\Pi_{Y,S}$ and fix any measurable subset $A\subset Y$ with $0<\nu(A)\leq\frac 12\nu(Y)$. 
 
If $\snu(A)\leq\frac 12\snu(Y)$, then \eqref{eq:edge.vertex.estimate} implies that
 \[
  \nu\paren{\acbdry_SA\cap Y}
  \geq \frac{1}{\abs{S}}\sum_{s\in S} \nu\bigparen{(s\cdot A\smallsetminus A)\cap Y}
  \geq \frac{1}{\abs{S}\sqrt{\Theta}}\locmkbdry{A}
  \geq \frac{\kappa}{\abs{S}\sqrt{\Theta}} \snu(A),
 \]
 where $\acbdry_SA=S\cdot A\smallsetminus A$. Together with \eqref{eq:comparing.measures} we obtain that
 \[
  \nu\paren{\acbdry_SA\cap Y}
  \geq \frac{\kappa}{\abs{S}\sqrt{\Theta}} \nu(A).
 \]
 If $\snu(A)>\frac 12\snu(Y)$, then $\acbdry_SA\cap Y\supseteq \acbdry_S(Y\smallsetminus S\cdot A)\cap Y$ implies that
 \[
  \nu\paren{\acbdry_SA\cap Y}\geq
  \nu\bigparen{\acbdry_S(Y\smallsetminus S\cdot A)\cap Y}
 \geq \frac{\kappa}{\abs{S}\sqrt{\Theta}} \nu\paren{Y\smallsetminus S\cdot A}.
 \]
Moreover, using $1\in S$ we note that
 \[
  \nu\paren{Y\smallsetminus S\cdot A}
  =\nu\paren{Y}-\nu(A) -\nu(\acbdry_S A\cap Y)\geq \nu(A)-\nu(\acbdry_S A\cap Y).
 \]
So it is easy to conclude that
\[
\nu\paren{\acbdry_SA\cap Y} \geq \frac{\kappa}{\abs{S}\sqrt{\Theta} + \kappa} \cdot \nu(A).
\]
This shows that $Y$ is a domain of $S$-expansion for the action.
\end{proof}

\begin{rmk}
 The statement of Lemma~\ref{lem:domain.of.exp iff Markov.exp} is an analogue of the fact that for graphs with bounded degrees, there are bounds between edge\=/expansion and vertex\=/expansion. More precisely, the Cheeger constant of the normalised (local) Markov kernel should be regarded as the ``measured'' Cheeger constant of the edge\=/expansion, while the notion of expansion in measure is clearly an analogue of the (exterior) vertex\=/expansion for graphs. The assumption that the Radon\=/Nikodym derivatives are bounded corresponds to that the graphs have bounded degree. 
\end{rmk}

Consequently, we obtain an alternative and direct proof for \cite[Lemma 5.2]{grabowski_measurable_2016}:
\begin{cor}[{\cite[Lemma 5.2]{grabowski_measurable_2016}}]
\label{cor:local.spec.gap.iff.domain.of.expansion}
Let $\Gamma\curvearrowright (X,\nu)$ be a measure-preserving action and $Y\subseteq X$ a domain. Then $Y$ is a domain of expansion \emph{if and only if} the action has local spectral gap with respect to $Y$.
\end{cor}

\begin{proof}
 Using indicator functions, it is easy to see that the existence of a local spectral gap implies that $Y$ is a domain of expansion. Hence, we only focus on the converse implication.
 
 Let $Y$ be a domain of $(c,k)$\=/expansion and let $S\coloneqq B_k$. Then the normalised local Markov kernel $\Pi_{Y,S}$ has a spectral gap by Theorem~\ref{thm:spectral.characterisation.markov.exp} and Lemma~\ref{lem:domain.of.exp iff Markov.exp}. Using \eqref{eq:mu for Markov from actions} in Proposition~\ref{prop:normalised.markov.is.reversible}, we obtain that for every $g\in L^2_0(Y,\snu)$:
\begin{equation}\label{eq: cor 4.17}
(1-\lambda_2)\cdot\norm{g}_{\snu,2}^2\leq \CE_2(g)=\frac{1}{2}\sum_{s\in S}\int_{Y\cap s^{-1}(Y)} \abs{g(x)-g(s\cdot x)}^2 \d\nu(x).
\end{equation}
 where $1-\lambda_2$ is bounded away from zero (Definition~\ref{spectral gap markov kernel}).
 
Now we fix an $f\in L^2(X,\nu)$ with $\int_Y f\d\nu=0$. Then $f|_Y\in L^2(Y,\snu)$ and $g\coloneqq f|_Y-\int_Y f|_Y\d\snu \in L^2_0(Y,\snu)$ by construction and Proposition~\ref{prop:normalised.markov.is.reversible}(2). Thus, it follows from $S=S^{-1}$ and (\ref{eq: cor 4.17}) that
 \[
  \sum_{s\in S}\norm{s\cdot f-f}_{\nu|_Y,2}
  \geq\Bigparen{\sum_{s\in S}\int_{Y} \abs{f(x)-f(s\cdot x)}^2 \d\nu(x)}^\frac 12
   \geq  \paren{2\CE_2(g)}^\frac 12
  \geq \sqrt{2(1-\lambda_2)}\cdot\norm{g}_{\snu,2}.
 \]
Moreover, since $\int_Y f\d\nu=0$ we see that
 \[
  \norm{g}_{\snu,2}^2\geq \norm{g}_{\nu|_Y,2}^2=\norm{f}_{\nu|_Y,2}^2+ \nu(Y)\Bigparen{\int_Y f|_Y\d\snu}^2\geq \norm{f}_{\nu|_Y,2}^2.
 \]
Combining the above inequalities we conclude that
 \[
 \sum_{s\in S}\norm{s\cdot f-f}_{\nu|_Y,2} \geq \sqrt{2(1-\lambda_2)}\cdot\norm{f}_{\nu|_Y,2},
 \]
 as required.
\end{proof}

\begin{rmk}\label{rmk:bdd ratio}
Note that the proof of Corollary~\ref{cor:local.spec.gap.iff.domain.of.expansion} holds also for non-measure-preserving actions as long as the action is measure-class-preserving and there is a uniform upper bound $\Theta\geq 1$ on the Radon--Nikodym derivatives $r(s,x)$.  
This can be used to provide an alternative proof for \cite[Theorem 3.2]{houdayer2017strongly}. 
\end{rmk}

\subsection{Markov expansion and the structure of strongly ergodic actions}
\label{ssec:Markov structure thm}
Now we are in the position to prove the Markovian analogue of the structure theorem for strongly ergodic actions (Theorem~\ref{thm:structure theorem general}). Let us start with the following local version (compare with Proposition \ref{prop:exhausting domains by domains}):
\begin{prop}\label{prop:Markov local version}
 Let $\rho\colon \Gamma\curvearrowright (X,\nu)$ be a measure\=/class-preserving action. If $Y\subseteq X$ is a domain of asymptotic expansion, then $Y$ admits an exhaustion by domains $Y_n$ of Markov $S^{(n)}$-expansion such that for each $n\in\NN$ there is a constant $\Theta_n\geq 1$ such that $1/\Theta_n\leq r(s,y)\leq\Theta_n$ for every $y\in Y_n$ and $s\in S^{(n)}_{Y_n,y}$.

Moreover, if $Y\subseteq X$ is a domain of $S$\=/asymptotic expansion, then $Y$ admits an exhaustion by domains $Y_n$ of Markov $S$\=/expansion.
\end{prop}

\begin{proof}
It follows from Proposition \ref{prop:exhausting domains by domains} that there exists an exhaustion $Y^{(k)}\nearrow Y$ by domains of $S^{(k)}$\=/expansion in measure.
Without loss of generality, we can assume that $S^{(k)}$ is symmetric, $1\in S^{(k)}$ and $S^{(k)}\subseteq S^{(k+1)}$ for every $k\in\NN$. Let 
 \[
  Z^{(k)}_m\coloneqq \bigbraces{y\in Y^{(k)}\bigmid r(s,y)<\frac{1}{m}\text{ or } r(s,y)> m\text{ for some }s\in S^{(k)}_{Y^{(k)},y}}.
 \]
Since each $Y^{(k)}$ has finite measure and the action is measure\=/class-preserving, for every $k\in\NN$ we have $\nu(Z^{(k)}_m) \to 0$ as $m\to \infty$.
Hence, we can choose for every $n\in\NN$ a sequence of integers $(m^{(n)}_k)_{k\in\NN}$ such that
 \[
  \sum_{k\in\NN}\nu\bigparen{Z^{(k)}_{m^{(n)}_k}}\leq \frac 1n.
 \]
 Let 
 \[
  \widetilde Z_n\coloneqq
  \bigcup_{k\in\NN}Z^{(k)}_{m^{(n)}_k},
 \]
 then we have $\nu(\widetilde Z_n)\to 0$ as $n\to \infty$. We can further assume that $m^{(n+1)}_k\geq m^{(n)}_k$ for every $n\in\NN$ so that $\widetilde Z_{n+1}\subseteq \widetilde Z_n$.
 
Now for every $k$, it follows from Proposition~\ref{prop:exhausting domains by domains} that $Y^{(k)}$ admits an exhaustion $Y^{(k)}_l\nearrow Y^{(k)}$ by domains of $S^{(k)}$\=/expansion such that $Y^{(k)}_l\cap\widetilde Z_l=\emptyset$. By a diagonal argument, there exists a sequence $(l_k)_{k\in\NN}$ such that $Y^{(k)}_{l_k}$ converges in measure to $Y$. Let 
 \[
  Y_n\coloneqq \bigcup_{k=0}^n Y^{(k)}_{l_k}.
 \]
 Ignoring finitely many $k$ if necessary, we can assume that $\nu(Y_{l_0}^{(0)})\geq\frac{3}{4}\nu(Y)$. Then we conclude from Lemma~\ref{lem:union of domains} that each $Y_n$ is a domain of $S^{(n)}$\=/expansion. Since $Y_n\cap \widetilde Z_n=\emptyset$, it follows from Lemma~\ref{lem:domain.of.exp iff Markov.exp} that it is also a domain of Markov $S^{(n)}$\=/expansion. The second statement is obtained by the special case where $S^{(k)}=S$ for all $k\in \NN$.
\end{proof}

\begin{rmk}
The main technical difficulty in the previous proof is to obtain \emph{increasing} sequences of domains of Markov expansion. This is largely due to the fact that it is hard to control Markov $S$\=/expansion as the set $S$ varies. More precisely, choosing different $S$ could yield widely different measures $\snu$ and this would in turn influence the Cheeger constant of $\Pi_{Y,S}$.
 
An alternative approach to Proposition~\ref{prop:Markov local version} would be to go through the proof of Proposition~\ref{prop:exhausting domains by domains} and reprove it using the language of Markov expansion.
\end{rmk}

It is now simple to prove a structure result for strongly ergodic actions in terms of Markovian expansion:
\begin{thm}\label{thm:structure theorem Markov}
 Let $\rho\colon\Gamma\curvearrowright (X,\nu)$ be a measure\=/class-preserving action. Then $\rho$ is strongly ergodic \emph{if and only if} every domain $Y\subseteq X$ admits an exhaustion by domains of Markov expansion.
\end{thm}

\begin{proof}
\emph{Necessity}: This follows from Theorem~\ref{thm:structure theorem general} ``(1)$\Rightarrow$(2)'' and Proposition~\ref{prop:Markov local version}.

\emph{Sufficiency}: It follows from the same argument as in the proof of \cite[Theorem 4.9 ``(5)$\Rightarrow$(6)'']{dynamics1} that $\rho$ must be ergodic.
By Theorem~\ref{thm:structure theorem general} ``(3)$\Rightarrow$(1)'', it is hence enough to show that $X$ admits a domain of expansion. 

We can choose a domain $Y\subseteq X$ for which there exist constants $C(\gamma) \geq 1$ depending on $\gamma\in \Gamma$ such that $C(\gamma)^{-1} \leq r(\gamma,y)\leq C(\gamma)$ for every $y\in Y$ (such a domain can be constructed using an argument similar to that in the proof of Proposition~\ref{prop:Markov local version}). By the hypothesis, there is an exhaustion $Y_n\nearrow Y$ by domains of Markov expansion. Finally, Lemma \ref{lem:domain.of.exp iff Markov.exp} implies that any such $Y_n$ produces the desired domain of expansion in measure. 
\end{proof}

\section{Warped cones and finite (dynamical) propagation approximations}\label{sec:warped.cones}
\label{sec:warped cones}
The aim of this section is to introduce warped cones associated with group actions on metric measure spaces and to study the effects of asymptotic expansion on the analytic properties of said warped cones.
More precisely, adapting the techniques in \cite{structure} to the context of group actions and using the structure results in Subsection~\ref{ssec:Markov structure thm}, we can characterise asymptotic expansion in terms of finite propagation approximations of the Druţu--Nowak projections. As an intermediate bridge, we introduce dynamical versions of quasi-locality and finite propagation approximation to connect actions and projections on warped cones. 
In turn, this allows us to construct a multitude of non\=/compact ghost projections which will be used in Section~\ref{sec:baum_connes} to construct counterexamples to the coarse Baum--Connes conjecture.

\subsection{Preliminaries on warped cones}\label{ssec:pre on warped cones}
Recall that the countable group $\Gamma$ is equipped with a proper length function $\ell$. Let $(X,d)$ be a metric space and $\rho\colon \Gamma\curvearrowright X$ be a continuous action. For every $t\geq 1$ let $d^t$ be the rescaling of $d$ by $t$, \emph{i.e.}, $d^t(x,y)\coloneqq td(x,y)$.

\begin{de}
 The \emph{warped cone} associated with the action $\Gamma\curvearrowright X$ is the family of metric spaces $\wc{\Gamma\curvearrowright X}\coloneqq\braces{(X,d_\Gamma^t)\mid t\in[1,\infty)}$, where $d_\Gamma^t$ is the largest metric such that
\[
d_\Gamma^t \leq d^t \quad \text{ and } \quad d_\Gamma^t\paren{x,\gamma\cdot x}\leq \ell(\gamma)
\]
for every $x\in X$ and $\gamma\in \Gamma$.

If the diameter of $(X,d)$ is at most $2$, we can also define the \emph{unified warped cone} as the metric space $(\CO_\Gamma X,d_\Gamma)$, where $\CO_\Gamma X= X\times[1,\infty)$ as a set and 
\[
 d_\Gamma\bigparen{(x_1,t_1),(x_2,t_2)}\coloneqq d_\Gamma^{t_1\wedge t_2}(x_1,x_2) + \abs{t_1-t_2}
\]
where $t_1\wedge t_2=\min\braces{t_1,t_2}$. The requirement on the diameter is necessary to ensure that $d_\Gamma$ is a metric.
\end{de}
 
We will also need the following:

\begin{lem}\label{lem:intersection.of.neighbourhoods}
 Let $\Gamma\curvearrowright (X,d)$ be a continuous action and $R>0$ fixed.
 Given $A\subseteq X$, let $N_R(A;d_\Gamma^t)\subseteq X$ be the closed $R$\=/neighbourhood of $A$ with respect to the metric $d_\Gamma^t$. Then
 \[
  \bigcap_{t\geq1}N_R(A;d_\Gamma^t) =\overline{B_R\cdot A}.
 \]
\end{lem}
\begin{proof}
 It is clear that $\overline{B_R\cdot A}$ is contained in $N_R(A;d_\Gamma^t)$ for every $t\geq1$. For the converse, it suffices to prove it for $A$ closed. If $R<1$, we see that $N_R(A;d_\Gamma^t)=N_R(A;d^t)=N_{R/t}(A;d)$, because $\ell$ only takes integer values. So the result holds trivially. 

 By induction on $n\in\NN$, we will prove that the claim holds for every $R< n$ and every closed $A\subset X$. First note that for every fixed $\gamma\in\Gamma$ we have
 \begin{equation}\label{eq:intersection nbhds}
  \bigcap_{t\geq1}\gamma\cdot N_R(A;d^t)=\gamma(A).
 \end{equation}
 Also note that the warped distance can be computed as
 \begin{equation}\label{eq:warped metric concrete}
  d^t_\Gamma(x,y)=\inf_\xi \Bigparen{ \sum_{i=0}^k d^t(x_i,y_i)+\sum_{i=1}^k\abs{\gamma_i} }  
 \end{equation}
 where the infimum is taken over $k\in\NN$ and sequences $\xi$ of points $x_0,\ldots x_k,y_0,\ldots y_k\in X$ and elements $\gamma_1,\ldots \gamma_k\in \Gamma$ so that $x=x_0$, $y=y_k$ and $x_i=\gamma_{i-1}(y_{i-1})$ for every $1\leq i\leq k$ (this expression is obtained by imposing that $d^t_\Gamma$ satisfies the triangle inequality).
 
 Fix now some $0< R<n$. For every $y\in N_R(A;d_\Gamma^t)\smallsetminus N_R(A;d^t)$ we can take a sequence of sequences  $\xi_l$ converging to the infimum in \eqref{eq:warped metric concrete}. Since $y\notin N_R(A;d^t)$ we can assume that each sequence $\xi_l$ has length $k\geq 1$. Since $B_R$ is finite, we can also pass to a subsequence $\xi_{l_m}$ so that each sequence $\xi_{l_m}$ has the same $\gamma_1$. Denote this element by $\gamma_{1,y}\in B_R$. Since $A$ is closed, it follows that 
 \[
  y\in N_{R-\abs{\gamma_{1,y}}}\bigparen{\gamma_{1,y}\cdot N_{R-\abs{\gamma_{1,y}}}(A;d^t) \, ;\, d^t_\Gamma}.
 \]
 As a consequence, we deduce that
 \begin{equation}\label{eq:neighbourhood.containment}
  N_R(A;d_\Gamma^t)\subseteq N_R(A;d^t)
  \cup \left( \bigcup_{m=1}^{n-1} 
  N_{R-m}\bigparen{B_{m}\cdot N_{R-m}(A;d^t)\, ;\, d^t_\Gamma} 
  \right).
 \end{equation}

 For every $1\leq m\leq n-1$ and $t>0$, the set $C_{m,t}\coloneqq B_{m}\cdot N_{R-m}(A;d^{t})$ is closed. If we fix $t_0>1$ we can apply the induction hypothesis on the neighbourhoods of $C_{m,t_0}$ to deduce that
 \[
  \bigcap_{t\geq1} N_{R-m}\bigparen{C_{m,t}\, ;\, d^t_\Gamma}
  \subseteq \bigcap_{t>t_0} N_{R-1}\bigparen{C_{m,t_0}\, ;\, d^t_\Gamma}
  = B_{R-m}\cdot C_{m,t_0}
  \subseteq B_R\cdot N_{R-m}(A;d^{t_0}).
 \]
 Therefore, for each $1\leq m\leq n-1$  we have
 \[
  \bigcap_{t\geq1} N_{R-m}\bigparen{C_{m,t}\, ;\, d^t_\Gamma}
  \subseteq \bigcap_{t_0>1} B_{R}\cdot N_{R-m}(A;d^{t_0})
  = B_{ R }\cdot A,
 \]
 where we used \eqref{eq:intersection nbhds} on the finitely many elements $\gamma\in B_{R}$ to obtain the last equality.
 This shows the the right hand side of \eqref{eq:neighbourhood.containment} shrinks down to $B_R\cdot A$ as $t$ goes to infinity, thus proving the lemma. 
%  If $R\geq 1$, first note that for every fixed $\gamma\in\Gamma$ we have
%  \begin{equation}\label{eq:intersection nbhds}
%   \bigcap_{t\geq1}\gamma\cdot N_R(A;\red{d^t})=\gamma(A).
%  \end{equation}
%  Assume now that $1\leq R<2$ and observe that 
%  \begin{equation}\label{eq:neighbourhood.containment}
%   N_R(A;d_\Gamma^t)\subseteq N_R(A;d^t)\cup N_{R-1}\bigparen{B_1\cdot N_{R-1}(A;d^t)\, ;\, d^t_\Gamma}.
%  \end{equation}
%  For every fixed $t_0>0$, the set $B_1\cdot N_{R-1}(A;d^{t_0})$ is closed. Since $R-1<1$, we deduce as before that
%  \[
%   \bigcap_{t\geq1} N_{R-1}\bigparen{B_1\cdot N_{R-1}(A;d^t)\, ;\, d^t_\Gamma}
%   \subseteq \bigcap_{t>t_0} N_{R-1}\bigparen{B_1\cdot N_{R-1}(A;d^{t_0})\, ;\, d^t}
%   = B_1\cdot N_{R-1}(A;d^{t_0})
%  \]
%  and therefore
%  \[
%   \bigcap_{t\geq1} N_{R-1}\bigparen{B_1\cdot N_{R-1}(A;d^t)\, ;\, d^t_\Gamma}
%   \subseteq \bigcap_{t_0>0} B_1\cdot N_{R-1}(A;d^{t_0})
%   = B_1\cdot A,
%  \]
%  \red{where we used \eqref{eq:intersection nbhds} on the finitely many elements $\gamma\in B_1$ to obtain the last equality.}
%  
%  This shows the the right hand side of \eqref{eq:neighbourhood.containment} shrinks down to $B_1\cdot A$ as $t$ goes to infinity. Thus, we have proved the claim for $R<2$. An obvious inductive generalisation of this argument proves the statement for every fixed $R>0$.
\end{proof}

For more details and elementary facts on the geometry of warped cones, we refer to \cite{roe_warped_2005, SawickiThesis,Vig18,Wil09b}.

\subsection{(Dynamical) quasi-local characterisations for asymptotic expansion}\label{ssec:Warped cones and Druţu--Nowak projections}
In this subsection, we will introduce a notion of dynamical quasi-locality and explain its relation with the ordinary quasi-locality for operators on warped cones. Using the dynamical quasi-locality, we will study the Druţu--Nowak projection associated to a warped cone, and show that the ordinary quasi-locality of this projection characterises asymptotic expansion in measure. 

Let $\rho\colon \Gamma\curvearrowright X$ be a continuous action on a 
metric space $(X,d)$ of diameter at most $2$. Let $\nu$ be a probability measure on $(X,d)$ and $\lambda$ be the Lebesgue measure on $[1,\infty)$. Equip the unified warped cone $\CO_\Gamma X=X\times [1,\infty)$ with the product measure $\nu\times\lambda$. 

For any measurable non-null $Y\subseteq X$, denote by $P_Y \in \B(L^2(X,\nu))$ the \emph{averaging projection on $Y$}, which is the orthogonal projection onto the one-dimensional subspace in $L^2(X,\nu)$ spanned by $\chi_Y$. In other words, 
\[
 P_Y f \coloneqq \langle f, \frac{1}{\nu(Y)}\cdot \chi_Y\rangle \chi_Y,
\]
 where $f\in L^2(X,\nu)$. The \emph{Druţu--Nowak projection} (see \cite[Section~6.c.]{DN17}) is defined as $\fkG= P_X\otimes \Id_{L^2([1,\infty))}\in \B(L^2(\CO_\Gamma X, \nu\times\lambda))$. In other words, it is the orthogonal projection onto $\CCC\otimes L^2([1,\infty),\lambda)$.

\

Recall from \cite{Roe88,Roe96} that an operator $T\in \B(L^2(\CO_\Gamma X, \nu\times\lambda))$ is \emph{quasi-local} if for every $\epsilon>0$, there exists an $R>0$ such that for any two measurable subsets $A,C \subseteq \CO_\Gamma X$ with $d_\Gamma (A,C) >R$ we have $\|\chi_A T \chi_C\| < \epsilon$.
Analogously, a family of operators $\{T_t\}_{t\in [1,\infty)}$ in $\B(L^2(X,\nu))$ is \emph{uniformly quasi-local on $\wc{\Gamma\curvearrowright X}$} if for every $\epsilon>0$ there exists an $R>0$ such that for every $t\in [1,\infty)$ and every pair of measurable subsets $A,C\subseteq X$ with $d_\Gamma^t(A,C)>R$, we have $\|\chi_A T_t \chi_C\| < \epsilon$.

Now we introduce the following dynamical analogue of quasi\=/locality for operators in $\B(L^2(X,\nu))$ where $(X,\nu)$ is a probability space with a $\Gamma$\=/action:

\begin{defn}\label{defn: quasi-local algebra wrt action}
Let $\rho\colon\Gamma \act (X,\nu)$ be an action on a probability space $(X,\nu)$.
An operator $T \in \B(L^2(X,\nu))$ is called \emph{$\rho$\=/quasi\=/local} if for every $\epsilon>0$ there exists a $k \in \N$ such that for any measurable subsets $A,C \subseteq X$ with $\nu((B_k \cdot A)\cap C)=0$, we have $\|\chi_A T \chi_C\|<\epsilon$ (recall that $B_k=\{\gamma\in\Gamma\mid\ell(\gamma)\leq k\}$).
\end{defn}

Similarly to \cite[Lemma~3.8]{intro}, quasi-locality of the averaging projection $P_X$ can be detected by the following calculation:

\begin{lem}\label{lem: averaging proj calculation}
For every measurable subsets $A,C$ in $X$, we have that
\[
\norm{\chi_A P_X \chi_C}_{\B(L^2(X,\nu))}=\sqrt{\nu(A)\nu(C)}.
\]
\end{lem}

\begin{proof}
By direct calculations, we have that
\begin{align*}
\|\chi_AP_X\chi_C\| 
&= \sup_{\|v\|=\|w\|=1} |\langle \chi_AP_X\chi_C v,w \rangle| = \sup_{\|v\|=\|w\|=1} |\langle P_X\chi_C v, P_X \chi_A w \rangle|\\
&= \sup_{\|v\|=\|w\|=1} \big| \big\langle \langle \chi_C v,1 \rangle 1, \langle \chi_A w,1 \rangle 1  \big\rangle \big| 
=  \sup_{\|v\|=\|w\|=1} |\langle v,\chi_C  \rangle \overline{\langle w,\chi_A \rangle} \langle 1,  1 \rangle| \\
& \leq \sqrt{\nu(A)\nu(C)},
\end{align*}
where the last inequality follows from the Cauchy--Schwarz inequality. On the other hand, if we let $v$ and $w$ be the normalised characteristic functions of $C,A$ respectively then we have that $\langle P_X \chi_C v, P_X \chi_A w \rangle = \sqrt{\nu(A)\nu(C)}$.
\end{proof}
The following corollary is a dynamical analogue of \cite[Proposition~3.9]{intro} and it is an immediate consequence of Lemma~\ref{lem: averaging proj calculation}: 

\begin{cor}\label{cor: char for averaging projection being quasi-local}
Let $\rho\colon\Gamma\curvearrowright(X,\nu)$ be an action on a probability space $(X,\nu)$ and $P_X$ be the associated averaging projection on $X$.
Then $P_X$ is $\rho$\=/quasi\=/local \emph{if and only if} 
\[
\lim_{k \to +\infty} \sup\big\{\nu(A)\nu(C)\bigmid A,C \subseteq X\mbox{ measurable with}\  \nu((B_k \cdot A)\cap C)=0 \big\} =0.
\]
\end{cor}

We are now ready to show that asymptotic expansion in measure can be characterised by $\rho$\=/quasi-locality of the associated averaging projections. This is an analogue of \cite[Theorem~3.11]{intro}.

\begin{prop}\label{prop: characterise quasi-locality}
Let $\rho\colon \Gamma\curvearrowright (X,\nu)$ be an action on a probability space $(X,\nu)$ and $P_X$ be the associated averaging projection on $X$. Then $\rho$ is asymptotically expanding \emph{if and only if} $P_X$ is $\rho$\=/quasi\=/local.
\end{prop}

\begin{proof}
\emph{Necessity:} Suppose $P_X$ is not $\rho$\=/quasi\=/local, then by Corollary \ref{cor: char for averaging projection being quasi-local} we have:
\[
\alpha\coloneqq \frac{1}{2}\lim_{k \to +\infty} \sup\big\{\nu(A)\nu(C)\bigmid A,C \subseteq X\mbox{ measurable with~} \nu((B_k \cdot A)\cap C)=0 \big\} >0.
\]
In particular, $\frac{1}{2} \leq 1-\alpha<1$. Thus, we can choose a sequence $(A_n,C_n)_{n\in \N}$, where $A_n, C_n \subseteq X$ are measurable subsets with $\nu((B_n \cdot A_n)\cap C_n)=0$ 
such that $\nu(A_n)\nu(C_n) \geq \alpha$.
Since $\nu(A_n) \leq 1$ and $\nu(C_n) \leq 1$, both $\nu(A_n)$ and $\nu(C_n)$ are at least $\alpha$.  Furthermore, $\nu((B_{n} \cdot A_n)\cap C_n)=0$ implies that $\nu(A_n \cap C_n)=0$. In particular, both $\nu(A_n)$ and $\nu(C_n)$ are not greater than $1-\alpha$ for each $n\in \NN$.

If the action was asymptotically expanding, then Definition~\ref{defn:asymptotic expanding in measure} and Lemma~\ref{lem:annoying 1/2 upper bound} would imply that there exist constants $b>0$ and $h \in \NN$ such that for every measurable subset $A\subseteq X$ with $\alpha\leq \nu(A) \leq 1-\alpha$, we have $\nu(B_{h} \cdot A) > (1+b)\nu(A)$.
Let $k\coloneqq mh$, where $m\coloneqq\lceil \log_{1+b}(\frac{1}{\alpha}-1) \rceil$. Then \emph{either} 
\[
 \nu(B_k \cdot A)> 1-\alpha
\]
\emph{or} we deduce by induction on $m$ that 
\[
\nu(B_k \cdot A) > (1+b)^m\nu(A)\geq 1-\alpha.
\]
Note that $A_n$ satisfies $\alpha \leq \nu(A_n) \leq 1-\alpha$ for all $n\in \NN$. Hence for $n\geq {k}$, we have $\nu(B_{{n}} \cdot A_n) > 1-\alpha$. This is a contradiction to $\nu(C_n) \geq \alpha$ and $\nu((B_{{n}} \cdot A_n)\cap C_n)=0$. 

~\

\emph{Sufficiency:} Assume that $\rho$ is not asymptotically expanding. Then there exists $\alpha_0 \in (0,\frac 12]$ such that for every $n \in \N$ there exists a measurable subset $A_n \subseteq X$ with $\alpha_0 \leq \nu(A_n) \leq\frac{1}{2}$ and $\nu(B_n \cdot A_n) \leq \frac{3}{2}\nu(A_n)$. 
For every $n$ we have
\[
\nu(X \smallsetminus (B_n \cdot A_n)) = 1- \nu(B_n \cdot A_n) \geq 1 - \frac{3}{2}\nu(A_n)\geq \frac{1}{4}.
\]
Hence, we have that
\[
\nu(A_n) \cdot \nu(X \smallsetminus (B_n \cdot A_n)) \geq \frac{\alpha_0}{4}>0,
\]
which implies that the limit
\[
\lim_{n \to +\infty} \sup\big\{\nu(A)\nu(C)\bigmid A,C \subseteq X \mbox{ measurable with~} \nu((B_n \cdot A)\cap C)=0 \big\} \geq \frac{\alpha_0}{4} >0.
\]
Hence, $P_X$ is not $\rho$-quasi-local by Corollary \ref{cor: char for averaging projection being quasi-local}.
\end{proof}

We will now show that the dynamical quasi\=/locality completely determines the ordinary quasi\=/locality for those operators of the (unified) warped cone that arise as transformations of the base space. To be precise, consider the following $*$-homomorphism:
\begin{equation}\label{eq:connecting homomorphism}
\Phi: \B(L^2(X,\nu)) \rightarrow \B(L^2(\CO_\Gamma X, \nu\times\lambda)), \quad T \mapsto T \otimes \mathrm{Id}_{L^2([1,\infty))}
\end{equation}
(note that the Druţu--Nowak projection $\fkG$ equals to $\Phi(P_X)$). We can then prove the following: 

\begin{prop}\label{prop: quasi-locality equiv}
Let $(X,d)$ be a metric space with diameter at most $2$ equipped with a probability measure $\nu$, and $\rho\colon \Gamma\curvearrowright X$ be a continuous action. For any $T \in \B(L^2(X,\nu))$, we consider the following conditions:
\begin{enumerate}
  \item $T$ is $\rho$\=/quasi\=/local;
  \item $\Phi(T)$ is quasi-local;
  \item the family of operators $T_t\equiv T$ for $t\in [1,\infty)$ is uniformly quasi-local on $\wc{\Gamma\curvearrowright X}$.
\end{enumerate}
Then we have (1) $\Rightarrow$ (2) $\Rightarrow$ (3). Furthermore, if $\nu$ is Radon then they are all equivalent.
\end{prop}

\begin{proof}
\emph{(1) $\Rightarrow$ (2):} 
Fix an $\epsilon>0$. Since $T$ is $\rho$\=/quasi\=/local, there exists $k \in \N$ such that for any measurable subsets $A',C' \subseteq X$ with $\nu((B_k \cdot A')\cap C')=0$, we have $\|\chi_{A'} T \chi_{C'}\|<\epsilon$.

Given a pair of measurable subsets $A,C \subseteq \O_\Gamma X = X\times [1,\infty)$ with $d_\Gamma (A,C) > k$, we can write $A=\bigsqcup_{t\in [1,\infty)} A_t\times \braces{t}$ and $C=\bigsqcup_{t\in [1,\infty)} C_t \times \braces{t}$, where $A_t, C_t$ are measurable subsets in $X$. 

For every $(x,t)\in \O_\Gamma X = X\times [1,\infty)$ and every $\gamma\in \Gamma$, we have $d_\Gamma ((\gamma\cdot x,t), (x,t)) \leq \ell(\gamma)$. Since $d_{\Gamma}(A,C)>k$, it follows that $\nu((B_k \cdot A_t) \cap C_t) = 0$ for every $t\in [1,\infty)$. Hence, we conclude that $\|\chi_{A_t} T \chi_{C_t}\|<\epsilon$ for every $t\in [1,\infty)$.

For every $\xi\in L^2(\CO_\Gamma X,\nu\times\lambda)$, we set $\xi_t(x)\coloneqq \xi(x,t)$ so that  $\xi_t\in L^2(X,\nu)$ for almost every $t\in [1,\infty)$. Using Fubini's Theorem, we obtain that 
\begin{align*}%\label{eq:quasi-locality}
\|\chi_A \Phi(T) \chi_C\xi\|^2 
&= \int_{\CO_\Gamma X} \big| \bigparen{\chi_{A}( T\otimes\Id_{L^2([1,\infty))})\chi_{C}\xi }(x,t) \big|^2 \d(\nu\times\lambda)(x,t) \\
&= \int_1^\infty 
  \int_{X} \big|\paren{\chi_{A_t}T\chi_{C_t}\xi_t} (x)\big|^2\d\nu(x)\d t 
  = \int_1^\infty 
  \big\|\chi_{A_t}T\chi_{C_t}\xi_t\big\|^2\d t \\
&\leq \int_1^\infty \epsilon^2\norm{\xi_t}^2\d t
\ = \epsilon^2\norm{\xi}^2.
\end{align*}
It follows that $\Phi(T)$ is quasi-local.

~\

\emph{(2) $\Rightarrow$ (3):}
For any measurable subsets $A,C\subseteq X$, we note that $d_\Gamma^t(A,C)=d_\Gamma(A\times[t,t+1],C\times [t,t+1])$. For every $f\in L^2(X,\nu)$ and $t\in [1,\infty)$, we construct a $F_t\in L^2(\CO_\Gamma X,\nu\times\lambda)$ by letting $F_t(x,s)=f(x)$ if $t\leq s\leq t+1$ and zero otherwise. Note that $\norm{f}_\nu= \norm{F_t}_{\nu\times\lambda}$ and
\(
 \norm{\chi_A T_t \chi_C f}=\norm{\chi_A T \chi_C f}=
 \norm{\chi_{A\times[t,t+1]} \Phi(T)\chi_{C\times[t,t+1]}F_t}
\). Now the rest of the proof is obvious.

~\

\emph{(3) $\Rightarrow$ (1):}
Fix an $\epsilon>0$. Then by the assumption, there exists an $R>0$ such that for every $t\in [1,\infty)$ and measurable subsets $A,C \subseteq X$ with $d_\Gamma^t (A,C) > R$ we have $\norm{\chi_{A} T \chi_{C}} < \epsilon$. We will verify that $\norm{\chi_{A} T \chi_{C}}< \epsilon$ for measurable subsets $A,C \subseteq X$ with $\nu\paren{(B_R\cdot A)\cap C}=0$.

Assume first that $A,C \subseteq X$ are compact subsets such that $(B_R \cdot A)\cap C = \emptyset$. It follows from Lemma~\ref{lem:intersection.of.neighbourhoods} that
\[
 \bigcap_{t\geq1} N_R(A;d_\Gamma^t) \cap C= (B_R\cdot A)\cap C=\emptyset.
\]
Since $C$ is compact and all $N_R(A;d_\Gamma^t)$ are closed, we deduce that $N_R(A;d_\Gamma^{t_0})\cap C=\emptyset$ for some $t_0$ large enough. This means that $d_\Gamma^{t_0}(A,C)>R$ and hence $\norm{\chi_{A} T \chi_{C}} < \epsilon$ by the hypothesis.

For general measurable subsets $A,C \subseteq X$ with $\nu((B_R \cdot A)\cap C)=0$, replacing $C$ by $C \smallsetminus (B_R \cdot A)$ if necessary (which only differ by a null set) we may assume that $(B_R \cdot A)\cap C=\emptyset$. 
Since the measure $\nu$ is Radon and finite, there exist increasing sequences of compact subsets $\{A_n \subseteq A\}_{n\in \NN}$ and $\{C_n \subseteq C\}_{n\in \NN}$ such that $\lim\limits_{n\to \infty}\nu(A \smallsetminus A_n) = 0$ and $\lim\limits_{n\to \infty}\nu(C \smallsetminus C_n) = 0$. \emph{A fortiori}, we have $(B_R \cdot A_n)\cap C_n=\emptyset$ and it follows from the discussion in the second paragraph that $\|\chi_{A_n} T \chi_{C_n}\| < \epsilon$ for all $n\in \NN$. Thus, $\|\chi_A T \chi_C\| = \sup_{n} \|\chi_{A_n} T \chi_{C_n}\| < \epsilon.$
\end{proof}

Combining Proposition~\ref{prop: characterise quasi-locality} with Proposition~\ref{prop: quasi-locality equiv} implies the desired characterisation of asymptotic expansion in measure in terms of quasi\=/locality:

\begin{thm}\label{thm: characterise quasi-locality}
Let $(X,d)$ be a metric space with diameter at most $2$ equipped with a Radon probability measure $\nu$, and $\rho\colon \Gamma\curvearrowright (X,d)$ be a continuous action. If $P_X$ is the associated averaging projection on $X$ and $\fkG= P_X\otimes \Id_{L^2([1,\infty))}$ is the Druţu--Nowak projection, then the following are equivalent:
\begin{enumerate}
  \item $\rho$ is asymptotically expanding;
  \item $P_X$ is $\rho$\=/quasi\=/local;
  \item $\fkG$ is quasi-local;
  \item the family of operators $(P_X)_t\equiv P_X$ for $t\in [1,\infty)$ is uniformly quasi-local on $\wc{\Gamma\curvearrowright X}$.
\end{enumerate}
\end{thm}

\subsection{Projections approximated by finite dynamical propagation operators}
\label{ssec:domains.of.exp.and.projections.as.limits}
In the previous subsection, we showed that asymptotic expansion in measure can be characterised by (dynamical) quasi-locality of certain projections (Theorem \ref{thm: characterise quasi-locality}). In the same spirit of \cite[Section 6]{structure}, we would like to connect (dynamical) quasi-locality with finite (dynamical) propagation operators.

In doing so, we will show that unified warped cones arising from asymptotically expanding actions admit plenty of projections which can be approximated by finite propagation operators. This greatly generalise \cite[Theorem 6.6]{DN17}.
It will turn out that all of these projections lie outside the image of the coarse Baum--Connes assembly map. We will return to these aspects in Section~\ref{sec:baum_connes}.

We once again introduce a dynamical analogue of an analytic property of operators, namely, the dynamical propagation (see Subsection \ref{ssec:finite propagation wcone} for the notion of ordinary finite propagation operators):

\begin{de}
 Let $\rho\colon\Gamma\curvearrowright (X,\nu)$ be an action on a probability space $(X,\nu)$. We say that an operator $T\in \B(L^2(X,\nu))$ has finite \emph{$\rho$\=/propagation} if there is a $k\in\NN$ such that $\chi_AT\chi_C=0$ for any measurable subsets $A,C \subseteq X$ with $\nu((B_k\cdot A)\cap C)=0$. The smallest $k$ satisfying the above condition is called the \emph{$\rho$\=/propagation} of $T$.
\end{de}

Throughout the rest of this subsection, let $\Gamma\curvearrowright (X,\nu)$ be a measure\=/class-preserving action, $Y\subseteq X$ be a domain and $S\subseteq \Gamma$ be a finite symmetric set containing the identity. Recall from Proposition~\ref{prop:normalised.markov.is.reversible} that such an action induces a normalised local Markov kernel $\Pi_{Y,S}$ on $Y$ (Definition \ref{defn:normalized.Markov.kernel}). This kernel is reversible with a reversing measure $\snu$, where $\d\snu=\sigma_{Y,S}\,\d(\nu|_Y)$ for the function $\sigma_{Y,S}$ defined in \eqref{function sigma}.
We denote by $\fkP_{Y,S}\in \B(L^2(Y,\snu))$ and $\Delta_{Y,S}=1-\fkP_{Y,S} \in \B(L^2(Y,\snu))$ the Markov and Laplacian operators associated with $\Pi_{Y,S}$, respectively. 

We will present two different ways to produce projections from operators of finite $\rho$\=/propagation using Markov $S$\=/expansion. One is normalised to better accommodate the associated Markov kernel, while the other is non-normalised and more related to the original averaging projection $P_X$. In either case, our construction relies heavily on the techniques developed in Section~\ref{sec:Markov.expansion}.

\subsubsection{Normalised projections}\label{secc:normailsed projections}
Let $\tilde{P}_{Y,S} \in \B(L^2(Y,\snu))$ be the orthogonal projection onto constant functions on $Y$ (this need not coincide with $P_Y$, as the projection is taken with respect to the inner product $\angles{\cdot,\cdot}_{\snu}$).
Let us consider the isometric embedding 
\[
\widehat I_{Y,S}\colon L^2(Y,\snu)\hookrightarrow  L^2(X,\nu) 
\]
defined by pointwise multiplication by the function $\sqrt{\sigma_{Y,S}}$ on $Y$ and then extending by $0$ on $X \smallsetminus Y$. This induces the following adjoint $\ast$-homomorphism:
\begin{equation}\label{eq:ad}
\widehat{\mathrm{Ad}}\colon \B(L^2(Y,\snu)) \rightarrow \B(L^2(X,\nu)),\quad \mbox{by} \quad T \mapsto \widehat I_{Y,S} \circ T \circ (\widehat I_{Y,S})^*.
\end{equation}
Note that $\widehat I_{Y,S}(1)=\sqrt{\sigma_{Y,S}}$, where $1$ is the constant function $1$ in $L^2(Y,\snu)$ and $\sigma_{Y,S}$ is defined to be $0$ on every $x\in X\smallsetminus Y$. 
It follows that
\[
\widehat P_{Y,S} \coloneqq \widehat{\mathrm{Ad}}(\tilde{P}_{Y,S}) \in \B(L^2(X,\nu))
\] 
is the orthogonal projection onto the $1$-dimensional subspace of $L^2(X,\nu)$ spanned by the vector $\sqrt{\sigma_{Y,S}}$. We also transfer the lazy Markov operator $\fkP_{Y,S} \in \B(L^2(Y,\snu))$ to $\widehat\fkP_{Y,S} \coloneqq \widehat{\mathrm{Ad}}(\fkP_{Y,S}) \in \B(L^2(X,\nu))$. Similarly, $\widehat{\mathrm{Ad}}$ sends the lazy Markov operator $\frac 12 +\frac 12\fkP_{Y,S}$ to $\frac 12\chi_Y +\frac 12\widehat\fkP_{Y,S}  \in \B(L^2(X,\nu))$.

Now the techniques developed in Section~\ref{sec:Markov.expansion} can be used to prove the following:
\begin{prop}\label{prop:projections.normalized limits.of FP operators}
Let $\rho\colon  \Gamma\curvearrowright (X,\nu)$ be a measure\=/class-preserving action and $Y\subseteq X$ be a domain of Markov $S$\=/expansion (Definition~\ref{defn:domain.Markov.expansion}). Then the associated projection $\widehat P_{Y,S}\in\B(L^2(X,\nu))$ is a norm limit of operators $(\frac 12\chi_Y +\frac 12\widehat\fkP_{Y,S})^n$, which all have finite $\rho$\=/propagation.
\end{prop}

\begin{proof}
Since the operator $\frac 12\chi_Y +\frac 12\widehat\fkP_{Y,S}$ has $\rho$\=/propagation at most $\max\braces{\ell(s)\mid s\in S}$, all of its powers $(\frac 12\chi_Y +\frac 12\widehat\fkP_{Y,S})^n$ have finite $\rho$-propagation as well. By Theorem \ref{thm:spectral.characterisation.markov.exp} and Lemma \ref{lem:spetral gap equiv}, the lazy Markov operator $\frac 12 +\frac 12\fkP_{Y,S}$ on $L^2(Y,\snu)$ has spectrum contained in $[-\frac 34, 1-\varepsilon] \cup \{1\}$ for some $\varepsilon>0$. It follows that the sequence $(\frac 12 +\frac 12\fkP_{Y,S})^n$ converges (as $n\to \infty$) in the operator norm to the projection onto the $1$-eigenspace, which is exactly the projection $\tilde{P}_{Y,S}$. Since $\widehat{\mathrm{Ad}}$ is a $*$-homomorphism, $\| (\frac 12\chi_Y +\frac 12\widehat\fkP_{Y,S})^n -  \widehat P_{Y,S}\| \rightarrow 0$ for $n\to \infty.$ This finishes the proof.
\end{proof}

Theorem \ref{thm:structure theorem Markov} shows that strongly ergodic actions provide plenty of domains of Markov expansion. We can hence use Proposition~\ref{prop:projections.normalized limits.of FP operators} as an abundant source of projections which can be approximated by finite $\rho$-propagation operators.

As a corollary to Proposition~\ref{prop:projections.normalized limits.of FP operators} we also recover the following result by Druţu and Nowak:
\begin{cor}[{\cite[Theorem 6.6]{DN17}}\footnote{Strictly speaking, \cite[Theorem 6.6]{DN17} concerns the operator $\fkG= P_X\otimes \Id_{L^2([1,\infty))}$ and also shows that it is ``ghost''. We will recover these facts in Section~\ref{sec:baum_connes}.}]
Let $\rho\colon\Gamma\curvearrowright (X,\nu)$ be a measure-preserving action on a probability space $(X,\nu)$. Suppose that $\rho$ has spectral gap, then the averaging projection $P_X$ is a norm limit of operators with finite $\rho$\=/propagation.
\end{cor}

\begin{proof}
Since the action is measure preserving, we have $\tilde\nu_{X,S}=|S| \cdot\nu$. It follows that $\widehat P_{X,S}= P_X$ for any choice of $S\subseteq \Gamma$. If $\rho$ has spectral gap, then $X$ is a domain of expansion (see e.g. Corollary~\ref{cor:local.spec.gap.iff.domain.of.expansion}). Hence, $X$ is a domain of Markov expansion by Lemma~\ref{lem:domain.of.exp iff Markov.exp} and we can apply Proposition~\ref{prop:projections.normalized limits.of FP operators} to conclude the proof. 
\end{proof}

\subsubsection{Non-normalised projections}\label{secc:non-normalised projections}
Now we move on to the second construction, where we show that the averaging projections $P_Y$ are norm limits of operators with finite $\rho$\=/propagation as well. Unlike the previous construction, these projections will not be limits of powers of a \emph{fixed} Markov operator. Instead, we will apply our structure theory for asymptotically expanding actions to produce appropriate sequences of operators.

We define a different embedding 
\[
I_{Y,S}\colon L^2(Y,\snu)\hookrightarrow  L^2(X,\nu) 
\]
simply by extending each function in $L^2(Y,\snu)$ by $0$ on $X \smallsetminus Y$.
In general, $I_{Y,S}$ is not isometric and may even be unbounded. 

Assume now that there exists $\Theta\geq 1$ such that $1/\Theta\leq r(s,y)\leq\Theta$ for every $y\in Y$ and $s\in S_{Y,y}$. Under this assumption, it is clear that $I_{Y,S}$ is bounded. So it induces the following adjoint map:
\begin{equation*}
\mathrm{Ad}\colon \B(L^2(Y,\snu)) \rightarrow \B(L^2(X,\nu)),\quad \mbox{by} \quad T \mapsto  I_{Y,S} \circ T \circ (I_{Y,S})^*.
\end{equation*}
Note that--- while being a bounded linear map preserving $\ast$-operations---the adjoint map $\mathrm{Ad}$ might not be multiplicative.

As before, let $\tilde{P}_{Y,S} \in \B(L^2(Y,\snu))$ be the orthogonal projection onto constant functions, while $P_Y \in \B(L^2(X,\nu))$ is the orthogonal projection onto the one-dimensional subspace in $L^2(X,\nu)$ spanned by $\chi_Y$. Since $(I_{Y,S})^*(g)=\frac{1}{\sigma_{Y,S}} g|_Y$ for $g \in L^2(X,\nu)$, we have that 
\begin{equation}\label{eq:projection calculation}
\mathrm{Ad}(\tilde{P}_{Y,S}) = \frac{\nu(Y)}{\snu(Y)}P_Y.
\end{equation}

We prove the following:

\begin{lem}\label{lem:projections limits.of FP operators}
Let $\rho\colon \Gamma\curvearrowright (X,\nu)$ be a measure\=/class-preserving action and $Y\subseteq X$ be a domain of Markov $S$\=/expansion. Assume further that there exists $\Theta\geq 1$ such that $1/\Theta\leq r(s,y)\leq\Theta$ for every $y\in Y$ and $s\in S_{Y,y}$. Then the averaging projection $P_Y\in\B(L^2(X,\nu))$ is a norm limit of operators with finite $\rho$\=/propagation.
\end{lem}

\begin{proof}
By Theorem \ref{thm:spectral.characterisation.markov.exp} and Lemma \ref{lem:spetral gap equiv}, the lazy Markov operator $\frac 12+\frac 12\fkP_{Y,S}$ on $L^2(Y,\snu)$ has spectrum contained in $[-\frac 34, 1-\varepsilon] \cup \{1\}$ for some $\varepsilon>0$. Hence, $(\frac 12+\frac 12\fkP_{Y,S})^n$ converges in the operator norm to the projection $\tilde{P}_{Y,S}$ in $\B(L^2(Y,\snu))$ as $n\to \infty$.

Since the embedding $I_{Y,S}$ is bounded, we obtain that
\[
 I_{Y,S} \circ (\frac 12+\frac 12\fkP_{Y,S})^n \circ (I_{Y,S})^*
 = \mathrm{Ad}\bigparen{(\frac 12+\frac 12\fkP_{Y,S})^n}
 \xrightarrow{n\to\infty} \mathrm{Ad}(\tilde{P}_{Y,S}) =  \frac{\nu(Y)}{\snu(Y)}P_Y
%  \quad \mbox{in} \quad \B(L^2(X,\nu)),
\]
where the last equality comes from \eqref{eq:projection calculation}. Since each $I_{Y,S} \circ (\frac 12+\frac 12\fkP_{Y,S})^n \circ (I_{Y,S})^*$ has $\rho$-propagation bounded by $n\cdot\max\braces{\ell(s)\mid s\in S}$, so the conclusion holds.
\end{proof}

Unlike Proposition~\ref{prop:projections.normalized limits.of FP operators}, Lemma~\ref{lem:projections limits.of FP operators} concerns projections that do not depend on the finite symmetric set $S$. This allows us to prove a result for domains of asymptotic expansion as well:

\begin{prop}\label{prop:projections limits.of FP operators}
Let $\rho\colon  \Gamma\curvearrowright (X,\nu)$ be a measure\=/class-preserving action. Then for any domain $Y\subseteq X$ of asymptotic expansion, the averaging projection $P_Y$ is a norm limit of operators with finite $\rho$\=/propagation.
\end{prop}

\begin{proof}
From Proposition~\ref{prop:Markov local version}, it follows that there is an exhaustion $Y_n\nearrow Y$ by domains of Markov $S^{(n)}$\=/expansion such that for every $n\in\NN$ there is a $\Theta_n\geq 1$ such that $1/\Theta_n\leq r(s,y)\leq\Theta_n$ for every $y\in Y_n$ and $s\in S^{(n)}_{Y_n,y}$.
 
Now it follows from Lemma~\ref{lem:projections limits.of FP operators} that each $P_{Y_n}$ is a norm limit of operators with finite $\rho$\=/propagation. Since $Y_n$ increasingly converges to $Y$ in measure and $Y$ has finite measure, then $P_{Y_n}$ converges to $P_Y$ in the operator norm. Hence, a diagonal argument will conclude the proof.
\end{proof}

It follows easily from the definitions that norm limits of operators with finite $\rho$\=/propagation are $\rho$\=/quasi\=/local. Hence, combining Proposition~\ref{prop: characterise quasi-locality} with Proposition~\ref{prop:projections limits.of FP operators} we immediately obtain the following:

\begin{cor}\label{cor: characterise finite ppg}
Let $\rho\colon \Gamma\curvearrowright (X,\nu)$ be a measure-class-preserving action on a probability space $(X,\nu)$. Then $\rho$ is asymptotically expanding \emph{if and only if} $P_X$ is a norm limit of operators with finite $\rho$\=/propagation.
\end{cor}

\subsection{Characterising asymptotic expansion by finite propagation approximations}
\label{ssec:finite propagation wcone}
Finally, we conclude this section by combining results in Subsections~\ref{ssec:Warped cones and Druţu--Nowak projections} and \ref{ssec:domains.of.exp.and.projections.as.limits} to prove that an action is asymptotically expanding \emph{if and only if} the Druţu--Nowak projection can be approximated by operators with finite propagation.

Let $(X,d)$ be a metric space  of diameter at most $2$, $\rho\colon \Gamma\curvearrowright X$ be a continuous action and $\CO_\Gamma X$ the associated unified warped cone. 
If $X$ is equipped with a probability measure $\nu$, we give $\CO_\Gamma X$ the product measure $\nu\times\lambda$ and say that an operator $T\in \B(L^2(\CO_\Gamma X, \nu\times\lambda))$ has \emph{finite propagation} if there exists an $R>0$ such that for any two measurable subsets $A,C \subseteq \CO_\Gamma X$ with $d_\Gamma (A,C)>R$, we have $\chi_A T \chi_C=0$.

\begin{prop}\label{prop: finite ppg equiv}
Let $(X,d)$ be a metric space with diameter at most $2$ equipped with a probability measure $\nu$, and $\rho\colon \Gamma\curvearrowright X$ be a continuous action. If $T \in \B(L^2(X,\nu))$ has finite $\rho$\=/propagation, then $\Phi(T)$ has finite propagation. If in addition $\nu$ is Radon, the converse implication holds as well.
\end{prop}
\begin{proof}
 The argument is identical to that of Proposition~\ref{prop: quasi-locality equiv} with $\epsilon=0$.
\end{proof}

Since norm limits of operators with finite propagation are quasi-local, 
we can combine Proposition \ref{prop: finite ppg equiv} and Corollary \ref{cor: characterise finite ppg} with Theorem~\ref{thm: characterise quasi-locality} $(3)\Rightarrow (1)$ to obtain a dynamical counterpart of \cite[Theorem~C]{structure}: 

\begin{thm}\label{thm: characterise finite ppg}
Let $(X,d)$ be a metric space with diameter at most $2$ equipped with a probability measure $\nu$, and $\rho\colon \Gamma\curvearrowright X$ be a continuous measure\=/class-preserving action. The following are equivalent:
\begin{enumerate}
  \item $\rho$ is asymptotically expanding;
  \item the averaging projection $P_X$ is a norm limit of operators with finite $\rho$\=/propagation;
  \item the Druţu--Nowak projection $\fkG$ is a norm limit of operators with finite propagation.
\end{enumerate}
\end{thm}

For later use, we record that we can apply Proposition \ref{prop: finite ppg equiv} to the projections constructed in Proposition \ref{prop:projections.normalized limits.of FP operators} and Proposition \ref{prop:projections limits.of FP operators} and obtain the following:

\begin{cor}\label{cor:projection finite ppg}
Let $(X,d)$ be a metric space with diameter at most $2$ equipped with a probability measure $\nu$, and $\rho\colon \Gamma\curvearrowright X$ be a continuous and measure\=/class-preserving action. Let $P\in \B(L^2(X,\nu))$ be one of the following rank\=/one projection:
\begin{enumerate}
  \item $P=\widehat P_{Y,S}$ for a domain $Y\subseteq X$ of Markov $S$\=/expansion;
  \item $P=P_Y$ for a domain $Y\subseteq X$ of asymptotic expansion.
\end{enumerate}
Then the projection $\Phi(P)=P \otimes \mathrm{Id}_{L^2([1,\infty))} \in \B(L^2(\CO_\Gamma X, \nu\times\lambda))$ is a norm limit of operators with finite propagation.
\end{cor}

\section{The coarse Baum--Connes Conjecture}
\label{sec:baum_connes}  
In this section, we will use the projections constructed in Section \ref{ssec:domains.of.exp.and.projections.as.limits} to provide new counterexamples to the coarse Baum--Connes conjecture. These arise from certain warped cones associated with asymptotically expanding actions. We will follow the outline of \cite[Section 3]{sawicki_warped_2017} (the origin of this method goes back to \cite{higson1999counterexamples} and \cite{MR2871145}).

Throughout this section, $(X,d)$ will be a compact metric space with diameter at most $2$ endowed with a non-atomic probability measure $\nu$ of full support (\emph{i.e.}, every singleton has measure zero and every open set has positive measure). As usual, $\Gamma$ is a countable discrete group with a proper length function $\ell$. Furthermore, $\Gamma\curvearrowright(X,d,\nu)$ will be a continuous measure-class-preserving action.

\subsection{Roe algebras and projections}\label{ssec:roe algebras}
Let us begin by recalling some basic notions concerning Roe algebras. 

Let $(Y,d)$ be any proper metric space. In particular, $Y$ is locally compact and $\sigma$-compact. Let $C_0(Y)$ be the $C^*$-algebra of continuous functions on $Y$ vanishing at infinity. A non-degenerate $*$\=/representation $C_0(Y) \rightarrow \B(\CH)$ on some separable Hilbert space $\CH$ is called \emph{ample} if no non-zero element of $C_0(Y)$ acts as a compact operator on $\H$. An operator $a\in \B(\CH)$ has \emph{finite propagation} if there is $r>0$ such that $fag=0$ whenever $f,g\in C_0(Y)$ satisfy $d(\text{supp}(f),\text{supp}(g))>r$.\footnote{%
It is easy to check that for a proper metric space $(X,d)$, $T\in \B(L^2(X, \nu))$ has finite propagation \emph{if and only if} there exists an $R>0$ such that $\chi_AT\chi_C=0$ whenever $A,C\subseteq X$ are measurable subsets with $d(A,C)>R$. In particular, this definition is equivalent to the one given in Subsection~\ref{ssec:finite propagation wcone}.
}
Moreover, an operator $a\in \B(\CH)$ is called \emph{locally compact} if $fa$ and $af$ are compact for all $f\in C_0(Y)$.

The \emph{algebraic Roe algebra $\mathbb{C}[Y]$} of $Y$ is the $*$\=/algebra of locally compact finite propagation operators in $\B(\CH)$, and the \emph{Roe algebra $C^*(Y)$} of $Y$ is the norm\=/closure of $\mathbb{C}[Y]$ in $\B(\CH)$. Note that the Roe algebra $C^*(Y)$ does not depend on the choice of the non-degenerate ample $*$-representation of $C_0(Y)$, but only up to non-canonical $*$\=/isomorphism (see \emph{e.g.} \cite[Remark~5.1.13]{WY20}). On the other hand, the $K$-theory groups $K_*(C^*(Y))$ do not depend on the choice of such representations up to \emph{canonical} $*$-isomorphism (see \emph{e.g.} \cite[Theorem~5.1.15]{WY20}). It is well-known that the isomorphism class of $C^*(Y)$ is a coarse invariant for the metric space $Y$.

Let now $(X,d,\nu)$ be a metric measure space as outlined at the beginning of Section~\ref{sec:baum_connes}. Since $\nu$ has full support and is non-atomic, the multiplication representation of $C(X)$ on $L^2(X,\nu)$ is non-degenerate and ample. 
Hence the multiplication representation of $C_0(\mathcal{O}_\Gamma X)$ on $L^2(X\times [1,\infty), \nu \times \lambda)$ is also non-degenerate and ample. We can thus
use it to form the Roe algebra $C^*(\mathcal{O}_\Gamma X)$.

As explained by Sawicki in \cite[Proposition 1.1]{sawicki_warped_2017}, the original Druţu--Nowak projection $\fkG \in \B(L^2(\mathcal{O}_\Gamma X,\nu \times \lambda))$ is \emph{not} locally compact because its image contains a copy of $L^2([1,\infty),\lambda)$. In particular, $\fkG$ cannot belong to the Roe algebra. One way to overcome this issue is to consider the subspace $(X \times \NN,d_\Gamma)$ of the unified warped cone $\mathcal{O}_\Gamma X$ instead. We will call this
the \emph{integral warped cone}. Since the embedding $(X \times \NN, d_\Gamma) \hookrightarrow (\mathcal{O}_\Gamma X, d_\Gamma)$ is a quasi-isometry, their Roe algebras are isomorphic. We will hence abuse the notation and denote also the integral warped cone by $\mathcal{O}_\Gamma X$.

Similarly, we also define the following integral analogue of the $*$-homomorphism $\Phi$ defined in \eqref{eq:connecting homomorphism}  (still denoted by $\Phi$):
\[
\Phi\colon \B(L^2(X,\nu)) \rightarrow \B(L^2(\CO_\Gamma X, \nu\times \lambda_\NN)), \quad T \mapsto T \otimes \mathrm{Id}_{\ell^2(\NN)},
\]
where $\lambda_\NN$ denotes the counting measure on $\NN$. It is elementary to check that Theorem~\ref{thm: characterise quasi-locality} and Proposition \ref{prop: finite ppg equiv} still hold in the integral setting. It follows that the integral analogues of Theorem \ref{thm: characterise finite ppg} and Corollary \ref{cor:projection finite ppg} hold true as well. We will henceforth use their integral versions without further notice.

Let us now focus on the projections considered in Corollary \ref{cor:projection finite ppg}. More precisely, we denote by $\mathcal{P}$ the set of rank one projections in $\B(L^2(X,\nu))$ as follows:
\begin{align*}
P\in \mathcal{P} \quad \Leftrightarrow \quad &\mbox{either}& &P=\widehat P_{Y,S} \mbox{~for~a~domain~} Y\subseteq X \mbox{~of~Markov~} S\=/\mbox{expansion}\\
&\mbox{or}& &P=P_Y \mbox{~for~a~domain~} Y\subseteq X \mbox{~of~asymptotic~expansion}.
\end{align*}
For the averaging projection $P_X$, the associated projection $\Phi(P_X) = P_X \otimes \mathrm{Id}_{\ell^2(\NN)}$ (still denoted by $\fkG$) is called the \emph{integral Druţu--Nowak projection} (see \cite[Proposition~1.3]{sawicki_warped_2017}). It follows from Corollary \ref{cor:projection finite ppg} that the projection $\Phi(P)$ can be approximated by finite propagation operators for every $P \in \mathcal{P}$. Actually, we can even show the following stronger statement:

\begin{prop}\label{prop:projection in Roe}
For every $P \in \mathcal{P}$, the projection $\Phi(P)$ is non-compact and belongs to the Roe algebra $C^*(\mathcal{O}_\Gamma X)$ of the integral warped cone $\mathcal{O}_\Gamma X$. In particular, when the action is asymptotically expanding the integral Druţu--Nowak projection $\fkG$ belongs to $C^*(\mathcal{O}_\Gamma X)$.
\end{prop} 
\begin{proof}
Clearly, each $\Phi(P)=P \otimes \mathrm{Id}_{\ell^2(\NN)}$ is non-compact for $P \in \mathcal{P}$. We only show that $\Phi(P)$ belongs to $C^*(\mathcal{O}_\Gamma X)$ when $P=\widehat P_{Y,S}$ for a domain $Y \subseteq X$ of Markov $S$-expansion, as the other case is similar and almost identical to the proof of \cite[Proposition~1.3]{sawicki_warped_2017}. Recall that $\widehat P_{Y,S}$ is the orthogonal projection onto the one-dimensional subspace of $L^2(X,\nu)$ spanned by the vector $\sqrt{\sigma_{Y,S}}$ defined in \eqref{function sigma}. 
 
Since $X$ is compact, there exists a Borel partition $\mathcal{V}=\{V_i\mid i\in I\}$ of $\mathcal{O}_\Gamma X$ such that each $V_i$ has diameter at most 1 and is contained in some level set $X \times \{n\}$, and for each $n\in \NN$ only finitely many $V_i$ are contained in $X \times \{n\}$. For each $i\in I$, we write $V_i=U_i \times \{n(i)\}$ for Borel $U_i \subseteq X$ and $n(i)\in \NN$.
We consider the closed subspace $W \subseteq L^2(\mathcal{O}_\Gamma X, \nu \times \lambda_\NN)$ spanned by 
\[
\big\{(\chi_{U_i} \cdot \sqrt{\sigma_{Y,S}}~) \otimes \chi_{\{n(i)\}}\bigmid i\in I \big\}.
\]
Let $R \in \B(L^2(\mathcal{O}_\Gamma X, \nu \times \lambda_\NN))$ be the orthogonal projection onto $W$. It is clear that $\Phi(P)$ is a subprojection of $R$, so $\Phi(P) = R \circ \Phi(P) \circ R$. Moreover, the projection $R$ has propagation at most one.

By Corollary \ref{cor:projection finite ppg}, $\Phi(P)$ is a norm limit of finite propagation operators $T_n \in \B(L^2(\mathcal{O}_\Gamma X, \nu \times \lambda_\NN))$. In particular, we have $\Phi(P)=\lim_{n\to \infty} RT_nR$. Since each $RT_nR$ has finite propagation, it suffices to show that it is also locally compact. If $\phi \in C_0(\mathcal{O}_\Gamma X)$ is a function of compact support, then its range is contained in $L^2(X \times \{1,2,\ldots, N_0\} )$ for some $N_0\in \NN$. This implies that $R\phi$ is of finite rank. Since the set of compact operators is norm-closed, we have that $R \psi$ is compact for every $\psi\in C_0(\mathcal{O}_\Gamma X)$. Since $R$ is self-adjoint, $\psi R$ is compact as well. Hence, we conclude that both $RT_nR\psi$ and $\psi RT_nR$ are compact for every $\psi\in C_0(\mathcal{O}_\Gamma X)$, as desired.
\end{proof}

Proposition~\ref{prop:projection in Roe} allows us to use Theorem~\ref{thm: characterise quasi-locality} to deduce the main theorem of this subsection. Namely, the following dynamical version of \cite[Theorem~C]{structure}:  

\begin{thm}\label{thm:char for asymp. expansion via Roe}
Let $(X,d)$ be a compact metric space with diameter at most $2$ equipped with a non-atomic Radon probability measure $\nu$ of full support, and $\rho\colon \Gamma\curvearrowright (X,d,\nu)$ be a continuous and measure-class-preserving action.

If $P_X$ is the associated averaging projection on $X$ and $\fkG= P_X\otimes \Id_{\ell^2(\NN)}$ is the integral Druţu--Nowak projection, then the following are equivalent:
\begin{enumerate}
  \item $\rho$ is asymptotically expanding;
  \item $P_X$ is $\rho$\=/quasi\=/local;
  \item $\fkG$ is quasi-local;
  \item $\fkG$ belongs to the Roe algebra $C^*(\mathcal{O}_\Gamma X)$ of the integral warped cone $\mathcal{O}_\Gamma X$.
\end{enumerate}
\end{thm}

Another important feature of the projections $\Phi(P)$ for $P\in \mathcal{P}$ is that they are \emph{ghost operators}. This notion was originally introduced by Yu (unpublished) in his study of the coarse Baum--Connes conjecture. We will use the following:

\begin{defn}[{\cite[Definition~6.5]{DN17}}]
Given a metric measure space $(Z,d,\nu)$, an operator $T \in \B(L^2(Z,\nu))$ is called \emph{ghost} if for every $R, \epsilon>0$, there exists a bounded subset $C \subseteq Z$ such that for any $\phi \in L^2(Z,\nu)$ with $\|\phi\|=1$ and $\supp (\phi) \subseteq B_R(x;d)$ for some $x\in Z \smallsetminus C$  we have $\|T \phi\| \leq \epsilon$.
\end{defn}

Firstly, we observe the following easy fact:
\begin{lem}\label{lem:upper uniform}
A non\=/atomic probability measure $\nu$ on a metric space $(Z,d)$ is necessarily upper uniform (\cite[Definition~6.1]{DN17}) in the sense that $ \lim_{r\to 0}\sup_{z\in Z} \nu(B_r(z;d))=0$.
\end{lem}

\begin{proof}
If there exist an $\epsilon>0$ and a sequence $z_n\in Z$ such that $\nu(B_{1/n}(z_n;d))\geq\epsilon>0$ for every $n\in\NN$, then there must be some point $\bar{z}\in Z$ that belongs to $B_{1/n}(z_n;d)$ for infinitely many $n$. To see this, it is sufficient to note that $\nu(\bigcap_{N\in \NN}\bigcup_{n>N}B_{1/n}(z_n;d))=\lim_{N\to \infty}\nu(\bigcup_{n>N}B_{1/n}(z_n;d))\geq\epsilon$ as the probability measure $\nu$ is continuous from above. On the other hand, such a $\bar z$ must be an atom for $\nu$ so that $\nu$ cannot be non-atomic.
\end{proof}

The following lemma is a generalisation of \cite[Theorem 6.6]{DN17}:

\begin{lem}\label{lem:DN projection is ghost}
If $T\in\B(L^2(X,\nu))$ is any orthogonal rank one projection, then $\Phi(T)\in\B(L^2(\CO_\Gamma X, \nu\times\lambda_\NN))$ is ghost. In particular, $\Phi(P)$ is a ghost projection for every $P \in \mathcal{P}$.
\end{lem}
\begin{proof} 
 
Since $X$ is compact, the action $\Gamma\curvearrowright X$ is uniformly continuous. Then the proof of Lemma~\ref{lem:intersection.of.neighbourhoods} can be adapted to show that the balls $B_R(x;d^{n}_\Gamma)$ are contained in $N_{\delta_n}(B_{\lfloor R\rfloor}\cdot x ; d)\subseteq X$ for some positive $\delta_n$ independent of $x$ and such that $\delta_n\to 0$, where $B_{\lfloor R\rfloor}$ denotes the ball in $\Gamma$. Since $N_{\delta_n}(B_{\lfloor R\rfloor}\cdot x ; d)\subseteq \bigcup_{\gamma\in B_{\lfloor R\rfloor}}B_{\delta_n}(\gamma\cdot x;d)$ and $B_{\lfloor R\rfloor}$ is finite, we easily deduce from the upper uniformity of $\nu$ (Lemma~\ref{lem:upper uniform}) that $\lim_{n\to \infty}\sup_{x\in X}\nu(B_R(x;d_\Gamma^{n}))=0$ (see also \cite[Lemma 6.3]{DN17}).

Let $\epsilon,R>0$ be fixed and let $C_{N}\coloneqq \mathcal{O}_\Gamma X\cap \big(X\times [1,N)\big)$ for $N\in \NN$. We note that $C_N$ is a bounded subset of $\mathcal{O}_\Gamma X$ and any point in $\mathcal{O}_\Gamma X\smallsetminus C_N$ is of the form $(x,n)$ for some $n\geq N$ and $x\in X$. It is well-known that every rank one projection $T\in\B(L^2(X,\nu))$ is of the form $T\eta=\angles{\eta,\xi}\xi$ for some unit vector $\xi\in L^2(X,\nu)$. In order to show that $\Phi(T)\in\B(L^2(\CO_\Gamma X, \nu\times\lambda_\NN))$ is ghost, we fix any $\phi\in L^2(\mathcal{O}_\Gamma X, \nu\times\lambda_\NN)$ with $\|\phi\|=1$ and $\supp (\phi) \subseteq B_R((x,n);d_\Gamma)$ for some $(x,n)\in \mathcal{O}_\Gamma X \smallsetminus C_N$. So we have that

 \begin{align*}
  \norm{\Phi(T)(\phi)}^2 &= \sum_{m\in \NN} \norm{\xi}^2\cdot\abs{\int_X \phi(y,m)\overline{\xi}(y)\d\nu(y)}^2 \\
   &\leq \sum_{m=n-R}^{n+R} \big(\int_{B_R(x;d^m_\Gamma)} \abs{\phi(y,m)\overline{\xi}(y)}\d\nu(y) \big)^2\\
   &\leq \sum_{m=n-R}^{n+R} \big(\int_X |\phi(y,m)|^2\d\nu(y) \cdot \int_{B_R(x;d^m_\Gamma)} |\xi(y)|^2\d\nu(y) \big) \\
   &\leq \sum_{m=n-R}^{n+R}  \int_{B_R(x;d^m_\Gamma)} |\xi(y)|^2\d\nu(y),
 \end{align*}
where the last inequality uses the fact that $\int_X |\phi(y,m)|^2\d\nu(y)\leq \norm{\phi}^2=1$ for every $m\in \NN$. Since $\xi\in L^2(X,\nu)$ and $\sup_{x\in X}\nu(B_R(x;d^{n}_\Gamma))\to 0$ as $n\to\infty$, it follows that $\norm{\Phi(T)(\phi)}^2 \to 0$ for $n\to \infty$. We can hence choose $N$ large enough so that $\norm{\Phi(T)(\phi)}\leq \epsilon$ for every $\phi$ with $\supp (\phi) \subseteq B_R((x,n);d_\Gamma)$ for some $(x,n)\in \mathcal{O}_\Gamma X \smallsetminus C_N$, as desired.
\end{proof}

Combining Proposition \ref{prop:projection in Roe} with Lemma \ref{lem:DN projection is ghost}, we obtain the following:

\begin{cor}\label{cor:ghost.projection.in.Roe}
Let $(X,d)$ be a compact metric space with diameter at most $2$ endowed with a non-atomic probability measure $\nu$ of full support, and $\Gamma\curvearrowright(X,d,\nu)$ a measure-class-preserving continuous action. Then each $\Phi(P) \in \B(L^2(\mathcal{O}_\Gamma X, \nu\times\lambda_\NN))$ for $P \in \mathcal{P}$ is a non-compact ghost projection in the Roe algebra $C^*(\CO_\Gamma X)$ of the integral warped cone $\CO_\Gamma X$.
\end{cor}

\subsection{Counterexamples to the coarse Baum--Connes conjecture}\label{ssec:ceg to cBc}
In this subsection, we will consider the subset $\mathcal{Q} X \coloneqq X \times \{2^n\mid n\in \NN\}$ of $X\times[1,\infty)$ and the associated subspace $\mathcal{Q}_\Gamma X$ (which we will call \emph{sparse warped cone}) of the unified warped cone $\mathcal{O}_\Gamma X$. The main goal is to  show that under certain mild assumptions all non-compact ghost projections in the Roe algebra $C^*(\mathcal{Q}_\Gamma X)$ lie outside the image of the coarse Baum--Connes assembly map. In particular, they all violate the coarse Baum--Connes conjecture. 

As before, we define a $*$-homomorphism $\Phi_\mathcal{Q}$ as follows:
\[
\Phi_\mathcal{Q}\colon \B(L^2(X,\nu)) \rightarrow \B(L^2(\CQ_\Gamma X, \nu\times\lambda_\NN)), \quad T \mapsto T \otimes \mathrm{Id}_{\ell^2(\{2^n\mid n\in \NN\})}.
\]
It is easy to see that Corollary \ref{cor:ghost.projection.in.Roe} still holds in this setting: under the same assumption, each $\Phi_\mathcal{Q}(P)$ with $P \in \mathcal{P}$ is a non-compact ghost projection in the Roe algebra $C^*(\CQ_\Gamma X)$ . We call $\fkG_\mathcal{Q}=\Phi_{\mathcal{Q}}(P_X)$ the \emph{sparse Druţu--Nowak projection}.

The idea of the proof is to construct two ``trace'' maps $\tau_\mathrm{d}$ and $\tau^\mathrm{u}$ on $K_0(C^*(\mathcal{Q}_\Gamma X))$, whose restrictions to the image of the coarse assembly map coincide and yet take different values on every non-compact ghost projection in $C^*(\mathcal{Q}_\Gamma X)$. 
The following argument is a combination of those in \cite{higson1999counterexamples, sawicki_warped_2017, MR2871145}. We have decided to provide here a fair amount of details, because it also requires a few (minor) adaptations and extensions.

\begin{rmk}
 The choice of $2^n$ in the definition of $\CQ_\Gamma X$ is rather arbitrary and made for the sake of concreteness. We could equally set $\mathcal{Q} X = X \times \{a_n\mid n\in \NN\}$ for any other sequence $\{a_n\}_{n\in \NN}\subseteq [1,\infty)$ as long as $\lim_{n,m \to \infty}|a_n - a_m|=\infty$.
\end{rmk}

\subsubsection{The trace $\tau_\mathrm{d}$} 

For each $n\in \NN$, we denote by $Q_n \in \B(L^2(\mathcal{Q}_\Gamma X,\nu \times \lambda_\NN))$ the orthogonal projection onto $L^2(X\times \{2^n\},\nu)$.
For $T \in \B(L^2(\mathcal{Q}_\Gamma X,\nu \times \lambda_\NN))$ with propagation at most $2^{n-1}$, we have $Q_n T = T Q_n$ and define $T_n \coloneqq Q_n T Q_n \in C^*(X\times \{2^n\} )$. Hence, the map
\[
\C[\mathcal{Q}_\Gamma X]
\ni T \mapsto (T_n)_{n\in\NN} \in \frac{\prod_n C^*(X\times\{2^n\} )}{\bigoplus_n C^*(X\times\{2^n\} )}
\]
is multiplicative, contractive and $\ast$-preserving on the algebraic Roe algebra $\C[\mathcal{Q}_\Gamma X]$. Thus, it yields a $\ast$\=/homomorphism on the entire Roe algebra $C^*(\mathcal{Q}_\Gamma X)$.

As each $X\times\{2^n\} $ is compact, $C^*(X\times\{2^n\} )$ is $*$-isomorphic to the $C^*$-algebra of compact operators $\K(L^2(X\times\{2^n\} ))$. Hence, the canonical trace map $\mathrm{Tr}$ on $\K(L^2(X\times\{2^n\} ))$ induces $\mathrm{Tr}_*: K_0(C^*(X\times\{2^n\} )) \to \mathbb{Z}$. 
As in \cite[Section 6]{MR2871145} and \cite[Section 3]{sawicki_warped_2017}, we define the trace map 
\[
\tau_\mathrm{d}: K_0(C^*(\mathcal{Q}_\Gamma X)) \rightarrow \frac{\prod \RR}{\bigoplus \RR}
\] 
as the composition of the trace $\mathrm{Tr}_*: K_0(C^*(X\times\{2^n\} )) \to \mathbb{Z} \subseteq \RR$ with the map 
\[
K_0(C^*(\mathcal{Q}_\Gamma X)) \rightarrow K_0 \bigg( \frac{\prod_n C^*(X\times\{2^n\} )}{\bigoplus_n C^*(X\times\{2^n\} )} \bigg)
\]
induced by $T \mapsto (T_n)_{n\in\NN}$ under the identification
\[
K_0 \bigg( \frac{\prod_n C^*(X\times\{2^n\} )}{\bigoplus_n C^*(X\times\{2^n\} )} \bigg) \cong \frac{K_0(\prod_n C^*(X\times\{2^n\} ))}{K_0(\bigoplus_n C^*(X\times\{2^n\} ))} \cong \frac{\prod_n K_0(C^*(X\times\{2^n\} ))}{\bigoplus_n K_0(C^*(X\times\{2^n\} ))}.
\]

The proof of the following lemma is almost identical to the proof of \cite[Theorem 6.1]{MR2871145}, we include here a short proof for the convenience of the reader.  
\begin{lem}\label{lem:trace d}
Let $p\in C^*(\mathcal{Q}_\Gamma X)$ be any projection, then $\tau_\mathrm{d}([p])=0$ \emph{if and only if} $p$ is compact. In particular, we have $\tau_\mathrm{d}([\Phi_\mathcal{Q}(P)]) \neq 0$ for every $P \in \mathcal{P}$.
\end{lem}

\begin{proof}
Firstly, we note that for every $T \in C^*(\mathcal{Q}_\Gamma X)$, we have $[T,Q_n] \to 0$ as $n\to \infty$.
In particular, for every projection $p\in C^*(\mathcal{Q}_\Gamma X)$ we have that $Q_n p Q_n$ gets arbitrarily close to some honest projections $q_n$ in $C^*(X\times\{2^n\})$ as $n\to \infty$. In other words, $[(Q_n p Q_n)_{n\in\NN}]=[(q_n)_{n\in\NN}]$ in $\prod_n C^*(X\times\{2^n\} )/\bigoplus_n C^*(X\times\{2^n\} )$.

By the definition of $\tau_\mathrm{d}$, we have that 
\[
\tau_\mathrm{d}([p]) = [(\mathrm{Tr}(q_1),\mathrm{Tr}(q_2),\ldots)] = [(\mathrm{dim}(q_1),\mathrm{dim}(q_2),\ldots)],
\]
where $\mathrm{dim}(q_n)$ denotes the dimension of the range of $q_n$. On the other hand, as $\|Q_npQ_n - q_n \|\to 0$ it follows that the projection $p$ is compact if and only if $\mathrm{dim}(q_n)=0$ for all but finitely many $n$. So we conclude that $\tau_\mathrm{d}([p])=0$ if and only if $p$ is compact.
\end{proof}

\subsubsection{The trace $\tau^\mathrm{u}$}
In order to construct the other trace map $\tau^\mathrm{u}$, we need some extra assumptions and preliminaries. 

Following \cite{sawicki_straightening_2017}, we equip $\CQ X=X\times\{2^n\mid n\in\NN\}$ with the open cone metric
\[
 d_{\CQ}((x_1,t_1),(x_2,t_2))\coloneqq  \paren{{t_1\wedge t_2}}\cdot d(x_1,x_2) + \abs{t_1-t_2}
\]
so that $\CQ X$ and $\CQ_\Gamma X$ coincide as sets but are equipped with different metrics. We can then define a metric $d_{\Gamma\times\CQ}$ on the product $\Gamma \times \mathcal{Q} X$ as the largest metric such that 
\begin{itemize}
 \item $d_{\Gamma\times\CQ}\paren{(\gamma,(x_1,t_1)),(\gamma,(x_2,t_2))} \leq d_\CQ((x_1,t_1),(x_2,t_2))$;
 \item $d_{\Gamma\times\CQ}\paren{(\gamma,(x,t)), (\eta\gamma, \eta\cdot (x,t)} \leq \ell(\eta)$
\end{itemize}
for every $\gamma,\eta\in\Gamma$ and $(x_1,t_1),(x_2,t_2)\in\CQ X$.\footnote{We remark that the metric $d_{\Gamma\times\CQ}$ is denoted by $d_1'$ in \cite{sawicki_warped_2017} and it is isometric---but not equal---to the metric $d^1$ used in \cite[Definition~3.6]{sawicki_straightening_2017}. The latter is also denoted by $d_1$ in \cite{sawicki_warped_2017}.}

The projection to the second coordinate gives a natural quotient map $\pi\colon \Gamma \times \mathcal{Q} X \to \mathcal{Q}_\Gamma X$ and the metric $d_{\Gamma\times\CQ}$ is defined so that the quotient metric on $\mathcal{Q}_\Gamma X$ coincides with the warped metric $d_\Gamma$. Since $X$ is compact, it is shown in \cite[Proposition 3.10]{sawicki_straightening_2017} that the action on $X$ is free \emph{if and only if} $\pi$ is asymptotically faithful. Recall that a surjective map between metric spaces $\pi\colon(Y,d_Y)\to(Z,d_Z)$ is called \emph{asymptotically faithful} if for every $R>0$ there is a bounded subset $C_R\subseteq Z$ such that the restriction of $\pi$ to every $R$\=/ball centred at a point outside of $\pi^{-1}(C_R)$ is an isometry \cite{sawicki_straightening_2017,MR2871145}. Asymptotic faithfulness will play an important role later on, we thus need to restrict our attention to free actions.

In order to estimate operator norms of finite propagation operators in $\B( L^2(\Gamma\times \CQ X))$, we assume that the metric space $(\Gamma\times\CQ X,d_{\Gamma\times \CQ })$ has the \emph{operator norm localisation property} (ONL) (see \cite{chen2008metric}). Namely, if we equip $\Gamma\times \CQ X$ with the product measure $\lambda_\Gamma\times\nu\times\lambda_\NN$ (here $\lambda_\Gamma$ is the counting measure on the discrete group $\Gamma$), we say that $(\Gamma\times\CQ X,d_{\Gamma\times \CQ})$ has ONL if for every $c\in (0,1)$ and $r>0$ there exists an $R>0$ so that for any operator $T \in \B( L^2(\Gamma\times \CQ X))$ of propagation at most $r$ there exists a unit vector $\xi \in L^2(\Gamma\times \CQ X)$ with $\diam (\supp \xi) \leq R$ satisfying $\|T \xi\| \geq c\|T\|$. 

\begin{rmk}
 It follows from \cite[Proposition~2.4]{chen2008metric} that the above definition of ONL is equivalent to the original definition in \cite[Definition~2.3]{chen2008metric}. It follows from the work of Sako \cite{Sako14} that---for metric spaces that are proper and have bounded geometry--- ONL is also equivalent to property A in the sense of \cite[Definition~2.1]{roe_warped_2005} (see \cite[Corollary~2.5]{sawicki_warped_2017} for a proof).
\end{rmk}

 \begin{rem}\label{rmk:ONL of product cone}
 For a Lipschitz action $\Gamma\curvearrowright X$ on a compact space $X$, the metric space $(\Gamma\times\CQ X,d_{\Gamma\times \CQ})$ has ONL under either of the following conditions:
 \begin{itemize}
 \item[(1)] if $\Gamma$ has property $A$ and $X$ is a manifold;
 \item[(2)] if the asymptotic dimension of $\Gamma$ is finite and $X$ is an ultrametric space.
 \end{itemize} 
 We refer to \cite[Corollary~2.11]{sawicki_warped_2017} for a more general statement.
 \end{rem}

As in \cite[Section~3.2]{sawicki_warped_2017}, let $\rho\colon \Gamma \act X$ be a free action so that $\pi\colon \Gamma\times\CQ X\to\CQ_\Gamma X$ is asymptotically faithful. Let $T \in C^*(\mathcal{Q}_\Gamma X)$ be an operator with propagation at most $r$ and let $n_0$ be large enough so that for every $n>n_0$ the quotient map $\pi$ restricts to an isometry on every ball of radius $3r$ in $\Gamma\times X\times\{2^n\}\subseteq \Gamma\times \CQ X$. This allows us to define, for every $n>n_0$, a canonical $\Gamma$\=/equivariant lift $T_n' \in \C[\Gamma \times X\times\{2^n\} ]^\Gamma$ of the operator $T_n=Q_n T Q_n \in C^*(X\times\{2^n\} )$.
Specifically, given $\xi,\eta\in L^2(\Gamma \times X\times\{2^n\} )$ with support of diameter at most $r$ we define
\begin{equation*}
\langle T'_n \xi, \eta \rangle \coloneqq
    \begin{cases}
           ~\langle T_n(\xi \circ \sigma), \eta \circ \sigma \rangle,  & \mbox{if}~ d_{\Gamma\times\CQ}(\supp \xi, \supp \eta) \leq r, \\
           ~0, & \mbox{otherwise},
     \end{cases}
\end{equation*}
where $\sigma$ is the inverse of the restriction of $\pi$ to $\supp (\xi) \cup \supp (\eta)$. Note that the subspace spanned by vectors with diameter of supports at most $r$ is dense in $L^2(\Gamma \times X \times \{2^n\})$, hence $T'_n$ is well-defined. It is verified in \cite[Lemma 3.1]{sawicki_warped_2017} that each $T'_n$ is bounded, and it is clear that each $T_n'$ has propagation at most $r$ and is locally compact and invariant under conjugations.

Moreover, \cite[Lemma 3.2]{sawicki_warped_2017} shows that if $(\Gamma \times \mathcal{Q} X, d_{\Gamma\times\CQ})$ has ONL, then $\|T'_n\| \leq c \|T\|$ for every $n\in \NN$ and some uniform constant $c>0$ coming from ONL. It follows that the map $T \mapsto [(T'_n)_{n\in\NN}]$ induces an algebraic $\ast$-homomorphism:
\[
\Psi\colon \C[\mathcal{Q}_\Gamma X] \longrightarrow \frac{\prod_n C^*(\Gamma \times \{2^n\} \times X)^\Gamma}{\bigoplus_n C^*(\Gamma \times \{2^n\} \times X)^\Gamma}
\]
which can be extended to a $C^*$-homomorphism on the whole $C^*(\mathcal{Q}_\Gamma X)$.
As a matter of fact, it is possible to obtain a slightly improved control on the norm of $\Psi(T)$ (the proof is omitted as it is equal to the proof of \cite[Lemma~13.3.11]{WY20}):

\begin{lem}\label{lem:ONL lifting}
Let $\Gamma\curvearrowright X$ be a free action and assume that $(\Gamma \times \mathcal{Q} X, d_{\Gamma\times\CQ})$ has ONL. Then for every $T \in \C[\mathcal{Q}_\Gamma X]$ we have that
\begin{equation*}
\|\Psi(T)\| = \sup_{R\geq 0} \varlimsup_{n\to \infty} \sup\{\|T_n \xi\| \mid \xi\in L^2(X\times\{2^n\} ), \|\xi\|=1, \mbox{and}~\diam(\supp \xi) \leq R\}.
\end{equation*}
\end{lem}

Lemma~\ref{lem:ONL lifting} allows us to identify the kernel of $\Psi$ with the closed ideal consisting of all ghost operators in $C^*(\mathcal{Q}_\Gamma X)$ (compare with \cite[Corollary~13.3.14]{WY20}):

\begin{cor}\label{cor:ghost iff kernel}
Let $\Gamma\curvearrowright X$ be a free action and assume that $(\Gamma \times \mathcal{Q} X, d_{\Gamma\times\CQ})$ has ONL. Given $T\in C^*(\mathcal{Q}_\Gamma X)$, then $\Psi(T)=0$ \emph{if and only if} $T$ is a ghost operator.
\end{cor}

\begin{proof}
By continuity of $\Psi$, it follows from Lemma~\ref{lem:ONL lifting} that the formula
\begin{equation*}
\|\Psi(T)\| = \sup_{R\geq 0} \varlimsup_{n\to \infty} \sup\{\|T_n \xi\| \mid \xi\in L^2(X\times\{2^n\} ), \|\xi\|=1, \mbox{and}~\diam(\supp \xi) \leq R\}
\end{equation*} 
also holds for every $T\in C^*(\mathcal{Q}_\Gamma X)$. So $\Psi(T)=0$ if and only if for every $R, \epsilon>0$, there exists an $N_0 \in \NN$ such that for every $n>N_0$ and every unit vector $\xi\in L^2(X\times\{2^n\} )$ with $\diam(\supp \xi) \leq R$, we have $\|T_n \xi\| \leq \epsilon$. The latter condition holds if and only if $T$ is a ghost operator.
\end{proof}

We resume the construction of the trace $\tau^{\mathrm{u}}$ following \cite[Section 3.2]{sawicki_warped_2017}. 
It can be shown that for every $n\in\NN$ we have a $*$-isomorphism
\[
C^*((\Gamma \times X\times\{2^n\} ), d_{\Gamma\times\CQ X})^\Gamma \cong C^*_r(\Gamma) \otimes \K(L^2(X\times\{2^n\} )),
\]
where $\K(L^2(X\times\{2^n\} ))$ denotes the compact operators.
The latter admits a trace $\tau$ coming from the canonical traces on both tensor factors. More precisely, we let
\[
\tau(p)\coloneqq\mathrm{Tr}(\chi_1 p \chi_1),
\]
where $\chi_1$ is the characteristic function of $\{1\} \times X\times\{2^n\} $, and $\mathrm{Tr}$ is the canonical trace on $\K(L^2(X\times\{2^n\} ))$. 
Finally, we define the trace $\tau^{\mathrm{u}}$ on $K_0(C^*(\mathcal{Q}_\Gamma X))$ as the following composition:
\begin{align*}
K_0(C^*(\mathcal{Q}_\Gamma X)) \stackrel{\displaystyle\Psi_*}{\longrightarrow} & K_0 \bigg(\frac{\prod_n C^*(\Gamma \times X\times\{2^n\} , d_{\Gamma\times\CQ})^\Gamma}{\bigoplus_n C^*(\Gamma \times X\times\{2^n\} , d_{\Gamma\times\CQ})^\Gamma}\bigg)  \\
& \cong \frac{\prod_n K_0(C^*(\Gamma \times X\times\{2^n\} , d_{\Gamma\times\CQ})^\Gamma)}{\bigoplus_n K_0(C^*(\Gamma \times X\times\{2^n\} , d_{\Gamma\times\CQ})^\Gamma)}
\stackrel{\displaystyle\tau_*}{\longrightarrow} \frac{\prod \RR}{\bigoplus \RR}.
\end{align*}

Consequently, Corollary~\ref{cor:ghost iff kernel} together with Corollary~\ref{cor:ghost.projection.in.Roe}  prove the following:
\begin{prop}\label{prop:trace=0}
Let $(X,d)$ be a compact metric space of diameter at most $2$ equipped with a non-atomic probability measure $\nu$ of full support, and $\Gamma \act (X,d,\nu)$ be a free measure-class-preserving continuous action. Assume that $(\Gamma \times \mathcal{Q} X,d_{\Gamma\times\CQ})$ has ONL. Then for every ghost projection $p \in C^*(\mathcal{Q}_\Gamma X)$, we have $\tau^{\mathrm{u}}([p])=0$. In particular, for any projection $P \in \mathcal{P}$ we have $\tau^{\mathrm{u}}([\Phi_{\mathcal{Q}}(P)])=0$.
\end{prop}

\subsubsection{Comparing the two traces}

The concluding argument goes exactly as in \cite{higson1999counterexamples, sawicki_warped_2017, MR2871145}. 
The key idea is to use Atiyah $\Gamma$-index Theorem \cite{atiyah1976elliptic} to show that whenever $p$ is a projection in the Roe algebra such that $[p]$ belongs to the range of the coarse assembly map, then
\[
\tau_{\mathrm{d}}([p]) = \tau^{\mathrm{u}}([p]) \in \frac{\prod \RR}{\bigoplus \RR}.
\]
This argument first appeared in \cite[Proposition 5.6]{higson1999counterexamples}. The detailed proof (in the case of graphs) can be found in \cite[Lemma 6.5]{MR2871145} (see also \cite[Theorem 3.3]{sawicki_warped_2017} for the case of compact metric spaces).

Together with Lemma \ref{lem:trace d} and Proposition \ref{prop:trace=0}, we deduce that  every $K$-theory class of a non-compact ghost projection in the Roe algebra is \emph{not} in the image of the coarse assembly map (see \cite[Theorem~6.1]{MR2871145} for the case of graphs). 
In particular this applies to any projection $\Phi(P)$ with $P\in\CP$. We record this fact as a theorem, as it is the main result of Section~\ref{sec:baum_connes} and it generalises \cite[Theorem~3.5]{sawicki_warped_2017} from measure\=/preserving actions with a spectral gap to measure\=/class-preserving asymptotically expanding actions:

\begin{thm}\label{thm:ceg to CBC}
Let $(X,d)$ be a compact metric space of diameter at most $2$ equipped with a non-atomic probability measure $\nu$ of full support, and $\rho\colon\Gamma \act (X,d,\nu)$ be a free continuous measure\=/class\=/preserving action. Further assume that $(\Gamma \times \mathcal{Q} X,d_{\Gamma\times\CQ})$ has ONL.

If $p=\Phi_{\mathcal{Q}}(P_Y)$ for any domain $Y \subseteq X$ of asymptotic expansion or $p=\Phi_{\mathcal{Q}}(\widehat P_{Y,S})$ for any domain $Y\subseteq X$ of Markov $S$\=/expansion, then $[p]$ does not belong to the image of the coarse assembly map.

In particular, if the action $\rho$ is asymptotically expanding then the class of the sparse Druţu--Nowak projection $\fkG_\mathcal{Q}$ violates the coarse Baum--Connes conjecture for the sparse warped cone $\mathcal{Q}_\Gamma X$. 
\end{thm}
The following corollary follows immediately from Theorem~\ref{thm:ceg to CBC} and Remark~\ref{rmk:ONL of product cone}:
\begin{cor}\label{cor:ceg to CBC}
Let $(X,d)$ be a compact metric space of diameter at most $2$ equipped with a non-atomic probability measure $\nu$ of full support, and let $\rho\colon\Gamma \act (X,d,\nu)$ be a free Lipschitz measure\=/class\=/preserving asymptotically expanding action under \textbf{either} of the following conditions:
 \begin{itemize}
 \item[(1)] if $\Gamma$ has property $A$ and $X$ is a manifold;
 \item[(2)] if the asymptotic dimension of $\Gamma$ is finite and $X$ is an ultrametric space.
 \end{itemize}  
Then the coarse Baum--Connes conjecture for the sparse warped cone $\mathcal{Q}_\Gamma X$ fails.
\end{cor}

\begin{ex}\label{eg:counterexample BCC}
Given a chain of finite index subgroups $\Gamma>\Gamma_1>\Gamma_2>\cdots$, we consider the inverse limit $X=\varprojlim\Gamma/\Gamma_i$. This space is homeomorphic to a Cantor set, and the uniform measures on $\Gamma/\Gamma_i$ induce a natural probability measure $\nu$ on $X$
($\nu$ is obviously non-atomic and with full-support). Further, $X$ can also be given an ultrametric by letting $d((\gamma_i\Gamma_i)_{i\in\NN},(\gamma_i'\Gamma_i)_{i\in\NN})=2^{-n}$ where $n$ is the smallest index such that $\gamma_n\Gamma_n\neq\gamma'_n\Gamma_n$.
Clearly, $\Gamma$ acts $X$ by left multiplication and the action is isometric and measure\=/preserving. Further, if $\bigcap_{i\in\NN}\Gamma_i=\{1\}$ then the action is free. Such an action is called a \emph{profinite action}.

Abért--Elek constructed in \cite[Theorem~5]{abert2012dynamical} a free profinite action $F_k \curvearrowright (X,d,\nu)$ of any finitely generated non\=/abelian free group $F_k$ that is strongly ergodic (and hence asymptotically expanding) but does not have a spectral gap. 

Since the free group has asymptotic dimension 1, we can hence apply Theorem~\ref{thm:ceg to CBC} and Corollary~\ref{cor:ceg to CBC} to deduce that the sparse Druţu--Nowak projection over $\CQ_{F_k}X$ violates the coarse Baum--Connes conjecture. This fact does not directly follow from \cite[Theorem~3.5]{sawicki_warped_2017}, as the action does not have spectral gap.
% Moreover, $X$ is an ultrametric space, so $\CQ_{F_k}X$ is not coarsely equivalent to a disjoint union of finite graphs (\cite[Proposition~4.6 (1)]{sawicki_warped_2017}). This implies that this result cannot be deduced by combining the approximating space construction from \cite{dynamics1} with results in \cite{structure}.
As pointed out by the anonymous referee, it is also possible to deduce that the sparse warped cone $\mathcal{Q}_{F_k} X$ violates the coarse Baum--Connes conjecture by combining the approximating space construction from \cite{dynamics1} with the results in \cite{structure}. However, the latter argument is somewhat more opaque. For example, it is not clear to us whether this approach implies that sparse Druţu--Nowak projection violates the coarse Baum--Connes. 
% conjecture by this approach, even though this projection indeed belongs to the Roe algebra by Theorem~\ref{thm:char for asymp. expansion via Roe}. 
On the contrary, the approach developed in this paper implies \emph{all} non-compact ghost projections---including the sparse Druţu--Nowak projection---violate the conjecture.

% \blue{As pointed out by the anonymous referee, the violation of the conjecture for this example itself can also be deduced by combining the approximating space construction from \cite{dynamics1} with the results in \cite{structure}. However, it is not clear to us whether the sparse Druţu--Nowak projection violates the coarse Baum--Connes conjecture by this approach, even though this projection indeed belongs to the Roe algebra by Theorem~\ref{thm:char for asymp. expansion via Roe}. Whereas the current approach implies all non-compact ghost projections including the sparse Druţu--Nowak projection violate the conjecture.}

 %we find that the current approach is more transparent and self\=/contained.}
\end{ex}

\begin{rmk}\label{rmk: unbounded geometry profinite}
 It is not hard to check that the sparse and unified warped cones arising from a free profinite action have bounded geometry\footnote{%
 A metric space $(X,d)$ has \emph{bounded geometry} if for every $\epsilon,R > 0$ there exists an $N \in \N$ such that any $\epsilon$-separated subset of an $R$-ball of $X$ has at most $N$ elements.} \emph{if and only if} there is a uniform upper bound on the indices $[\Gamma_i:\Gamma_{i+1}]$ for $i \in \NN$.

 It follows that the sparse warped cone in Example~\ref{eg:counterexample BCC} does \emph{not} have bounded geometry in general: the construction of Abért--Elek requires chains of subgroups with indices growing very quickly (this is important in the proof of \cite[Lemma 6.2]{abert2012dynamical}). It would be interesting to know if it is possible to find a chain $\Gamma>\Gamma_1>\cdots$ with uniformly bounded indices $[\Gamma_i:\Gamma_{i+1}]$ such that the induced profinite action is strongly ergodic but has no spectral gap.
\end{rmk}

\subsection{Non-coarse embeddability}
In this subsection we prove that warped cones arising from asymptotically expanding actions do not coarsely embed into any Hilbert space.  
One of the ground-breaking results by Yu was to verify the coarse Baum--Connes conjecture for every proper bounded geometry metric space which coarsely embeds into some Hilbert space \cite{Yu00}. It follows from Theorem~\ref{thm:ceg to CBC} that the sparse warped cone $\mathcal{Q}_\Gamma X$ coming from an asymptotically expanding action of $\Gamma$ on a compact metric space $X$ cannot coarsely embed into Hilbert spaces provided that $(\Gamma \times \mathcal{Q} X,d_{\Gamma\times\CQ})$ has ONL and $\mathcal{Q}_\Gamma X$ has bounded geometry (it is not hard to show that if $(X,d)$ is proper, so is $(\CO_\Gamma X,d_\Gamma)$). Below we will strengthen this result and show both ONL of $(\Gamma \times \mathcal{Q} X,d_{\Gamma\times\CQ})$ and bounded geometry of $\mathcal{Q}_\Gamma X$ are redundant.

Recall that a map $F:(X,d_X)\rightarrow (Z,d_Z)$ is a \emph{coarse embedding} between metric spaces if there exist non-decreasing unbounded functions $\rho_{\pm}\colon[0,\infty)\rightarrow [0,\infty)$ such that 
\begin{align*}
\rho_{-}(d_X(x,x'))\leq d_Z(F(x),F(x'))\leq \rho_{+}(d_X(x,x')),
\end{align*}
for all $x,x'\in X$.

The following proposition is a (partial) extension of \cite[Theorem~3.1]{nowak_warped_2017} from the setting of measure-preserving actions with a spectral gap to asymptotically expanding measure\=/class\=/preserving actions. The proof combines the idea in \cite[Theorem~3.1]{nowak_warped_2017} with Proposition~\ref{prop:Markov local version}:  

\begin{prop}\label{prop:non CE}
Let $(X,d)$ be a compact metric space of diameter at most $2$ equipped with a non-atomic probability measure $\nu$, and $\rho\colon \Gamma\curvearrowright (X,d,\nu)$ be a continuous measure-class-preserving and asymptotically expanding action. 
 
If $A\subseteq [1,\infty)$ is any unbounded subset and $d_\Gamma$ is the warped cone metric on $\CO_\Gamma X$, then $(X\times A,d_\Gamma)$ does not admit a coarse embedding into any Hilbert space. 
\end{prop}
\begin{proof}
From Proposition~\ref{prop:Markov local version}, there exist a finite symmetric subset $1\in S\subseteq \Gamma$ and a domain $Y \subseteq X$ of Markov $S$-expansion such that there is a constant $\Theta\geq 1$ such that $1/\Theta\leq r(s,x)\leq\Theta$ for every $x\in Y$ and $s\in S_{Y,x}=\{s\in S\ \vert\  s\cdot x \in Y\}$. Hence, we have $1-\lambda_2>0$ by Theorem~\ref{thm:spectral.characterisation.markov.exp} (see also (\ref{eq: spectral gap})) and we let $\kappa\coloneqq \frac{1}{2(1-\lambda_2)}>0$. 

%For every $g\in L^2_0(Y,\nu)$, it follows from Proposition~\ref{prop:normalised.markov.is.reversible}(1)-(2) that $g\in L^2_0(Y,\snu)$. 

By Proposition~\ref{prop:normalised.markov.is.reversible}(3) (see also (\ref{eq: cor 4.17})) we have that for every $g\in L^2_0(Y,\snu)$
\begin{align*}
\norm{g}_{\snu,2}^2 &\leq \kappa\sum_{s\in S}\int_{Y\cap s^{-1}(Y)} r(s,x)^{1/2}\abs{g(x)-g(s\cdot x)}^2 \d\nu(x) \\
&\leq \kappa\sqrt{\Theta} \sum_{s\in S}\int_{Y\cap s^{-1}(Y)} \abs{g(x)-g(s\cdot x)}^2 \d\nu(x).
\end{align*}

Assume now that $(X\times A,d_\Gamma)$ admits a coarse embedding into Hilbert space $\ell^2(\NN)$. If $F\colon (Y\times A,d_\Gamma) \rightarrow \ell^2(\NN)$ denotes the coarse embedding, then we let $F_t\colon Y \rightarrow \ell^2(\NN)$ be the restriction of $F$ to the level set $Y \times \{t\}$ for $t\in A$. For each $t\in A$ and  $n\in \NN$, denote by $F_t^{(n)}$ the associated coefficient of the function of $F_t$. Every $F_t^{(n)}$ is a bounded function (which is hence in $L^2(Y,\nu)$), and we have
\begin{align*}
\sum_{n\in \NN}&\sum_{s\in S}\int_{Y\cap s^{-1}(Y)} \abs{F_t^{(n)}(x)-F_t^{(n)}(s\cdot x)}^2 \d\nu(x) = \sum_{s\in S}\int_{Y\cap s^{-1}(Y)} \|{F_t(x)-F_t(s\cdot x)}\|^2_{\ell^2(\NN)} \d\nu(x) \\
&\leq \sum_{s\in S}\int_{Y\cap s^{-1}(Y)} \rho_+(d^t_\Gamma(x,s\cdot x))^2 \d\nu(x) \leq \rho_+(M)^2|S|,
\end{align*}
where $M\coloneqq\max_{s\in S}\ell(s)$.

After translating $F_t$ if necessary, we may assume that $F_t^{(n)} \in L^2_0(Y,\snu)$.
This implies that for each $t\in A$ we have
\[
\norm{F_t}_{\snu,2}^2=\sum_{n\in \NN} \norm{F^{(n)}_t}_{\snu,2}^2 \leq \kappa\sqrt{\Theta}\rho_+(M)^2|S|<\infty.
\]
On the other hand, since $F_t^{(n)} \in L^2_0(Y,\snu)$ we also have
\begin{align*}
 \iint_{Y\times Y} \|F_t(x)-F_t(y)\|^2_{\ell^2(\NN)} \d \snu(x)\snu(y) = 2 \|F_t\|_{\snu,2}^2.
\end{align*}
We will thus reach a contradiction by showing that
\begin{equation}\label{eq: integral tending to infty1}
\iint_{Y\times Y} \|F_t(x)-F_t(y)\|^2_{\ell^2(\NN)} \d \snu(x)\snu(y)\to \infty,\  \text{as}\  t\to \infty.
\end{equation}

Since $\snu$ is non-atomic and $\Gamma$ is countable, the set
\[
N\coloneqq\{(x,y)\in Y\times Y\ | \ \Gamma \cdot x=\Gamma \cdot y\}
\]
is measurable and has measure zero by Fubini's theorem. On the other hand, for any $x,y\in Y$ lying in different $\Gamma$-orbits the distance $d_\Gamma((x,t),(y,t))=d^t_\Gamma(x,y)\to \infty$ as $t\to \infty$. Since $F$ is a coarse embedding, it follows that $ \|F_t(x)-F_t(y)\|^2_{\ell^2(\NN)}\to \infty$. Hence, we deduce (\ref{eq: integral tending to infty1}) and obtain the desired contradiction.
\end{proof}

\begin{rmk}
Let $1<p<\infty$. By interpolation, if $Y$ is a domain of Markov $S$\=/expansion then the lazy Markov operator $\frac 12+\frac 12\fkP_{Y,S}$ has norm strictly less than one also when regarded as an operator on $L^p_0(Y,\snu)$. An easy modification of the proof of Proposition \ref{prop:non CE} shows that the warped cone does not coarsely embed into $L^p$ for any $1<p<\infty$.
Moreover, the warped cone cannot coarsely embed into $L^1$\=/spaces as well because it is shown in \cite[Proposition~4.1]{MR2146202} that every $L^1$\=/space coarsely embeds into a Hilbert space.
\end{rmk}

\begin{ex}
 \cite[Theorem~3.1]{nowak_warped_2017} does not apply to the warped cones arising from the profinite actions $F_k \curvearrowright (X,d,\nu)$ of Abért--Elek (Example~\ref{eg:counterexample BCC}). However, we may use Proposition~\ref{prop:non CE} to conclude that the sparse warped cone $\mathcal{Q}_{F_k} X$ as well as the unified warped $\CO_{F_k} X$ cannot be coarsely embedded into any Hilbert space.
 
 Note also that the non\=/embeddability of $\mathcal{Q}_{F_k} X$ does not immediately follow from the fact that it violates the coarse Baum--Connes conjecture. In fact, Yu's argument only applies to \emph{bounded geometry} proper metric spaces, while the warped cones in Example~\ref{eg:counterexample BCC} have unbounded geometry (Remark \ref{rmk: unbounded geometry profinite}).
\end{ex}

\appendix

\section{Proof of the spectral characterisation of expansion for Markov kernels}

In this appendix, we will provide a proof of Theorem~\ref{thm:spectral.characterisation.markov.exp}.
As mentioned before, some special cases of this result are especially well\=/known. Two such instances are given by simple random walks on finite or countably infinite graphs: the former gives a spectral characterisation of expansion \cite{alon1986eigenvalues,alon1985lambda1,dodziuk1984difference}, while the latter characterises non\=/amenability \cite{dodziuk1984difference,kesten1959full,mohar1988isoperimetric}. 
Lawler--Sokal \cite{lawler1988bounds} proved the general result already in the 1980s. However, their work seems to have been overlooked by a part of the mathematical community. In \cite{Kai92} Kaimanovich proved a version of Theorem~\ref{thm:spectral.characterisation.markov.exp} for reversible Markov kernels on infinite measure spaces (he actually proved much more refined results concerning $p$\=/capacities and Dirichlet norms). 
Lyon--Nazarov proved it for Markov kernels arising from measure\=/preserving actions on probability spaces \cite[Theorem 3.1]{lyons2011perfect}\footnote{%
\cite[Theorem 3.1]{lyons2011perfect} also claims that the spectrum of the Markov operator is bounded away from $-1$, but this is not correct: there is a small mistake at the very end of their proof. 
It is also worth pointing out that their proof is based on an inequality which they claim holds true by ``checking cases''. We are unable to verify such inequality.}.

Before finding out about \cite{lawler1988bounds}, we managed to prove Theorem~\ref{thm:spectral.characterisation.markov.exp} by extending the standard argument used for Markov processes on finite state spaces \cite{Pet17,Woe09}.
Most of its key points generalise without difficulties (Lemma~\ref{lem:sobolev.inequality} and Lemma~\ref{lem:dirichlet.estimate}), but the concluding argument is considerably more involved.
We decided to include the proof in this appendix for the convenience of the reader and also because it provides a slightly better lower bound on the spectral gap.

% 
% Yet, we could not find an actual proof of Theorem~\ref{thm:spectral.characterisation.markov.exp} in the literature and thus prefer to provide one here. Our proof is modelled on the one for Markov processes on finite state spaces \cite{Pet17,Woe09}: most of its key points generalise without difficulties (Lemma~\ref{lem:sobolev.inequality} and Lemma~\ref{lem:dirichlet.estimate}), but the concluding argument is considerably more involved---especially for the claimed bounds.

\begin{lem}\label{lem:sobolev.inequality}
 If $g\in L^1(X,m)$ is a function that takes value in $[0,\infty)$ and such that $m(\braces{g>0})\leq \frac 12 m(X)$, then $\CE_1(g)\geq \kappa \norm{g}_{m,1}$.
\end{lem}

\begin{proof}
 By hypothesis, we have that
 \begin{equation*}%\label{eq:integral.of.g}
  \int_X g(x)\d m 
  = \int_0^\infty  m(\braces{g\geq t}) \d t 
  \leq \frac{1}{\kappa} \int_0^\infty \mkbdry{\braces{g\geq t}} \d t,
\end{equation*}
where $\kappa$ denotes the Cheeger constant. Using \eqref{eq:dirichlet.equals.boundary} we deduce:
 \begin{align*}
  \int_X g(x)d m 
  &\leq \frac{1}{\kappa} \int_0^\infty \CE_1(\chi_{\braces{g\geq t}}) \d t. \\
  &= \frac{1}{2\kappa}\int_0^\infty\int_{X\times X}\abs{\chi_{\braces{g\geq t}}(x)-\chi_{\braces{g\geq t}}(y)}\, \d\mu(x,y) \d t \\
  &= \frac{1}{2\kappa}\int_{X\times X} \left({\int_0^\infty\abs{\chi_{\braces{g\geq t}}(x)-\chi_{\braces{g\geq t}}(y)}\,\d t\; }\right) \d\mu(x,y) \\
  &= \frac{1}{2\kappa}\int_{X\times X} \abs{g(x)-g(y)}\d\mu(x,y),
 \end{align*}
thus proving the lemma.
\end{proof}

In turn, this is used to prove the estimate that lies at the heart of the proof of Theorem~\ref{thm:spectral.characterisation.markov.exp}:

\begin{lem}\label{lem:dirichlet.estimate}
 If $g\in L^2(X,m)$ is a function that takes value in $[0,\infty)$ and such that $m(\braces{g>0})\leq \frac 12 m(X)$, then 
 \[
 \CE_2(g)\geq \frac{\kappa^2}{2}\norm{g}_{ m,2}^2.
 \]
\end{lem} 
\begin{proof}
Firstly, we note that $\norm{g}_{ m,2}^2=\norm{g^2}_{ m,1}$. We can hence apply Lemma~\ref{lem:sobolev.inequality} to the function $g^2$ to obtain
 \begin{equation}\label{eq:norm_dirichlet.energy.bound}
  \kappa\norm{g}_{ m,2}^2\leq \CE_1(g^2).
 \end{equation}
Using the Cauchy--Schwarz inequality, we can estimate the value $\CE_1(g^2)$ as follows:
 \begin{align*}
  \CE_1(g^2)
  &=\frac 12\int_{X\times X} \abs{g^2(x) - g^2(y)}\d\mu(x,y) \\
  &=\frac 12\int_{X\times X} \abs{g(x) - g(y)} \cdot \abs{g(x) + g(y)}\d\mu(x,y) \\
  &\leq\frac 12\Bigparen{\int_{X\times X} \abs{g(x) - g(y)}^2 \d\mu(x,y) }^\frac 12 
  \Bigparen{\int_{X\times X} \abs{g(x) + g(y)}^2\d\mu(x,y) }^\frac 12\\
  &=\frac 12 \bigparen{2\CE_2(g)}^\frac 12 \Bigparen{2 \norm{g}^2_{m,2}+2\angles{g,\fkP g}_m}^\frac 12
  \\
  &\leq \sqrt{2}\CE_2(g)^\frac 12 \norm{g}_{m,2}.
 \end{align*}
The proof is complete once we combine the above estimate with \eqref{eq:norm_dirichlet.energy.bound}. 
\end{proof}

Finally, we are ready to prove the main theorem of this subsection: 

\begin{proof}[Proof of Theorem~\ref{thm:spectral.characterisation.markov.exp}]  
Given a measurable $A\subseteq X$ with $0< m(A)\leq \frac 12 m(X)$, let $f_A\coloneqq \chi_A-\frac{ m(A)}{ m(X)}$ be the projection of $\chi_A$ to $L^2_0(X, m)$. Then 
 \[
  \norm{f_A}_{ m,2}^2=( m(X)- m(A))\frac{ m(A)}{ m(X)}\geq \frac{1}{2} m(A). 
 \]
 Using \eqref{eq:dirichlet.equals.boundary} we deduce that
 \[
  1-\lambda_2\leq \frac{\CE_2(f_A)}{\norm{f_A}_{ m,2}^2}
  = \frac{\mkbdry{A}}{\norm{f_A}_{ m,2}^2}\leq 2 \frac{\mkbdry{A}}{m(A)},
 \]
 and hence $1-\lambda_2 \leq 2\kappa$.

For the other direction, we need to show
 \[
  \frac{\kappa^2}{2}
  \leq \inf_{f\in L^2_0(X,m)} \frac{\angles{f,\Delta f}_m}{\norm{f}_{m,2}^2}
  =\inf_{f\in L^2_0(X,m)} \frac{\CE_2(f)}{\norm{f}_{m,2}^2}=1-\lambda_2.
 \]
 
 Since $\fkP$ is self\=/adjoint, the spectral theorem implies that there exists a sequence of real\=/valued functions $f_n\in L^2_0(X,m)$ with $\|f_n\|_{m,2}=1$ such that $\norm{\fkP f_n-\lambda_2 f_n}_{m,2}\to 0$. In particular, $\langle f_n, \Delta f_n \rangle \to 1-\lambda_2$. Write $f_n=f_n^+-f_n^-$, where $f_n^+(x)\coloneqq \max\braces{0, f_n(x)}$ and $f_n^-(x)\coloneqq \max\braces{0, -f_n(x)}$. Replacing $f_n$ with $-f_n$ if necessary, we can assume that $ m(\braces{f_n(x)>0})\leq \frac 12 m(X)$. 

If each $f_n$ was an eigenfunction for $\Delta$, we would immediately have 
 \begin{equation}\label{eq:dirichlet positive part}
  \frac{\angles{f_n^+\,,\, (\Delta f_n)^+}_m}{\norm{f_n^+}_{m,2}^2}
  = \frac{\angles{f_n\,,\,\Delta f_n}_{m}}{\norm{f_n}_{m,2}}.
 \end{equation}
In this case, the proof of the theorem would easily follow from Lemma~\ref{lem:dirichlet.estimate}. Yet, this need not be the case for general Markov kernels. This is the place where our argument differs from the classical proof for finite\=/state processes. 
 
On the way to overcome this difficulty, we will first need to modify $f_n$ to ensure that $\norm{f_n^+}_{m,2}$ is bounded away from $0$.
If $\norm{f_n^+}_{m,2}$ does not tend to $0$, we simply pass to a subsequence $h_n\coloneqq f_{k_n}$ so that $\norm{h_n^+}_{m,2}$ is bounded away from $0$. Otherwise, we have $\norm{f_n^+}_{m,2}\to 0$. Since $m$ is finite, we also have $\norm{f_n^+}_{m,1}\to 0$. On the other hand, $\norm{f_n^+}_{m,1}=\norm{f_n^-}_{m,1}$ because $f_n\in L^2_0(X,m)$. It follows that there exists a sequence $c_n>0$ such that $c_n\to 0$ and $m(\braces{f_n^-(x)\geq c_n})\to 0$. 
We then define $h_n\coloneqq -(f_n+c_n)$ and also note that 
\[
\norm{\fkP h_n - \lambda_2 h_n}_{m,2} \leq \norm{\fkP f_n - \lambda_2 f_n}_{m,2} + \|\fkP - \lambda_2\| \cdot \|c_n\| \to 0 \mbox{~as~} n\to \infty.
\]
For $n$ large enough we have $m(\braces{h_n^+(x)> 0})\leq m(X)/2$ and $\norm{h_n^+}_{m,2}\geq \norm{f_n^-}_{m,2}-\norm{c_n}_{m,2}$ tends to $1$, as $1=\norm{f_n^+}_{m,2}^2+\norm{f_n^-}_{m,2}^2$ and $\norm{f_n^+}_{m,2}\to 0$. 
 
 We are now ready to complete the proof. Note that
 \[
  \angles{h_n^+, (\fkP h_n)^+}_m -\lambda_2\norm{h_n^+}_{m,2}^2
  =\angles{h_n^+\,,\; (\fkP h_n)^+ -\lambda_2h_n^+} _m
 \]
 and by the Cauchy--Schwarz inequality, we have that
 \[
  \frac{\abs{\angles{h_n^+\,,\; (\fkP h_n)^+ -\lambda_2h_n^+}}}{\norm{h_n^+}_{m,2}^2}
  \leq \frac{\norm{(\fkP h_n)^+ - (\lambda_2 h_n)^+}_{m,2}}{\norm{h_n^+}_{m,2}}
  \leq \frac{\norm{\fkP h_n - \lambda_2 h_n}_{m,2}}{\norm{h_n^+}_{m,2}}.
 \]
 Since $\norm{h_n^+}_{m,2}$ is bounded away from $0$, the right hand side in the above inequality tends to $0$ and therefore
 \[
  1-\lambda_2
  =
  \lim_{n\to\infty} \frac{\norm{h_n^+}_{m,2}^2-\angles{h_n^+\,,\; (\fkP h_n)^+}_m}{\norm{h_n^+}_{m,2}^2}.
 \]
 
 Finally, since
 \(  
  \angles{h_n^+,\fkP(h_n^+)}_m 
  \geq \angles{h_n^+,(\fkP h_n)^+}_m
 \), we deduce that 
 \[
  1-\lambda_2 
  \geq \varlimsup_{n\to\infty} \frac{\norm{h_n^+}_{m,2}^2-\angles{h_n^+,\fkP(h_n^+)}_m}{\norm{h_n^+}_{m,2}^2}
  = \varlimsup_{n\to \infty}\frac{\angles{h_n^+\,,\; \Delta( h_n^+)}_m}{\norm{h_n^+}_{m,2}^2}
  = \varlimsup_{n\to \infty}\frac{\CE_2\paren{h_n^+}}{\norm{h_n^+}_{m,2}^2}
 \]
 and the latter is greater or equal to $\frac{\kappa^2}{2}$ by Lemma~\ref{lem:dirichlet.estimate}, as desired.
\end{proof}

\bibliographystyle{plain}
\bibliography{bibfile,ExpanderishVig}

\end{document}